\documentclass[10pt,twoside]{amsart}
\usepackage{amsmath,amssymb}
\usepackage{amsfonts}
\usepackage{amsthm,amscd}
\usepackage{amstext}
\usepackage{color}
\usepackage{latexsym}
\usepackage{amscd,graphics}
\begin{document}
\newcommand{\comment}[1]{}

\newcommand{\comments}[1]{} 

\newtheorem{proposition}{Proposition}[section]
\newtheorem{lemma}[proposition]{Lemma}
\newtheorem{sublemma}[proposition]{Sublemma}
\newtheorem{theorem}[proposition]{Theorem}

\newtheorem{maintheorem}{Main Theorem}
\newtheorem{corollary}[proposition]{Corollary}

\newtheorem{ex}[proposition]{Example}

\theoremstyle{remark}

\newtheorem{remark}[proposition]{Remark}

\theoremstyle{definition}
\newtheorem{definition}[proposition]{Definition}
\def\real{\mathbb{R}}
\def\integer{\mathbb{Z}}
\def\complex{\mathbb{C}}
\def\supp{\mathrm{supp}}
\def\var{\mathrm{var}}
\def\sgn{\mathrm{sgn}}
\def\sp{\mathrm{sp}}
\def\id{\mathrm{id}}
\def\Imm{\mathrm{Image}}
\def\cc{\Subset}
\def\Lip{\mathrm{Lip}}
\def\BB{\mathcal{B}}
\def\CC{\mathcal{C}}
\def\DD{\mathcal{D}}
\def\EE{\mathcal{E}}
\def\FF{\mathcal{F}}
\def\GG{\mathcal{G}}
\def\II{\mathcal{I}}
\def\JJ{\mathcal{J}}
\def\KK{\mathcal{K}}
\def\LL{\mathcal{L}}
\def\LLL{\mathbb{L}}
\def\MM{\mathcal{M}}
\def\NN{\mathcal{N}}
\def \OO{\mathcal {O}}
\def \PP{\mathcal {P}}
\def \QQ{\mathcal {Q}}
\def \RR{\mathcal {R}}
\def\SS{\mathcal{S}}
\def\TT{\mathcal{T}}
\def\YY{\mathcal{Y}}
\def\ZZ{\mathcal{Z}}
\def\FFF{\mathbb{F}}
\def\PPP{\mathbb{P}}

\title[Linear response  for  deformations
of generic unimodal maps]{Linear response for smooth
deformations of  generic nonuniformly
hyperbolic unimodal maps}
\author{Viviane Baladi and Daniel Smania} 
\address{D.M.A., UMR 8553,\'Ecole Normale Sup\'erieure,  75005 Paris, France}
\email{viviane.baladi@ens.fr}

\address{ Departamento de Matem\'atica,
ICMC-USP, Caixa Postal 668,  S\~ao Carlos-SP,
CEP 13560-970 S\~ao Carlos-SP, Brazil}
\email{smania@icmc.usp.br}
\date{October 2, 2011, more typos corrected, Banach spaces and Lasota-Yorke cleaned up, details for overlap control and
commutation added} 
\begin{abstract} We consider $C^2$ families $t\mapsto f_t$
of $C^{4}$ unimodal maps $f_t$ whose critical point
is slowly recurrent, and we show that the unique absolutely
continuous invariant measure $\mu_t$ of $f_t$
depends differentiably on $t$, as a distribution of
order $1$. The proof uses transfer operators on towers whose
level boundaries are mollified via smooth cutoff functions, in
order to avoid artificial discontinuities.
We give a new representation of $\mu_t$ for
a Benedicks-Carleson map $f_t$, in terms of a single
smooth function and the inverse branches of $f_t$ along
the postcritical orbit.
Along the way, we prove that the twisted cohomological
equation $v=\alpha \circ f - f' \alpha$ has a continuous
solution $\alpha$, if $f$ is Benedicks-Carleson and $v$
is horizontal for $f$.
\end{abstract}
\thanks{We thank D. Ruelle for conversations and
e-mails, and D. Schnellmann for pointing out many typos in
a previous version of this text.
This work was started while D. Smania was visiting
the DMA of Ecole Normale Sup\'erieure in 2008-09 and
finished while V. Baladi participated in a semester at Institut Mittag-Leffler 
(Djursholm, Sweden) in 2010.
We
gratefully acknowledge the hospitality of both 
institutions. 
V.B. is  partially supported by GDRE GREFI-MEFI
and ANRBLAN08-2\_313375, DynNonHyp.
D.S. is partially supported by 
FAPESP 2008/02841-4, CAPES 1364-08-1, CNPq 310964/2006-7 and 303669/2009-8.}

%%%%%%%%%%%%%%%%%%%%%

\maketitle

\section{Introduction}
The linear response problem for discrete-time
dynamical systems
can be posed in the following way. Suppose that for each parameter $t$  (or many parameters $t$)
in a smooth family of maps $t\mapsto f_t$ with
$f_t \colon M \to M$,
($M$ a compact Riemann manifold, say) there 
exists  a unique physical (or SRB)
measure $\mu_t$. (See \cite{youngsrb} for a discussion
of SRB measures.) One can ask for conditions which ensure the differentiability, possibly in the sense of Whitney, of the function
 $\mu_t$
in a weak sense (in
the weak $*$-topology, i.e., as a distribution of
order $0$, or possibly as a distribution
of higher order). 
Ruelle has discussed this problem in several survey papers
\cite{ruelle0}, \cite{ruelle00}, \cite{RuRe}, to which we refer for motivation.

The case of smooth hyperbolic dynamics has been settled
over a decade ago (\cite{KKPW}, \cite{rsbr}), although recent
technical progress in the functional analytic tools (namely,
the introduction of anisotropic Sobolev spaces on which
the transfer operator has a spectral gap) has allowed for
a great simplification of the proofs (see, e.g., \cite{BuLi}):
For smooth Anosov diffeomorphisms $f_s$
and a $C^1$ observable $A$,
letting 
$$X_s=\partial_t f_t|_{t=s}\circ f_s^{-1}\, , 
$$ 
Ruelle \cite{rsbr}, \cite{rsbr'} obtained the following explicit
{\it linear response formula}  (the derivative here is in the usual sense)
$$
\partial_t \int A \, d \mu_t|_{t=0}=\Psi_A(1)\, ,
$$
where $\Psi_A(z)$ is the {\it susceptibility function}
$$\Psi_A(z)=
\sum_{k=0}^\infty \int z^k \langle X_0, \mbox{grad}\,  (A \circ f_0^k)
\rangle  \,
d\mu_0 \, ,
$$
and the series $\Psi_A(z)$  at $z=1$ converges exponentially.
In fact, in the Anosov case, the susceptibility function is holomorphic in
a disc of radius larger than $1$. This is related to the fact
(see \cite{baladi} for a survey and references) that
the transfer operator of each $f_s$ has a spectral gap on a space
which contains not only the product
of the distribution $\mu_s$ and the smooth vector field $X_s$, but also the derivative of that product, that is,
$\langle X_s, \mbox{grad}\,  \mu_s\rangle + (\mbox{div}\,  X_s )\, \mu_s$.

One feature of smooth hyperbolic dynamics is structural stability:
Each $f_t$, for small $t$,
is topologically conjugated to $f_0$ via
a homeomorphism $h_t$, which turns out to depend smoothly
on the parameter $t$.
With the exception of a deep result of Dolgopyat \cite{dolgo}
on rapidly mixing partially hyperbolic systems (where structural
stability may be violated, but where there are no critical points
and shadowing holds for
a set of points of large measure, so
that the bifurcation structure is
relatively mild), the study of linear response in the absence of
structural stability, or in the presence of critical points, has
begun only recently.

However, the easier property
of {\it continuity} of $\mu_t$ with respect to $t$ (in other words,
{\it statistical stability}) has been established also in the
presence of critical points: For piecewise expanding unimodal interval maps,
Keller \cite{Ke} proved in 1982 that the density $\phi_t$ of $\mu_t$, viewed as an element
of $L^1$, has a modulus of continuity  at least $t\log t$, so that
$t\mapsto \phi_t$
is  $r$-H\"older, for any  exponent $r\in (0,1)$. For nonuniformly
smooth unimodal maps, in general not all nearby
maps $f_t$  admit an
SRB measure  even if $f_0$ does. Therefore,
continuity of $t\mapsto \mu_t$  can only
be proved in the sense of Whitney, on a
set of ``good" parameters. This was done by Tsujii
\cite{tsujiicont} and Rychlik--Sorets \cite{RySo}
in the 90's. More recently, Alves et al. \cite{ACF}, \cite{ACF2}
proved that for H\'enon maps,
 $t\mapsto \mu_t$ is continuous 
in the sense of Whitney in the weak $*$-topology.
(We refer, e.g., to \cite{BalT} for more references.)

Differentiability of $\mu_t$, even in the sense of Whitney,
is a more delicate issue, even in dimension one.
For nonuniformly hyperbolic smooth unimodal maps $f_t$ with
a quadratic critical point ($f''_t(c)<0$),
it is known
\cite{young92}, \cite{keno} that
the density $\phi_t$ of the absolutely continuous invariant measure $\mu_t$ of $f_t$ has singularities of the form $\sqrt{x-c_{k,t}}^{-1}$, where the
$c_{k,t}=f^k_t(c)$ are the points along
the forward  orbit of  the critical point $c$. 
(Following  Ruelle\cite{ruelle},  we call these  singularities {\it spikes}.)
Thus, the derivative $\phi'_t$ of the invariant 
density has nonintegrable singularities, and  the transfer operator cannot have a spectral
gap in general on a space
containing $(X_t \phi_t)'$. In fact,  the radius of convergence of the
susceptibility function $\Psi_A(z)$ is very likely
strictly smaller than $1$ in general.
Ruelle \cite{redherring}
observed however that, in the case of a subhyperbolic
(preperiodic) critical
point for a real analytic
unimodal map, $\Psi_A(z)$ is meromorphic in a disc of radius larger than $1$,
and that $1$ is not a pole of $\Psi_A(z)$.
He expressed the hope that the value $\Psi_A(1)$ obtained by analytic
continuation could correspond to the actual derivative
of the SRB measure, at least in the sense of Whitney.

This analytic continuation phenomenon in the subhyperbolic smooth unimodal
case (where a finite Markov partition exists)
could well be a red herring, in view
of the
linear response theory for the ``toy model"
of piecewise expanding interval maps  that we recently established in a series of papers
\cite{baladi}, \cite{bs1}, \cite{bs2}, \cite{bs4}:
 Unimodal piecewise expanding interval  maps $f_t$ 
 have a unique SRB measure, 
 whose density $\phi_t$ is a function of bounded variation (since
 $\phi'_t$ is a measure, the situation is much easier than 
 for smooth unimodal
 maps).
 In \cite{baladi}, \cite{bs1}, we showed that 
 Keller's \cite{Ke} $t\log t$ modulus of continuity
 was optimal (see also \cite{mazzo}): In fact,
 there exist smooth families
 $f_t$ so that $t\mapsto \mu_t$ is not Lipschitz
 (all sequences $t_n \to0$ so that the critical
 point is not periodic under $f_{t_n}$ are allowed), even when
 viewed as a distribution of arbitrarily high order, and even in
 the sense of Whitney. Such counter-examples
 $f_t$ are transversal to the topological class of $f_0$.
 If, on the contrary, the family $f_t$ is tangent at $t=0$
 to the topological
 class of $f_0$ (we say that $f_t$ is {\it horizontal}) then
 (\cite{bs1}, \cite{bs2}) we proved that the map $t\mapsto \mu_t$ is  differentiable for the weak $*$-topology.
 The series for $\Psi_A(1)$ may diverge (for the preperiodic
case, see \cite[\S 5]{baladi}), but can be resummed
 under the horizontality condition  \cite{baladi}, \cite{bs1}. This
 gives an explicit linear response formula.
 In fact, the susceptibility function $\Psi_A(z)$ is holomorphic
 in the open unit disc, and, under a condition
 slightly stronger than horizontality,
  $\partial_t \int A \, d\mu_t|_{t=0}$
 is the Abel limit of 
$ \Psi_A(z)$ as $z \to 1$.
 
  Worrying about lack of differentiability of the SRB measure is not just a mathematician's pedantry: Indeed, this phenomenon 
can be observed numerically, for example in the guise of fractal
transport coefficients. We refer, e.g., to the work of Keller et al. \cite{khk}
(see also references therein), who obtained a $t\ln (t)$ modulus of continuity
compatible with the results of \cite{Ke}, for drift and diffusion coefficients of models 
related to those analysed in \cite{bs1}.

  \medskip
   
Let us move on now to the topic of the present work, linear response
for smooth unimodal interval maps:
 Ruelle recently
obtained  a linear response  formula for real analytic families of  analytic unimodal maps of Misiurewicz type \cite{ruelle}, that is, assuming $\inf_k |f^k(c)-c|>0$, a nongeneric condition
which implies the existence of a hyperbolic Cantor set. 
(Again, this linear response formula can be viewed as a resummation of the generally divergent series $\Psi_A(1)$.)
In \cite{bs3}, we showed that $t\mapsto \mu_t$ is real analytic in the weak sense for complex analytic families of Collet-Eckmann quadratic-like maps (the -- very rigid -- holomorphicity assumption allowed
us to use tools from complex analysis).
Both  these  recent
results  are for families $f_t$ in the conjugacy class of a single
(analytic) unimodal map, and the assumptions were somewhat ungeneric.

The main result of the present work, Theorem~\ref{linresp},
is a linear response formula for $C^2$
families $t\mapsto f_t$ of $C^{4}$
unimodal maps
\footnote{The $C^4$ regularity is only used to get
$H^2_1$ regularity in Proposition~\ref{acip} and Lemma~\ref{truncspec},
and one can perhaps weaken this to
$C^{3+\eta}$.}  with quadratic critical points satisfying the so-called {\it topological slow recurrence} (TSR) condition (\cite{sandst},\cite{tsujii},\cite{luzzatto}, see \eqref{tsr} below).  (We assume that the maps have negative Schwarzian and
are symmetric, to limit technicalities, and we only consider infinite postcritical
orbits, since the preperiodic case is much easier.)
The topological slow recurrence condition is much weaker
than Misiurewicz, so that we give a new proof
of Ruelle's result \cite{ruelle} in the symmetric infinite
postcritical case (this may shed light on the
informal study in \S17 there).
Topological slow recurrence implies the well-known 
Benedicks-Carleson and Collet-Eckmann  conditions. Furthermore,  
the work of Tsujii \cite{tsujii} and Avila-Moreira
\cite{am} gives that  real-analytic
unimodal maps with a quadratic critical point satisfying the TSR condition are 
measure-theoretical {\it generic} among non regular parameter in non trivial real-analytic families
unimodal maps. (See Remark \ref{generic2}.)
If all maps in a family of unimodal maps $f_t$ satisfy the topological  slow recurrence condition then 
\cite{Str} this family is a {\it deformation,} that is, the family
$\{f_t\}$ lies entirely in the topological
class of $f_0$ (there exist
homeomorphisms $h_t$ such that $h_t(c)=c$ and
$h_t\circ f_0 = f_t\circ h_t$). In particular, horizontality holds.

\smallskip
We next briefly discuss a few new ingredients of our arguments, as well
as a couple of additional results we obtained along the way.
A first remark is that we need  uniformity of the hyperbolicity constants
of $f_t$ for all small $t$. We deduce this uniformity from previous work of Nowicki, making
use of the TSR assumption (Section~\ref{recurr}). 

When one moves the parameter $t$, the orbit of the critical point also moves, and so do  the spikes. Therefore, in order to understand
$\partial_t \mu_t$, we need upper bounds on
 $$
\partial_t c_{k,t}|_{t=0}=\partial_t f_t^k(c)|_{t=0}=\partial_t h_t(f_0^k(c))|_{t=0}=\partial_t h_t(c_{k,0})|_{t=0}\, ,
$$
uniformly  in $k$. It is not very difficult to show 
(Lemma~\ref{first}, see also Proposition~\ref{htC1})
that $\partial_t c_{k,t}|_{t=0}=\alpha(c_{k,0})$
if $\alpha$ solves the twisted cohomological equation (TCE) for $v=\partial_t f_t |_{t=0}$, given by,
$$
v=\alpha \circ f_0 + f'_0 \cdot \alpha \, ,
\quad \alpha(c)=0 \, .
$$
(Such a function $\alpha$ is called an {\it infinitesimal conjugacy.)}
In fact, we prove in Theorem~\ref{alphabded} that if $f_0$ is Benedicks-Carleson
and $v$ satisfies
a horizontality condition for $f_0$, then the TCE above
has a unique solution $\alpha$. In addition, $\alpha$ is continuous.

In  the case of piecewise expanding maps on the interval, the invariant density $\phi_t$ is a fixed point of a Perron-Frobenius
type transfer operator $\LL_t$ in an appropriate space, where $1$ is a simple isolated eigenvalue. 
So if we are able to verify some (weak) smoothness in the family $t \to \LL_t$, 
then we can  show  (weak) differentiability of $\mu_t$ by using perturbation theory. (We may use
different  norms in the range and the domain, in the spirit of Lasota-Yorke or Doeblin-Fortet inequalities.)
This is, roughly speaking, what was done in \cite{bs1} and \cite{bs4}
(as already mentioned, a serious additional difficulty in the presence
of critical points, which had to be overcome even in
the toy model, is the absence of a spectral gap on a space containing
the derivative of the invariant density). For Collet-Eckmann unimodal maps $f_t$, however,  
an inducing procedure or a
tower construction 
(\cite{keno}, \cite{young92}, \cite{young98})
is needed to obtain
good spectral properties  for the transfer operator
and to  properly analyse the density $\phi_t$, even for a single map.

We use the tower construction from \cite{BV}, under a Benedicks-Carleson
assumption.   However,
when we consider a one-parameter
family of maps $f_t$, the phase space of the tower moves
with $t$. To compare the operators
for $f_t$ and $f_0$,  it is convenient
to work with a finite part of the tower, the height of which
goes exponentially to infinity as $t\to 0$.
(We use results
of Keller and Liverani \cite{kellerliverani} to control the spectrum of the
truncated operator.)
The uniform boundedness of
$\alpha(c_k)$ is instrumental  in working with such truncated towers and operators.
In fact,  the tower construction in \cite{BV}  also has
a key role in the proof of boundedness for $\alpha$: The natural candidate for
the solution is a divergent series, but, under the horizontality
condition, we devise a dynamical resummation  
(the mantra being: ``don't perform a partial sum for the series
while you are climbing the
tower, unless you are ready to fall").

The tower from \cite{BV} has a drawback: The orbits of the edges
of the tower levels apparently create ``artificial discontinuities" in
the functions. To eliminate these potential
discontinuities, we modify the construction of the
Banach spaces and transfer operators on the towers by introducing
smooth cutoff functions (called $\xi_k$ below, see Section~\ref{spectralstuff}).
As a consequence, we obtain a 
new expression for the invariant density of a Benedicks-Carleson
unimodal map (Proposition~\ref{acim}), in terms
of a single smooth function and of the dynamics.

\medskip

We would like to list now a few directions for further
work. Several of them can be explored by
exploiting the techniques developed in the present article (see \cite{BalT} for other open problems):

\smallskip

$\bullet$ In the setting of the present paper, e.g., can one show that
$\partial_t \mu_t(A)|_{t=0}$ is a resummation of the divergent series
$\Psi_A(z)$ at $z=1$? (Presumably, a dynamical resummation is possible,
maybe using the operator $\PP (\psi)=\psi\circ f$
dual to $\LL$ acting on dual Banach spaces, and using, e.g., the proof of the main result in
\cite{HM}.) Can one get an Abelian limit along the real axis? The radius of convergence
of $\Psi_A(z)$ is strictly smaller than $1$ in general. There
appears to be an essential boundary, except in the subhyperbolic cases
when the critical point is preperiodic.  Analytic continuation in the usual sense is thus probably not available,
some kind of Borel or Abelian continuation seems necessary.
(In subhyperbolic
cases $\Psi_A(z)$ is  meromorphic, and
horizontality very likely implies vanishing of the residue of the pole in
$[0,1]$.)

$\bullet$ Can one replace the topological slow recurrence condition on $f_0$
by
Be\-ne\-dicks-Carleson, Collet-Eckmann, or possibly just a summability
condition on the inverse of the postcritical derivative
(see \cite{nssum} and \cite{BLS}), and still get differentiability
\footnote{Perturbation theory of isolated eigenvalues cannot be used
if there is no spectral gap, 
but the analysis  in Hairer-Majda \cite{HM}, e.g., indicates that
existence of the resolvent
$(\id-\LL)^{-1}$  (up to replacing $\LL$ by $\PP$ if necessary)
should be enough.}
of $t\mapsto \mu_t$, as a distribution of order
$1$, at $t=0$?

$\bullet$
If $f_t$ is a smooth
family of  quadratic unimodal maps, with $f_0$ 
a good map (summable, or Collet-Eckmann,
or Benedicks-Carleson, or TSR), and if
$v=\partial f_t|_{t=0}$ is horizontal for $f_0$, that is, \eqref{hor} holds
\footnote{By \cite{ALM}, we can heuristically view $f_t$ as tangent to the
topological class of $f_0$.},  is
$t\mapsto \mu_t$ differentiable, as
a distribution of order $1$, in the sense of Whitney, at $t=0$?

$\bullet$
If $f_t$ is a smooth (possibly transversal, that is, not horizontal)
family of  quadratic unimodal maps, with $f_0$ a good map, is
$t\mapsto \mu_t$ always $r$-H\"older in the sense of Whitney for 
$r\in (0,1/2)$ at $t=0$? Which is the strongest topology one can
use in the image? (Possibly, one could show H\" older continuity
in the sense of Whitney of the Lyapunov exponent.)

$\bullet$ Can one construct a (non-horizontal) smooth
family $f_t$ of  quadratic unimodal maps, with $f_0$ 
a good map, so that
$t\mapsto \mu_t$, as a distribution of
any order, is not differentiable (even
in the sense of Whitney, at least for large subsets) at $t=0$?
So that it is not H\"older for any exponent $>1/2$?

$\bullet$
What about H\'enon-like maps? Note that even the formula defining horizontality
is not available in this case, see \cite{BalT}
(Numerical results of Cessac \cite{cessac} indicate that $\Psi_A(z)$
has a singularity in the interior of the unit disc. In view of the
above discussion, we expect that this singularity
is not an isolated pole in general.)

$\bullet$
The dynamical zeta function associated to a
Collet-Eckmann map $f$ and describing
part of the spectrum of
$\LL$ was studied by Keller and Nowicki \cite{keno}.
Can one study the analytic properties of 
a {\it dynamical determinant} for $\LL$ 
in the spirit of what was done for subhyperbolic analytic
maps  \cite{BJR}? (Analyticity would hold only
in a disc of finite radius, and the correcting rational factor from
\cite[Theorem B]{BJR} would be replaced by
an infinite product, corresponding to the essential boundary
of convergence within this disc.) Can one find and describe a
dynamical determinant playing for $\Psi_A(z)$ the part
that $\LL$ plays for the Fourier transform
of the correlation  function of the SRB measure of $f$?
(See  \cite{baladi} for piecewise expanding interval
maps.)
\medskip

The structure of paper is as follows. In Section ~\ref{formall}, we give precise definitions
and state our main results formally. Section~\ref{alphabdedproof} is devoted to the
proof (by dynamical resummation)
that horizontality implies that the TCE has a continuous solution $\alpha$
(Theorem~\ref{alphabded}). 
In particular, we recall in Subsection~\ref{tower} the construction of the
tower map $\hat f :\hat I \to \hat I$ from \cite{BV} which will be used in later sections.
We also show (Subsection~\ref{prdiv})
that the formal candidate for $\alpha$ diverges at countably
many points (Proposition~\ref{div}).
In Section ~\ref{spectralstuff}, we revisit the tower construction, introducing
Banach spaces and a transfer operator $\widehat \LL$ involving the smooth cutoff
functions discussed above. In particular, Proposition~\ref{acip}, which immediately
implies our new expression for
the invariant density (Proposition ~\ref{acim}), is proved in
Subsection~\ref{gap}. Also, we study truncations $\widehat \LL_M$ on finite
parts of the tower in Subsection~\ref{trunk}.
Uniformity in $t$ of the hyperbolicity
constants of $f_t$ involved in the construction of Sections~\ref{alphabdedproof}
and ~\ref{spectralstuff}, is the topic of Section~\ref{recurr}, the main result
of which is Lemma~\ref{est1aunif} (proved by exploiting previous work of Nowicki).
Finally, our linear response result, Theorem~\ref{linresp}, is proved in
Section~\ref{finalproof}. The argument borrows some ideas from \cite{bs1}, but their implementation
required several nontrivial innovations, as explained above.
The three appendices contain proofs of a more technical nature. 
%%%%

\section{Formal statement of our results}
\label{formall}
\subsection{Collet-Eckmann, Benedicks-Carleson, and topologically
slowly recurrent (TSR) unimodal maps}

We start by formally defining the classes of maps that
we shall consider. Note  that
we shall sometimes write $[a, b]$ with 
$b<a$ to represent $[b, a]$. Another frequent abuse of notation is that
we sometimes use $C>0$ to denote different (uniform)
constants in the same
formula.

Let $I=[-1,1]$. We say that $f$ is  {\it S-unimodal} if   $f:I\to I$ is a $C^3$ map with negative Schwarzian derivative  such that 
$f(-1)=f(1)=-1$, $f'|_{[-1,0)} >0$, $f'|_{(0,1]} <0$, and
$f''(0)<0$ (i.e., we only consider the
quadratic case).  The following notation will be convenient throughout: For
$k\ge 1$, we let $J_+$ be the monotonicity
interval of $f^k$ containing $c$ and to the right of $c$,
$J_-$ be the monotonicity
interval of $f^k$ containing $c$ and to the left of $c$,
and we put 
\begin{equation}
\label{deff+}
f^{-k}_+:=(f^k|_{J_+})^{-1}\, ,
\, \, f^{-k}_-:=(f^k|_{J_-})^{-1}\, .
\end{equation}

\begin{remark} It is likely that the negative Schwarzian derivative assumption
is not needed for our results, see \cite{kozlo}. Note however that we cannot
apply trivially the work of Graczyk-Sands-\'{S}wi\c{a}tek \cite{GSS} to study linear
response:  If $f_t$ is
a one-parameter family of $C^3$ unimodal maps, the 
smooth changes of coordinates
which make their Schwarzian derivative  negative will depend on $t$,
and this dependency will require a precise study. In view of keeping
the length of this paper within reasonable bounds, we refrained
from considering the more general case.
\end{remark}

Let $c=0$ be the critical point of $f$,
and
put $c_k =f^k (c)$ for all $k \ge 0$. 
We say that  an $S$-unimodal map $f$ is
{\it $(\lambda_c,H_0)$-Collet-Eckmann (CE)} if $\lambda_c > 1$, $H_0\geq 1$, and 
\begin{equation}\label{00}
|(f^k)' (f(c)) |\ge  \lambda_c ^k\, , \quad \forall k \ge H_0 \, .
\end{equation}
All periodic orbits of Collet-Eckmann maps are repelling, and  \cite[Theorem B]{NoS} gives that
for {\it any} $C^2$ unimodal (or multimodal) map without periodic attractors
there exists $\gamma >0$
so that $|f^n(c)-c| \ge e^{-\gamma n}$ for all
large enough $n$. Benedicks and Carleson \cite{BC} showed 
that $S$-unimodal Collet-Eckmann maps which satisfy the  following {\it Benedicks-Carleson} 
assumption
\begin{equation}\label{BeC}
\exists 0<\gamma <\frac{\log( \lambda_c)}{4}
\mbox{ so that } |f^k(c)-c| \ge e^{-\gamma k}\, , \quad \forall k \ge H_0\, 
\end{equation}
form
a positive measure set of parameters of non degenerate families. 
The Be\-ne\-dicks-Carleson assumption will suffice for some of our results, sometimes up to replacing $4$  in the denominator
by a larger constant.

A stronger condition, topologically slow recurrence (TSR), will allow us to obtain linear response.
To define TSR, we shall use the following
auxiliary sequence:
Let $f$ be an $S$-unimodal map whose critical point is not preperiodic. The itinerary of a point $x \in I$ is 
the sequence  $\sgn (f^i(x))\in \{-1,0,1\}$.  We put 
\begin{equation}\label{Rf}
R_f(x):= \min \{j\mid \sgn (f^j(c))\neq \sgn (f^j(x)), \ j\geq 1   \} \, .
\end{equation}

We say that an $S$-unimodal map $f$  with non preperiodic critical
point
satisfies the {\it topological slow recurrence (TSR)} condition  if 
\begin{equation}\label{tsr}
\lim_{m\rightarrow \infty} \limsup_{n\rightarrow \infty} 
\frac{1}{n} \sum_{\substack{1\leq j\leq n  \\ R_f(f^j(c)) \ge m}} R_f(f^j(c))=0 \, .
\end{equation}

It follows from the definition  that any $S$-unimodal map topologically
conjugated with a map $f$  satisfying
TSR also satisfies  TSR. (Indeed,
$R_f(c_n)=j$ if and only if
  $f^{j}$ is a diffeomorphism on $(c,c_n)$ and $c \in f^{j}(c,c_n)$.)  
We also have the much less trivial result below: 
\begin{proposition}[\cite{sandst}, \cite{wang}. See also \cite{luzzatto}] 
An $S$-unimodal map $f$ with non preperiodic
critical point satisfies the TSR condition if and only if $f$ is a Collet-Eckmann map and 
\begin{equation}\label{sr}
\lim_{\eta\rightarrow 0^+} \liminf_{n\rightarrow \infty} 
\frac{1}{n} 
\sum_{\substack{1\leq j\leq n  \\ |f^j(c) -c|< \eta }} \log |f'(f^j(c))|=0
\, .
\end{equation}
\end{proposition}

In Section ~ \ref{recurr}, we shall prove that TSR implies Collet-Eckmann
and Benedicks-Carleson-type conditions, uniformly
in a subset of small enough $C^3$ diameter
of a topological class.

\begin{remark}[TSR is generic]\label{generic2}
Avila and Moreira \cite{bif}  proved that for almost every parameter $s$ in a non-degenerate 
analytic family of quadratic unimodal maps $f_s$, the map $f_s$ is either regular 
or Collet-Eckmann with subexponential recurrence of its critical orbit
(i.e., for every $\gamma >0$, there is $H_0$ so that
$|c_k-c|> \exp(-\gamma k)$ for all $k\ge H_0$).  
(Non-degenerate, or transversal, means that the family is not contained in a topological class.)
Tsujii \cite{tsujii} had previously proved
% a generalization of Jakobson's theorem which implies in particular  
that the set of Collet-Eckmann and subexponentially recurrent
parameters $s$  in a transversal family $f_s$ of $S$-unimodal maps  
has positive Lebesgue measure.
By combining the results of Avila and Moreira \cite{bif} and Tsujii \cite{tsujii},
we can see that TSR is  a generic condition:
In a nondegenerate analytic family $f_s$
of $S$-unimodal maps, almost every parameter
is either regular or TSR.
\end{remark}

\subsection{Boundedness and continuity of the infinitesimal conjugacy $\alpha$}
\label{sscont}

Let $f$ be an $S$-unimodal Collet-Eckmann map, and let $v : I \to \complex$ be bounded.
We want to find   a bounded solution $\alpha: I \to \complex$ 
of the {\it twisted cohomological equation (TCE):}
\begin{equation}\label{tce}
v(x)= \alpha(f(x))- f'(x) \alpha(x)  \, , \forall  x \in I \, .
\end{equation}

By analogy with the piecewise expanding unimodal case (that we studied
in previous works \cite{bs1}, \cite{bs2}), a candidate $\alpha_{cand}$ for the solution $\alpha$
of \eqref{tce}  is defined, for those
$x \in I$ so that $f^j(x)\ne c$ for all $j \ge 0$, by the  {\it formal series}
\begin{equation}\label{defa}
\alpha_{cand}(x)=- \sum_{j=0}^\infty \frac{v(f^j(x))}{(f^{j+1})'(x)}  \, ,
\end{equation}
and, for those $x \in I$ so that there  exists $j \ge 0$ 
with  $f^j(x)= c$, but $f^\ell (x) \ne c$ for $0 \le \ell\le j-1$, 
by the sum
\begin{equation}\label{defa'}
\alpha_{cand}(x)=- \sum_{\ell=0}^{j-1} \frac{v(f^\ell(x))}{(f^{\ell+1})'(x)}  \, .
\end{equation}
(In particular,  $\alpha_{cand}(c)=0$.) 
Clearly, the series \eqref{defa} converges absolutely at every point $x$ for which 
the Lyapunov exponent
$$
\Lambda(x)=\lim_{j\to \infty} \log |(f^j)'(x)|^{1/j}
$$ 
is well-defined and strictly positive.
In particular, \eqref{defa} converges absolutely  for $x$ in  the forward orbit 
$\{ c_k , k \ge 1\}$
of the critical point  of 
the Collet-Eckmann $S$-unimodal map $f$, and also on the set of its
preperiodic  points.

We say that $v$  satisfies the {\it horizontality condition}  if
\begin{equation}\label{hor}
v(c)=- \sum_{j=0}^{\infty} \frac{v(f^j(c_1))}{(f^{j+1})'(c_1)}  \, ,
\end{equation}  
(note that
the right-hand-side of the above identity is just $\alpha_{cand}(c_1)$). If  $v$  satisfies the horizontality, 
then it is easy to see that whenever the  formal series \eqref{defa}
for $\alpha_{cand}(x)$ converges absolutely, then the corresponding series
 $\alpha_{cand}(f(x))$  also converges absolutely, and $\alpha_{cand}$ satisfies the
twisted cohomological equation \eqref{tce} at $x$.
Violation of horizontality (that is,
$v(c)\ne \alpha_{cand}(c_1)$)  is a {\it transversality} 
condition which
has been used for a long time 
in one-parameter families
$f_t$  of smooth unimodal maps  with $v(x)=\partial_t f_t|_{t=0}$
(see, e.g., \cite{TTY} for the transversality
condition, see \cite{tsujii} for the transversality condition expressed as a postscritical sum,
see, e.g., \cite[\S 5]{bs2} for the link between the two expressions, see \cite{ALM} for a recent
occurrence, and see \cite{ruelle} for its use in linear response).

Nowicki and van Strien \cite{nssum}  showed that the absolutely continuous invariant probability
measure  $\mu$ of a quadratic Collet-Eckmann map satisfies 
$$\mu(A)\leq C m(A)^{1/2}\, ,
$$ 
where $m$ is the Lebesgue measure. 
In particular $\log |f'|$ is $\mu$-integrable, and for  
Lebesgue almost every point $x$ the Lyapunov exponent 
$\Lambda(x)$
is well-defined and positive and
coincides with $\int \log |f'| \ d\mu$
(see Keller \cite{keller}). So  the series $\alpha_{cand}(x)$ converges absolutely at Lebesgue almost every point
$x$, and if $v$ is horizontal then
$\alpha_{cand}$ satisfies the TCE \eqref{tce}
along  the forward orbit of each such good $x$. However it is not clear a priori that there exists an upper bound for $|\alpha_{cand}(x)|$ on the set where $\alpha_{cand}(x)$ converges absolutely (for example,  $c_k$ may
be very close to $c$).

One can ask whether the formal series $\alpha_{cand}(x)$ converges everywhere. We shall show in Proposition \ref{div}
that for fairly general  $v$ (see Remark ~ \ref{rkdiv}), the series
\eqref{defa} for $\alpha_{cand}(x)$ diverges on a uncountable and dense subset (this set has Lebesgue measure
zero, however, by the observations in the previous paragraph). 
This lack of convergence is a new phenomenon with respect
to \cite{bs1}, \cite{bs2}. 
In order to prove that the TCE nevertheless has
a bounded solution in the horizontal case, we shall make
a Benedicks-Carleson assumption \eqref{BeC}
 on $f$, and we shall group  the terms 
of the formal series  $\alpha_{cand}(x)$  
to obtain an absolutely convergent series. The resummation procedure depends on $x$ through its  dynamics with respect
to  an induced map on the tower introduced in
\cite{BV}, using strong expansion properties available in the
Benedicks-Carleson case. 
This dynamical resummation will allow us to prove our first main result:

\begin{theorem}[Boundedness and continuity of $\alpha$]\label{alphabded}
Assume that $f$ is a $(\lambda_c,H_0)$-Collet-Eckmann $S$-unimodal map satisfying the Benedicks-Carleson condition \eqref{BeC}.

For any bounded function $v:I \to \complex$, if the TCE \eqref{tce} admits
a bounded solution $\alpha: I \to \complex$ with $\alpha(c)=0$, then this solution is unique and $v$  satisfies the horizontality condition   \eqref{hor}.

Let $X: I \to \complex$ be Lipschitz, and let  $v=X \circ f$.
If   $v$  satisfies the horizontality condition   \eqref{hor},
then there exists a  continuous function $\alpha: I \to \complex$ with $\alpha(c)=0$ solving the TCE  \eqref{tce}. 
In addition, $\alpha(x)=\alpha_{cand}(x)$ for all
$x$ so that $f^j(x)=c$ for some
$j\ge 0$, or so that the infinite series $\alpha_{cand}(x)$ 
in \eqref{defa} converges absolutely.
\end{theorem}

The condition $v=X \circ f$ can be weakened but the term
$k=0$ of $(I)$  in \eqref{I-II-III}  in the proof of Proposition~\ref{abs}
shows that we need something like 
$v'(c)=0$. (If we allowed $f''(c)=0$, then we would 
need $v''(c)=0$, etc.)

We do not know whether Theorem ~ \ref{alphabded}
holds for all  $S$-unimodal Collet-Eckmann maps, i.e., whether the Benedicks-Carleson
assumption is needed. In any case, we shall use the stronger, but
still generic (recall Remark ~\ref{generic2}), TSR assumption 
in Section ~ \ref{recurr} to show uniformity of the various
hyperbolicity constants. This uniformity is required
to prove linear response, the other main result of this paper.

The proof of Theorem \ref{alphabded} is given in Section \ref{alphabdedproof} 
and  organised as follows: In Section \ref{tower}, we recall the tower construction from 
Baladi and Viana  \cite{BV}. We study its properties in Section~\ref{tower'}, which also
contains two new (and key) estimates,
Proposition ~\ref{ubalpha} and its Corollary ~ \ref{errorii}. 
In  Section \ref{resume},  we  define a function
$\alpha(x)$ by 
grouping the terms of the formal series \eqref{defa} to obtain an absolutely convergent series (Definition ~\ref{resumdef} and Proposition~\ref{abs}). The  
resummation procedure for $\alpha(x)$ depends on the dynamics of $x$ on the tower. Finally, in Section \ref{alphac} 
we complete the proof of Theorem \ref{alphabded}:
We show that $\alpha(x)$ is a continuous function, that it satisfies the TCE,
and that if the TCE admits a bounded solution then it is unique.

We end this section with a result on the lack of convergence of
the formal power series  for $\alpha_{cand}(x)$
(recall that it converges at Lebesgue almost every  $x$):

\begin{proposition}\label{div} Let $f$ be an $S$-unimodal map, with all its periodic points repelling and an
infinite postcritical orbit.  Let $v$  be a $C^1$ function on $I$, with $v'(c)=0$, such that $v(f^{i_0}(c))\neq 0$ for some $i_0$. Let $\Sigma$ be the set of points $x$ such that $f^n(x)\neq c$ for every $n\geq 0$ and so that the series 
$\alpha_{cand}(x)=-\sum_{i=0}^{\infty}  \frac{v(f^i(x))}{(f^{i+1})'(x)}$ diverges.
Then for every non empty open set $A \subset I$, the intersection $A \cap \Sigma$
contains a Cantor set. 
\end{proposition} 

\begin{remark}\label{rkdiv}
If $f$ is a Collet-Eckmann map whose critical orbit is not preperiodic, an open and dense set of horizontal vectors $v$ satisfies the conditions of Proposition~ \ref{div}. Indeed, the set 
$\{v\mid v(f^i(c))=0, \forall \ i\}$
is a subspace of infinite codimension, and the subspace of horizontal directions $v$ has codimension one.
\end{remark}

The proof of Proposition ~\ref{div} is to be found in Section ~ \ref{prdiv}.

%%%%%%%%%%%%%%%%%%%%%%%%%%%%%%%%%%%%%%%%%%%%%%%%%%%%%%%%%%%%%%%%

\subsection{A new expression for the a.c.i.m. of a Benedicks-Carleson unimodal map}
 
It is well-known that an $S$-unimodal map which is Collet-Eckmann  admits an absolutely continuous invariant measure.
The following  expression for the invariant density of a Benedicks-Carleson
unimodal map appears to be new. It is a byproduct of our
proof, and follows immediately from
Proposition ~\ref{acip} and the definitions in Section~\ref{spectralstuff}
(the case when the critical point is preperiodic  can
be obtained by a much more elementary proof).
The remarkable feature of \eqref{smoothrho} is that
the defining function $\psi_0$ is {\it smooth,} and that the
square-root singularities appear dynamically, through
the inverse interates of $f$ and their jacobians.

\begin{proposition}\label{acim}
Let $f$ be a $(\lambda_c, H_0)$-Collet-Eckmann $S$-unimodal map 
satisfying the  Benedicks-Carleson condition
\eqref{BeCxtra}, with $c$ not preperiodic. 
If $f$ is $C^{4}$, then there exist
\begin{itemize}
\item
a $C^1$ function $\psi_0 : I \to \real_+$, which belongs to the
Sobolev space $H^2_1$,
\item 
for each $k\ge 1$,
neighbourhoods 
$V_k \subset W_k$ of $c=0$, so that $f^k|_{W_k\cap [0,1]}$ and
$f^k|_{[-1,0]\cap W_k}$ are injective, 
\item
for each $k\ge 1$, a
$C^\infty$ function
$\xi_k: I \to [0,1]$,  supported in $W_k$ and $\equiv 1$ on $V_k$,
\end{itemize}
so that
the density $\phi$
of the unique absolutely continuous invariant
probability measure of $f$ satisfies
\begin{equation}\label{smoothrho}
\phi(x)=\psi_0(x) + \sum_{k = 1}^{\infty}\sum_{\varsigma\in \{+,-\}}
 \frac{\prod_{j=0}^{k-1}\xi_j (f^{-k}_\varsigma(x))}{|(f^{k})'(f^{-k}_\varsigma(x))|} 
 \chi_k(x) \psi_0(f^{-k}_\varsigma(x)) 
\, ,
\end{equation}
where  $\chi_k=1_{[-1, c_k]}$ if
$f^k$ has a local maximum at $c$, while $\chi_k=1_{[ c_k,1]}$ if
$f^k$ has a local minimum at $c$.
\end{proposition}

The length of $W_k$ must decay exponentially,
but there is some flexibility in
choosing the intervals $V_k$, $W_k$, and the functions $\xi_k$, see Definition
~\ref{defxi}
for details, noting also the parameter $\delta$ used in the
construction of the tower.
The function $\psi_0$
depends on these choices. 

By Lemma~ \ref{rootsing}, which describes the nature of the
singularities of $|(f^{k})'f^{-k}_\pm(x)|$
on the support of $\xi_k(f^{-k}_\pm(x))$, the expression for $\phi$ belongs to
$L^p(I)$ for all $p < 2$. 
In fact,  Lemma~ \ref{rootsing} (see also \eqref{kappa'})
implies that the invariant density of $f$ can be written as
$$
\psi_0+\sum_{k\ge 1}\phi_k  \frac{\chi_k }{\sqrt{|x-c_k|}} \, ,
$$
where the $C^1$ norms of the $\phi_k$  decay exponentially with $k$. 
(A slightly weaker version of this  result, replacing
differentiable by bounded variation, was first proved by L.S.Young
\cite{young92}. Ruelle obtained a formula involving differentiable
objects in the analytic
Misiurewicz case \cite{ruelle}, but his expression is somewhat less dynamical.)

%%%%%%%%%%%%%%%%%%%%%%%%%%%%%%%%%%%%%%%%%%%%%%%%%%%%%%%%%%%%%%%%

\subsection{Uniformity of hyperbolicity constants in deformations of 
slowly recurrent maps}

We shall  study one-parameter families $t \mapsto f_t$
of $S$-unimodal maps 
which stay in a topological class, i.e., {\it deformations:}

\begin{definition}($C^r$ deformations $f_t$. Notations $v_t$, $X_t$, $h_t$.)
Let $f:I \to I$ be an  $S$-unimodal 
Collet-Eckmann map.
For $r \ge 1$, a {\it $C^r$ one-parameter family} through $f$ is a 
$C^r$ map 
$$t \mapsto f_t\, , \, \, t \in [-\epsilon, \epsilon]
\, , 
$$
(taking the topology of $C^3$ endomorphisms 
of $I$ in the image),
with
$f_0=f$, and so that
each $f_t$ is $S$-unimodal. We use the notations: 
$$
c_{k,t}=f_t^k(c)\, , k \ge 1 \, , \, v_s:= \partial_t f_t|_{t=s}\, ,
v=v_0\, , \, c_k=c_{k,0}\, .
$$

A {\it $C^r$ deformation of $f$} is  a $C^r$ one-parameter family through $f$
so that, in addition, for each $|t|\le \epsilon$, there exists
a homeomorphism $h_t:I \to I$ with
\begin{equation}\label{starstar}
h_0(x)\equiv x\, , \mbox{ and }
 f_t \circ h_t = h_t \circ f_0\, ,\forall t \in [-\epsilon, \epsilon]\, ,
\end{equation}
and
\footnote{This is mostly a technical
assumption, see also the remark after Theorem ~ \ref{alphabded}.}
$v_s = X_s \circ f_s$ for each $|s|\le \epsilon$,
with $X_s : I \to \real$ a $C^2$ function.
(We write $X=X_0$.)
\end{definition}

\begin{remark}
\label{defhor}
If $f_t$ is a deformation
then each $v_s$ is horizontal. (This was proved by Tsujii \cite{tsujii}.)
\end{remark}

Given Theorem ~ \ref{alphabded},
the next lemma is  easy to prove. It is essential
in our argument:

\begin{lemma}\label{first}
Let $f_t$ be a $C^1$ one-parameter family of Collet-Eckmann
$S$-unimodal maps
through $f=f_0$. Assume that $v$ is horizontal, that is,
$\alpha_{cand}(c_1)=v(c)$. 

Then, we have for all $k \ge 1$
$$
\lim_{t \to 0}\frac{c_{k,t}-c_k}{t}=\alpha_{cand}(c_k)\, .
$$

If, in addition, $f_t$ is a $C^1$ deformation of $f_0$
then 
$$\partial_t h_t(c_k)|_{t=0}=\alpha_{cand}(c_k)=\alpha(c_k)
\, , \quad \forall k \ge 1\, .
$$
\end{lemma}

\begin{proof}
Our assumptions ensure that for each $k\ge 1$ the limit
$$
a(c_k)= \lim_{t \to 0} \frac{c_{k,t}-c_k}{t}
$$
exists. Clearly,  $a(c_1)=v(c)$. More generally, it is easy to check
that we have
$$
a(c_k)=\sum_{j=0}^{k-1} (f^j)' (c_{k-j}) v(c_{k-j-1}) \, ,
$$
so that $a$  satisfies the TCE (\ref{tce}). By the horizontality
condition, this implies that
$a=\alpha_{cand}$ on $\{ c_k\}$.

The additional assumption that $f_t$ and $f_0$ are
conjugated via $h_t$ implies that
$h_t(c_k)=c_{k, t}$, for all $t$ and
all $k\ge 0$.
The last statement
of Theorem ~ \ref{alphabded} implies that $\alpha_{cand}(c_k)=\alpha(c_k)$.
\end{proof} 

The following fact is an immediate consequence of 
van Strien's remark on ``robust chaos" \cite[Theorem 1.1]{Str},
using the well-known fact that 
Collet-Eckmann maps do not have any attracting periodic
orbit:

\begin{lemma}\label{robustchaos}
Let $f:I \to I$ be a  $S$-unimodal $(\lambda_c, H_0)$
Collet-Eckmann map and
let $f_t$ be   a $C^1$ one-parameter family through $f$.
If each $f_t$ is Collet-Eckmann for some
parameters $\lambda_c(t)$ and $H_0(t)$, then $f_t$ is  a
$C^1$ deformation of $f$.
\end{lemma}

In the other direction, although topological invariance of
the Collet-Eckmann condition is known,
we do not know how to prove topological invariance of
Be\-ne\-dicks-Car\-leson conditions \eqref{BeC} (or variants of the
type \eqref{stBeC}--\eqref{BeCxtra}).
Since we also need uniformity of the various constants in
the definitions, we shall work with the stronger, but still generic
(see Remark ~ \ref{generic2}),
assumption of topological slow recurrrence TSR (recall \eqref{tsr}).
In Section ~\ref{recurr}, assuming  for simplicity
that the maps are symmetric, we shall prove that if $f_s$
is a $C^0$ deformation of a TSR map $f_0$
then the various hyperbolicity constants of $f_s$ (that is
$\lambda_{c}(f_s)$ and $H_0(f_s)$, 
 but
also $\sigma(f_s)$, $c_{f_s}(\delta)^{-1}$, $\rho(f_s)$
from Subsection ~ \ref{tower'},
and, especially, $\gamma(f_s)$)
are uniform in small $s$.
We refer to Lemma ~ \ref{est1aunif} for a precise statement.
%(Uniformity of
%the Lipschitz norm of $X_s$ and of the supremum of $v_s$, where $v_s=\partial_t f_t|_{t=s}$
%and $X_s=v_s \circ f_s$, are obvious if $f_s$ is
%a $C^2$ deformation.)
Also, it will follow from Propositions~\ref{unce} and \ref{ubc} that (TSR) 
implies (CE) and (BeC).

Uniformity of constants implies the following result, essential in
many places in the proof of Theorem ~ \ref{linresp}:

\begin{lemma}\label{first'''}
Let $f_t$ be a $C^1$  deformation of symmetric
$S$-unimodal maps so that $f_0$ enjoys
topological slow recurrence TSR. Then there exist $\epsilon >0$ and $L < \infty$ 
so that
\begin{equation}\label{uniforma}
\sup_x \sup_{|s|< \epsilon} |\alpha_s(x)|\le L\, ,
\end{equation}
and
\begin{equation}\label{Lbound}
|c_k-c_{k,t}| \le L |t|\, \, \quad \forall k \ge 1\, ,\quad
\forall |t|< \epsilon \, .
\end{equation}
\end{lemma}

\begin{proof}
We may assume that the critical point is not preperiodic.
(If it is, the proof is much easier.)

Denote by $\alpha_s$  the continuous solution to the
TCE given by Theorem ~\ref{alphabded} applied to
each $f_s$ (the assumptions of the theorem are satisfied
because of Lemma~ \ref{est1aunif}).
The proof of Proposition ~ \ref{abs}
shows that for each
fixed $s$ the supremum $\sup_x |\alpha_s(x)|$ 
may be estimated in terms
of $c_{f_s}(\delta)^{-1}$, $\sup |v_s|$, $\Lip X_s$,
$(1-\sigma(f_s)^{-1})^{-1}$, $(1-(\rho(f_s))^{-1})^{-1}$, in the
notation of Lemma~\ref{est1a}. 
By Lemma~ \ref{est1aunif}, this implies \eqref{uniforma}.

Next, applying Lemma ~\ref{first} to each $f_s$,
we get
$$
 \lim_{t \to 0} \frac{c_{k,s+t}-c_{k,s}}{t}=\alpha_s(h_s(c_k))
 =\alpha_s(c_{k,s})\, .
$$
In other words, $t\mapsto h_t(c_k)$ is differentiable
on $[-\epsilon, \epsilon]$ (with $\epsilon$ independent of
$k$),
with derivative $\alpha_s(c_{k,s})$.

Then, for each $k \ge 1$
and each  $|t|< \epsilon$ 
the mean value theorem gives $s$ with $|s|\le |t|$ so that
$$
\frac{c_{k,t}-c_k}{t}=
\frac{h_t(c_k)-h_0(c_k)}{t}= \alpha_s(c_{k,s}) \, .
$$
\end{proof}

%%%%%%%%%%%%%%%%%%%%%%%%%%%
\subsection{Linear response}

Our main result will be proved in Section ~ \ref{finalproof}:

\begin{theorem}[Linear response and linear response formula]\label{linresp}
Let $\eta>0$ and let $f_t$  be a $C^2$ 
deformation of a $C^{4}$ $S$-unimodal map $f_0$ which satisfies TSR. 
Assume that  all maps $f_t$ are symmetric.
Write $\mu_t=\phi_t\, dx$ for the unique absolutely continuous invariant
probability of $f_t$. Then, letting $(C^1(I))^*$
be the dual of $C^1(I)$, the map
\begin{equation}\label{mapp}
t \mapsto \mu_t \in (C^{1}(I))^*\, , \,\, t \in [-\epsilon, \epsilon]
\end{equation}
is differentiable. In addition, for any $A \in C^1(I)$,
\begin{equation}\label{ruelleformula}
\int_I (A- A \circ f_s) \partial_t \mu_t |_{t=s}
=\int_I A' X_s \phi_s \, dx \, ,
\end{equation}
\end{theorem}

The formula \eqref{ruelleformula} is an easy consequence of
the differentiability of \eqref{mapp}, as was pointed out to
us by Ruelle \cite{ruellepriv}.

We next give an  explicit formula  for 
$\partial_t \mu_t|_{t=0}$ (we choose $t=0$
for definiteness). For this, we need further notation.
Introduce
$Y_{k,s}=\lim_{t \to s}\frac{ f^k_t- f^k_s}{t-s}$ for $k \ge 1$
(we  write $Y_k$ instead of $Y_{k,0}$). Then
$Y_{1,s}= X_s \circ f_s$, $Y_{2,s}=X_s \circ f^2_s+ (f'_s\circ f_s) X_s \circ f_s$, and
\begin{equation}\label{defvk}
Y_{k,s}= \sum_{j=1}^{k} ((f_s^{k-j})'\circ f_s^{j})
\cdot
X_s \circ f_s^{j}\, , \, k \ge 1 \, . 
\end{equation}
Put (note the shift in indices!)
\begin{equation}\label{hatY}
\hat Y_s=(\hat Y_{s}(x,k)=Y_{k+1,s}(x), k\ge 0) \, ,
\quad \hat Y=\hat Y_0\, .
\end{equation}
Referring to Section ~\ref{spectralstuff}
for the definitions 
of $\lambda$, $\widehat \LL$, $\hat \phi$,  $\Pi$,
and $\TT_0$, and
summing
\eqref{formula1} and \eqref{formula2} from the proof of 
Theorem~\ref{linresp}, we get
\begin{align*}
\nonumber
&\int A \partial_t \mu_t |_{t=0}\\
%\label{formula}
&\quad\qquad=-\int A\cdot 
\Pi((\id - \widehat \LL)^{-1}\TT_0(\widehat \LL (\hat Y \hat \phi))'
\, dx
-\lambda
\int A' \cdot \Pi  (( \id- \TT_0) ( \widehat \LL (\hat Y \hat \phi)) )\, dx\, .
\end{align*}
Using the definitions of $\widehat \LL$, $\hat \phi$,  $\Pi$,
and $\TT_0$,
the linear response formula 
above  can be rewritten in terms
of $f$ and the functions $\psi_0$, $\xi_k$ and $\chi_k$ from
Proposition~\ref{acim}. The reader 
 can then  compare this rewriting to the  expression  in 
\cite[\S 17 and \S18]{ruelle},
obtained under the additional assumptions
that $f_0$ is  Misiurewicz and
all the $f_t$ are real analytic.
(See also Remark~\ref{resolv} below.)

We next discuss \eqref{ruelleformula}.
Using \eqref{formula1''}, we find
$$
\int (A-A \circ f) \partial_t \mu_t |_{t=0}
=\int A'
\TT_0(\widehat \LL (\hat Y \hat \phi))\,  dx
-\lambda
\int (A' -
(A\circ f)')\cdot \Pi ( ( \id- \TT_0) \widehat \LL (\hat Y \hat \phi)  )  \, dy \, .
$$
By Theorem~\ref{linresp},
the right-hand-side above coincides 
with the expression \eqref{ruelleformula}.
We sketch here a direct proof of this fact: 
The left-hand-side above
being independent of the parameter $\delta$ used in the construction
of the tower (note however that
$\Pi$, $\widehat \LL$, and $\hat \phi$ depend on
$\delta$), we can let $\delta\to 0$. 

We expect that, when $\delta\to 0$, the function
$ \phi_{0}$ converges to $\phi$
in the $L^1$ topology, and that
for any continuous function $B$, on the one hand, we have
\begin{equation}\label{ruelleok1}
\lim_{\delta \to 0}
\int_I B \cdot \Pi ( ( \id- \TT_0)  (\widehat \LL
( \hat Y \hat \phi))) \, dy=0\, ,
\end{equation}
and on the other hand, using in particular \eqref{claim0}
and the facts that 
\footnote{Note also that, as $\delta \to 0$, the smallest 
level 
from which points may fall from the tower
tends to infinity.}
$Y_1=X_0 \circ f_0$
and $\widehat \LL (\hat\phi)=\hat\phi$, we get
\begin{equation}\label{ruelleok2}
\lim_{\delta \to 0}
\int B
\TT_0(\widehat \LL (\hat Y \hat \phi))\, dx
=\int B X \phi\, dx\, .
\end{equation}

\begin{remark}\label{resolv}
If $\LL$ denotes the transfer operator
defined on distributions $\upsilon$ of order
one  by $\int A \LL \upsilon=\int (A \circ f) \upsilon$,
for all $C^1$ functions $A$, 
expression \eqref{ruelleformula} can be written
as a  left inverse  (both sides should be viewed as
distributions of order one)
\begin{equation}\label{implicit}
\partial_t \mu_t|_{t=0}=
-(\id -\LL)^{-1}_L ( X \phi)'\, dx \, .
\end{equation}
Indeed
\begin{align*}
\int (A- A \circ f) \partial_t \mu_t|_{t=0}
&=\int A (\id - \LL)\partial_t \mu_t|_{t=0}\, ,
\end{align*}
and, in the sense of distributions (writing $\mu=\mu_0$
as usual), 
$$
\int A' X \phi \, dx=
-\int A( X \phi)' \, dx=
-\int A \mbox{div}_{\mu} (X)\,  d\mu\, .
$$
We next explain the heuristics of the connection with the susceptibility
function.
In situations where more information is available (such as
smooth expanding circle maps), the following
formal manipulations become licit (they are not licit in the
present case of smooth unimodal maps, in particular the sum below
diverges in general):
\begin{align*}
\int A (\id -\LL)^{-1}_L ( X \phi)'\, dx
&=\int A \sum_{j=0}^\infty \LL^j ( X \phi)'\, dx\\
&=\int \sum_{j=0}^\infty \LL^j ((A\circ f^j) ( X \phi)')\, dx
=-\sum_{j=0}^\infty\int  (A\circ f^j)'  X \phi\, dx\, ,
\end{align*}
so that

$$
\int A \partial_t d\mu_t|_{t=0}
=\sum_{j=0}^\infty\int  (A\circ f^j)'  X d\mu=\Psi_A(1)
$$
\end{remark}

\medskip

We end this section
with a  result that we shall not need, but which
is
of independent interest (the proof is given in Appendix ~ \ref{htC1etc}):

\begin{proposition}[The solution of the TCE is an infinitesimal conjugacy]\label{htC1}
Let $t \mapsto f_t$  be a 
$C^1$
deformation of the $S$-unimodal $(\lambda_c, H_0)$
Collet-Eckmann map $f_0$. Assume furthermore that for each $|t|\le \epsilon$ there exists a unique
continuous function $\alpha_t$ on $I$ which solves the TCE
\eqref{tce} for $v=v_t := \partial_s f_s|_{s=t} $ and $f=f_t$, 
and in addition that the family $\{ \alpha_t\}_{|t|\le \epsilon}$ of 
\footnote{By compactness of $I$ and $[-\epsilon,\epsilon]$,
continuity is equivalent to uniform continuity here.}
continuous maps is equicontinuous.
Then for each $x\in I$ the function $t \mapsto h_t(x)$
is $C^1$, and  
$$\partial_s h_s(x)|_{s=t}=\alpha_t (h_t(x))\, ,\qquad
\forall \,  t \in (-\epsilon, \epsilon)\, .
$$
\end{proposition}

Note that if each $f_t$ satisfies the Benedicks-Carleson 
assumption \eqref{BeC} 
then Theorem ~ \ref{alphabded} ensures that the TCE associated to
$f_t$ and $v_t$ has a unique solution $\alpha_t$, which is continuous
(recall Remark ~ \ref{defhor}).
We expect that equicontinuity of the family $\alpha_t$
can be obtained, possibly under  the topological slow recurrence condition
TSR.

%%%%%%%%%%%%%%%%%%%%%%%%%%%%%%%%%%%%%%%%%%%%%%%%%%

\section{Proof of Theorem \ref{alphabded}: Boundedness and continuity
of the solution $\alpha$ of the TCE}\label{alphabdedproof}

\subsection{\label{tower} The tower map $\hat f$, the times $S_i(x)$
and $T_i(x)$, and the intervals $I_j$.}

Before  recalling the tower construction from \cite{BV}, 
we mention crucial expansion properties
of Collet-Eckmann maps which improve
\footnote{The improvements are: The Benedicks-Carleson condition is not needed,
the expansion factor $\rho$ in  \eqref{estlaeq2} can be taken arbitrarily close to
$\sqrt {\lambda_c}$, and  $c(\delta) > C|\delta|$.
The flexibility on $\rho$ means we can 
take any $\rho \in (e^\gamma, e^{-\gamma}\sqrt {\lambda_c})$
in  Lemma ~\ref{expiii}. This makes 
\eqref{BeC} sufficient for Proposition ~\ref{ubalpha}, with no condition relating
$\sigma$ and $\gamma$.} 
over \cite[Lemma 1]{BV}:

\begin{lemma}[Collet-Eckmann maps expansion]\label{est1a}  
Let $f$ be an 
$S$-unimodal $(\lambda_c,H_0)$-Collet-Eckmann map.

There exist $\sigma > 1$  and $C>0$
and for every small 
$\delta > 0$ there exists  $c(\delta) > C\delta$  such that
\begin{equation}
\label{estlaeq1} 
|(f^i)'(x)|\geq c(\delta) \sigma^i\, , \, \forall 0 \le i \le j \, , \,
\forall x \mbox{ so that }
|f^k(x)|>\delta\, , \,
\forall 0 \le k < j
\, . 
\end{equation}
 
For  every 
$1<\rho < \sqrt {\lambda_c}$ 
there exists $C_1=C_1(\rho)\in (0,1]$
and for each $\delta_0>0$ there exists
$\delta \in (0, \delta_0)$   such that  
 \begin{equation}
\label{estlaeq2}|(f^j)'(x)|\geq C_1  \rho^j\, , \, 
\forall
x \mbox{ so that } |f^i(x)|> \delta\, , \, \forall 0 \le i < j
\, , \, |f^j(x)|\leq \delta \, .
\end{equation}
In addition, we can assume that either $\pm \delta$
are preperiodic points, or that they have infinite orbits and that their Lyapunov exponents
exist and are strictly positive.
\end{lemma}

\begin{remark}
Except in remarks~\eqref{ruelleok1}
and \eqref{ruelleok2}, and, more importantly, in Remark~\ref{overlap}
(which is used in Appendix~\ref{app2}),
we do not use  that $c(\delta)\ge C \delta$, only
that $c(\delta) >0$ if $\delta >0$.
\end{remark}

\begin{proof} By Theorem 7.7 in \cite{ns}, there exist $\sigma > 1$ and $K > 0$ such that
$$|(f^i)'(x)|\geq K  \sigma^i \min_{0\leq k < i} |f'(f^k(x))|\, .
$$
Since $|f'(y)|\geq \widetilde K |y|$ and $|f^k(x)|\geq \delta$ for $k < j$ and $i \le j$, we have (\ref{estlaeq1}). 

To prove (\ref{estlaeq2}), we use Proposition 3.2(6) in \cite{No} which says that
for every $1<\rho < \lambda_c^{1/2}$ there exists $\widetilde C$ such that if $f^j(y)=c$ then 
\begin{equation} 
\label{invim} |(f^j)'(y)|\geq \widetilde C \rho^j \, .
\end{equation}
By Theorem 3.2 in \cite{martens} and the Koebe lemma, there exist $K > 0$ and arbitrarily small 
$\delta > 0$  such that the following holds: if $|f^i(x)| > \delta$ for $0\le i < j$ and $|f^j(x)|\leq \delta$ 
then there exists
an interval $J$, with $x \in J$,  such that $f^j(J)=[-\delta,\delta]$,
 $f^j$ is a diffeomorphism on 
$J$, and 
\begin{equation}
\label{invimbd} \frac{1}{K} \leq \frac{|(f^j)'(y)|}{|(f^j)'(x)|}\leq K\, ,\, 
\forall y \in J \, .
\end{equation}
Let $y \in J$ be such that $f^j(y)=c$. By (\ref{invim}) and (\ref{invimbd})  it follows that 
$$|(f^j)'(x)|\geq \frac{\widetilde C}{K} \rho^{j  } \, .   
$$
By the principal nest construction
in \cite{lyu}, we  can choose $\delta$
so that $\pm \delta$ are preperiodic points, or $\pm \delta$  non preperiodic
with $\Lambda(\pm\delta)$  well-defined and positive.
\end{proof}

We now recall the tower 
$\hat f: \hat I \to \hat I$ associated in \cite{BV}
to a $(\lambda_c, H_0)$-Collet Eckmann $S$-unimodal map $f$
satisfying the Benedicks-Carleson assumption \eqref{BeC}.
{\it As the (so-called subhyperbolic) case of a finite postcritical orbit is much simpler, we shall
assume in this construction that this orbit is infinite.}
Choose $\rho$   so that
\begin{equation}\label{2.1}
e^{\gamma}< \rho < e^{-\gamma} \sqrt {\lambda_c }\, ,
\end{equation}
and  fix 
\footnote{Our lower bound on $\beta_1$ is stronger than the
one in \cite{BV} because we use some estimates
in \cite{V}.} two constants
\begin{equation}\label{betas}
\frac{3}{2} \gamma< \beta_1 < \beta_2 < 2\gamma \, .
\end{equation}
The {\it tower} $\hat I$ is the union
$\hat I = \cup_{k \ge 0} E_k$
of levels $E_k = B_k \times \{k\}$ satisfying the following properties:
The ground floor interval $B_0=[a_0,b_0]$ is just the interval $I$.
For $k \ge 1$, the interval $B_k = [a_k, b_k]$ is such that
\begin{equation}\label{betaas}
[c_k - e^{-\beta_2 k}, c_k + e^{-\beta_2 k}] \subset B_k
\subset [c_k - e^{-\beta_1 k}, c_k + e^{-\beta_1 k}] \, .
\end{equation}
(Observe that $0=c \notin B_k$ for all $k \ge H_0$.)
Fix  $\delta > 0$ such that the Lyapunov exponents $\Lambda(\pm \delta)$
are well defined and strictly positive,
so that both claims of
Lemma~ \ref{est1a} hold (for our present choice of $\rho$),
and small enough so that
\begin{equation}\label{2.5}
|f^j(x)-c_j|<\min\{|c_j| e^{-\gamma j},\, e^{-\beta_2 j}\}
\quad
\text{ for all } 1 \le j \le  H_0 \text{ and } |x|\le\delta \, .
\end{equation}
(Just after \eqref{cass},
and later on, in  Section ~ \ref{tower'}, we may need to take a smaller
choice of $\delta$ still assuming that $\Lambda(\pm \delta)>0$
and that both claim of
Lemma~ \ref{est1a} hold.)

We may assume that the Lyapunov exponents $\Lambda(a_k)>0$
and $\Lambda(b_k)>0$ for all $k$, recalling that the set of points with a positive
Lyapunov exponent has full Lebesgue measure.
Let us write
$$
\{ 0, \pm \delta\}
\cup \{a_{j} \mid j \ge 0\} \cup \{b_{j} \mid j \ge 0\}
=\{e_0=c, e_1=\delta, e_2=-\delta, e_3, \ldots \} \, .
$$
We may  and do require additionally that 
\begin{equation}\label{condek}
f^j(e_k)\ne e_k \mbox { and  }
f^j(e_k) \ne f^i(e_\ell) \quad \forall i, j\ge 1 \, ,
k\ne \ell \ge 0 \, .
\end{equation}
(Indeed, \eqref{condek} is a co-countable set of conditions, while the set
of points $x$ with Lyapunov exponent $\Lambda(x)$  well defined
and strictly positive has full Lebesgue measure, as recalled in Section ~ \ref{sscont}.)
The positivity  condition on the Lyapunov
exponents $\Lambda(e_k)$ ($k\ne 0$)
ensures that $\alpha_{cand}(e_k)$ converges absolutely for each
$k\ge 1$, and this will be used in the proof of Theorem ~ \ref{alphabded}.

For $(x,k) \in E_k$  we set
\footnote{With respect to the definition in \cite{BV}, note that
we replaced $(-\delta, \delta)$ by $[-\delta, \delta]$, this is
not essential but convenient, e.g. in \eqref{a}.}
\begin{equation}
\label{cass}
\hat f (x,k) =\begin{cases}
(f(x),k+1) &\text{ if } k \ge 1 \text{ and } f(x) \in B_{k+1}\, ,\cr
(f(x),k+1) &\text{ if } k = 0 \text{ and } x \in [-\delta,\delta]\, , \\
(f(x),0) &\text{ otherwise.}
\end{cases}
\end{equation}
Denoting $\pi : \hat I \to I$ the projection to the first factor,
we have $f \circ \pi = \pi \circ \hat f$ on $\hat I$.

Define $H(\delta)$ to be the minimal
$k \ge 1$ such that there exists
some $x\in (-\delta,\delta)$
such that $\hat f^{k+1}(x,0) \in E_0$.
By continuity, $H(\delta)$ can be made
arbitrarily large by choosing small enough $\delta$, and we assume that $H(\delta) \ge \max(2,H_0)$.

Having defined the tower $\hat f$, we next introduce
notations $I_j$, $T_i(x)$ and $S_i(x)$ which will play a key part in the proof.
\smallskip
We decompose $(-\delta,\delta)\setminus\{0\}$ as a disjoint union of intervals
\begin{align}
\nonumber &(-\delta, \delta)\setminus \{0\}=
\cup_{j \ge  H(\delta)} I_j\, , \quad I_j:= I_{j}^+ \cup I_{j}^-
\, , \\
\label{ints} &I_{j}^\pm:=\{ |x|<\delta, \pm x > 0,
\hat f^\ell(x,0)\in E_\ell, 0 \le \ell < j,
\hat f^j(x,0)\in E_0\}\, .
\end{align}
(Note that $I_{j}^\pm$ can be empty for some $j$.)
For any $k\ge H(\delta)$ both sets $J^+_{k}:=\cup_{H(\delta)\le  j \le k} I_{j}^+$
and $J^-_{k}:= \cup_{H(\delta)\le  j \le k} I_{j}^-$ are intervals.

For each $x \in I$ we next define inductively an infinite non decreasing sequence
$$
0=S_0(x) \leq T_1(x) < S_1(x) \leq  \dots < S_i(x) \leq T_{i+1}(x) < S_{i+1}(x) \leq \dots \, , 
$$
with $S_i(x), T_i(x) \in \mathbb{N}\cup\{\infty\}$ as follows: Put $T_0(x)=S_0(x)=0$ for every $x \in I$. 
Let $i \ge 1$ and assume recursively that $S_j(x)$ 
and $T_j(x)$ have been defined for $j\le i-1$.
Then, we set (as usual, we put $\inf \emptyset = \infty$)
$$
T_i(x)= \inf \{ j\geq S_{i-1}(x) \mid \ |f^j(x)|\le  \delta \} \, .
$$
If $T_i(x)=\infty$ for some $i \ge 1$ then we set  $S_i(x)=\infty$. 
Otherwise, either $f^{T_i(x)}(x)=c$, and then we put $S_i(x)=\infty$,
or  $f^{T_i(x)}(x)\in I_j$, for some $j \ge H(\delta)$, and  we put $S_i(x)=T_i(x)+j$.

Note that if $T_i(x) < \infty$ for some $i\geq 1$ then
\begin{align*}
&\hat f^{j}(x,0) \notin E_0\, , \, \, T_{i}(x)+1\le j \le S_{i}(x)-1 \, , \\ 
&
\hat f^{\ell}(x,0) \in E_0 \, , \,  S_{i-1}(x) \le \ell \le T_{i}(x) \, .
\end{align*}
If $T_{i_0}(x)=\infty$ for  $i_0\ge 1$,  minimal
with this property, we have  $\hat f^\ell(x,0) \in E_0$ for all $\ell \ge S_{i_0-1}$
(that is, $|f^\ell(x)|> \delta$ for all $\ell \ge S_{i_0-1}$).

In other words, $T_i$ is the beginning of the $i$-th {\it bound period}
and $S_i-1$ is the end of the $i$-th bound period,
\footnote{Bound period refers to the fact that the orbit is
bound, i.e., sufficiently exponentially close, to the postcritical orbit.} and if
$S_i < T_{i+1}$ then
$S_i$ is the beginning of the $i+1$-th {\it free period}
(which ends when the $i+1$-th bound  period starts).

In order to give a meaning to some expressions below, e.g. when
$S_i=\infty$ or $T_i=\infty$, we set
$$S_i-T_i=0
\mbox { if }
S_i=T_i=\infty\, , \, \, \,\,\,\,
T_i-S_{i-1}=0 \mbox{ if } S_{i-1}=T_i=\infty\, ,
$$ 
and, for all $x \in I$, we  set
$
(f^\infty)'(x):=\infty$ and  
$f^\infty(x):=c_1$.

%%%%%%%%%%%%%%%

\subsection{\label{tower'} Properties of the tower map.}

After recalling in Proposition ~\ref{cd} and
Lemma~\ref{sizeIj} some results of \cite{BV},
we shall  state in Lemma ~\ref{expiii} expansion and distorsion control properties of the tower map $\hat f$
 (invoking  Lemma \ref{est1a}
instead of \cite[Lemma 1]{BV}). 
Then we shall  prove two new estimates (Proposition ~ \ref{ubalpha}
and its Corollary \ref{errorii}) which will play a key part in the resummation argument
of Proposition ~ \ref{abs}.

For the sake of completeness (we shall use an estimate from the proof
later on), we first recall how to obtain
distorsion bounds (see \cite{BV} or \cite[Lemma 5.3(1)]{V}):

\begin{lemma}[Bounded distortion in the bound period]\label{cd}
Let $f$ be an $S$-unimodal $(\lambda_c,H_0)$-Collet-Eckmann map satisfying the Benedicks-Carleson 
condition \eqref{BeC}, with non preperiodic critical point.
 Then, if $\delta$ is small enough, there exists $C > 0$ such that for every $j \ge 1$, 
and every $k\leq j-1$, recalling  \eqref{ints} 
\begin{equation}\label{infprodgena}
 C^{-1} \leq  \frac{ |(f^k)'(x)|}{|(f^k)'(y)|}\le C \, ,\, 
\forall  x,y \in U_j:= f\bigl (\{c\}\cup \bigcup_{m\geq j} I_m\bigr )\, .
\end{equation}
\end{lemma}

Note that $U_j$ is the  set of points in $(-\delta, \delta)\times \{0\} \subset E_0$ 
which climb at least up to level
$j-1$ before their first return to $E_0$.

\begin{proof} 
For $1 \le \ell \le k  \le j-1$, pick $x_\ell$ and $y_\ell$  in $\cup_ {m \ge j}(f^\ell(I_m))
\subset B_\ell$. We have 
\begin{align}\label{infprod}
\prod_{\ell=1}^{k}
\frac{|f'(x_\ell)|}{|f'(y_\ell)|}
\nonumber&
\le  \prod_{\ell=1}^{k}
\biggl (1+\frac{\sup  |f''|}{|f'(y_\ell)|}  |x_\ell-y_\ell|\biggr )
\le  \prod_{\ell=1}^{k}
\biggl (1+C\sup  |f''| \frac{|x_\ell-y_\ell|}{|y_\ell|}\biggr )
\\
&
\le  \prod_{\ell=1}^{\infty}
 (1+\tilde C C\sup  |f''| e^{-\beta_1 \ell} )< \infty \, ,
\end{align}
uniformly in $m\ge j$. We used that $|x_\ell-y_\ell|\le e^{-\beta_1 \ell}$ and, if
$\ell \ge H_0$, that
$|y_\ell|\geq e^{-\gamma \ell}-e^{-\beta_1 \ell}$
with $\beta_1 > 3 \gamma/2$, but any summable condition would be enough here. If we choose  $y_\ell=f^{\ell-1}(y)$ and $x_\ell=f^{\ell-1}(x)$ we get the upper bound in (\ref{infprod}). If
we pick $y_\ell=f^{\ell-1}(x)$ and $x_\ell=f^{\ell-1}(y)$ then we obtain the lower bound in 
(\ref{infprod}). 
\end{proof}

The following upper and lower bounds from \cite{BV}, about 
points which climb for exactly $j-1$ steps, will be used several times:

\begin{lemma}[The $j$-bound intervals $I_j^\pm$]\label{sizeIj}
Let $f$ be an $S$-unimodal $(\lambda_c,H_0)$-Collet-Eckmann map satisfying the Benedicks-Carleson
condition \eqref{BeC} and with non preperiodic critical point.
Then there exist $C$ 
and $C_2$ so that for any $j \ge H(\delta)$, recalling \eqref{ints},  we have
\begin{equation}\label{upbdx}
 |x-c |\le  C 
e^{\frac{-3\gamma (j-1)}{4} } | (f^{j-2})'(c_1)|^{-1/2}\, , \, \, 
\forall x\in  I_j\, ,
\end{equation}
and
\begin{equation}\label{expii} 
 |(f^j)'(x)|\geq C_2 e^{-\frac{\beta_2 j}{2}} |(f^{j-1})'(c_1)|^{1/2}
\, , \, \forall
x \in I_j \, ,
\end{equation}
and, finally,
\begin{equation}\label{eight'} |f'(x)|\ge C^{-1} e^{-\gamma j} |(f^{j-1})'(c_1)|^{-1/2}\, ,
\, \forall
x \in I_j \, .
\end{equation}
\end{lemma}

\begin{proof}
If $I_j$ is empty, there is nothing to prove. Otherwise,
our definitions and the mean value theorem
imply that there exists $y$ with
$f(y)\in [f(x), c_1]$ so that
$$
|(f^{j-2})'(f(y))| |f(x)-c_1| \le C e^{-\beta_1 (j-1)}\, .
$$ 
Therefore, since $\beta_1 > 3\gamma/2$ (recall \eqref{betas})
 the lower bound in  \eqref{infprodgena} and the fact that
$|f(x)-c_1|\ge C^{-1} |x-c|^2$ yield \eqref{upbdx}.

The bound  (\ref{expii}) follows from \cite[(3.10), Proof of Lemma 2]{BV}
(see  top of p. 495 there, or see \cite[Lemma 5.3, eq. after (5.15)]{V}).

For \eqref{eight'},  use \eqref{expii} and
that $\beta_2<2\gamma$ from \eqref{betas}.
\end{proof}

Recall the times $S_i$, $T_i$ from Subsection~\ref{tower}, for
suitably small $\delta$.
The following lemma gives expansion at the end 
of the free period $T_i-1$ (just before climbing 
the tower), at the end $S_i-1$  of the
bound period (after falling from the tower), and  during the free period
(when staying at level zero):

\begin{lemma}[Tower expansion for Benedicks-Carleson maps]\label{expiii}
Let $f$ be an $S$-unimodal $(\lambda_c,H_0)$-Collet-Eckmann map satisfying the Benedicks-Carleson condition \eqref{BeC}, with non preperiodic critical
point, and let $\rho$ satisfy \eqref{2.1}.
For every small enough $\delta_0>0$,  if $\delta < \delta_0$,   
$\sigma >1$,  $C_1=C_1(\rho)\in(0, 1]$,  and $c(\delta)>0$ are
as in Lemma ~\ref{est1a}, letting
$S_i(x)$ and $T_i(x)$ be the times associated to the tower for
$\delta$,  then
\begin{equation}\label{star'}
|(f^{S_i(x)})'(x)| \geq  \rho^{S_i(x)}\, , \qquad  |(f^{T_i(x)})'(x)| \geq  C_1 \rho^{T_i(x)} 
\, , \quad \forall x \in I \, , \quad \forall i \ge 0\, ,
\end{equation}
and 
$$
|(f^{S_i(x)+j})'(x)| \geq  c(\delta) \rho^{S_i(x)}\sigma^j\, , 
\qquad \forall x \in I \, ,\qquad 
\forall i \ge 0 \, , \, \forall 0 \le  j < T_{i+1}(x)-S_i(x) \, .
$$
\end{lemma}

\begin{remark} \label{liminf}
An immediate consequence of Lemma \ref{expiii} is that,  for every $x$ such that 
$f^n(x)\neq c$, for every $n$, we have
$\limsup_n |(f^n)'(x)|^{1/n}\geq \xi >1$, 
where $\xi=\min (\rho,\sigma)$. 
\end{remark}

\begin{proof}  [Proof of Lemma \ref{expiii}]
The lemma will easily follow from Lemma ~ \ref{est1a} and 
\eqref{expii} from
Lemma~\ref{sizeIj}.

Choose $\delta< \delta_0$ as in the second claim of 
Lemma \ref{est1a}, small enough so  that 
$$
C_1 C_2 \cdot  e^{-\frac{\beta_2 j}{2}} \lambda_c^{\frac{j-1}{2}} \geq \rho^{j} \, ,
\quad \forall j \ge H(\delta) \, .
$$
Let now $x \in I$. Recall that for any $\ell\ge 1$, the definitions imply
$f^{S_{\ell-1}(x)+k}(x) \in I\setminus [-\delta,\delta]$
for all  $0\le k < T_\ell(x) - S_{\ell-1}(x)$ and
$f^{T_\ell(x)}(x)\in I_j$ with  $j=S_\ell(x)-T_\ell(x)\ge H(\delta)$.
Therefore, the second claim of   Lemma  \ref{est1a} and 
 (\ref{expii}) give for all $i\ge 0$
$$
|(f^{S_i})'(x)|= \prod_{\ell=1}^{i} |(f^{S_\ell(x)-T_{\ell}(x)})'(f^{T_\ell(x)}x)| |(f^{T_\ell(x)-S_{\ell-1}(x)})'(f^{S_{\ell-1}(x)}x)|\geq  \rho^{S_i(x)} \, ,
$$
and 
$$
|(f^{T_i})'(x)|= |(f^{T_i(x)-S_{i-1}(x)})'(f^{S_{i-1}(x)}x)|
|(f^{S_i})'(x)| \geq C_1  \rho^{T_i(x)-S_i(x)}\rho^{S_i(x)} \, .
$$
Using in addition the first claim of   Lemma  \ref{est1a}, we get, for $0 \le  j \le T_{i+1}(x)-S_i(x)$, 
$$
|(f^{S_i(x)+j})'(x)|=| (f^j)'(f^{S_i(x)}(x))||(f^{S_i})'(x)|
\geq c(\delta) \sigma^j  \rho^{S_i(x)} \, .
$$
\end{proof}

The information on the tower will allow us to prove the next proposition, which is a crucial ingredient
\footnote{Proposition \ref{ubalpha} will also be
used in an essential way in the proof of Theorem ~ \ref{linresp},
in particular
in the proof of Lemma~\ref{rootsing}, and also in
Lemma~\ref{truncspec}.} to show that $\alpha_{cand}$ can be resummed to a bounded
function (Proposition~\ref{abs} and Definition~\ref{resumdef}):

\begin{proposition}[Key estimate for Benedicks Carleson maps]\label{ubalpha}
Let $f$ be an $S$-unimodal $(\lambda_c,H_0)$-Collet-Eckmann map satisfying the Benedicks-Carleson
condition \eqref{BeC},
with non preperiodic critical point. There exists $C> 0$ such that for every $j \geq 0$ we have
 \begin{equation}\label{serieb}
  \sum_{k =j+1}^\infty \frac{1}{|(f^{k-j})'(f^j(c_1))|} \leq C e^{\gamma j} \, .
\end{equation}
\end{proposition}

The proof shows that $C=O((c(\delta))^{-1})$, where $\delta$ is the parameter used in the
construction of the tower and $c(\delta)$ is given by
Lemma~\ref{est1a}.
More importantly,  the proposition implies that $|\alpha_{cand}(c_j)|\le C \sup |v|
e^{\gamma j}$. This bound
is of course not uniform in $j$, but it will act as a bootstrap 
for the proof of the  Proposition~\ref{abs} which performs the resummation.

\begin{proof}
Fix $j\ge 1$. Since the coefficients of the series are all positive, we may (and shall) group them
in a convenient way, using the
times $T_i:=T_i(c_{j+1})$  and $S_i:=S_i(c_{j+1})$ defined in the tower construction
for a small enough $\delta$.
We have
\begin{align*}
 &\sum_{k =j+1}^\infty \frac{1}{|(f^{k-j})'(f^j(c_1))|}\\
&\quad =\sum_{i=0}^{\infty} \frac{1}{|(f^{S_i})'(c_{j+1})|} u_{T_{i+1}-S_i}(f^{S_i}(c_{j+1})) +\sum_{i=1}^{\infty} \frac{1}{|(f^{T_i})'(c_{j+1})|} u_{S_{i}-T_i}(f^{T_i}(c_{j+1})) \, ,
\end{align*}
where we use the notation
$$
u_n(y)= \sum_{\ell=1}^{n} \frac{1}{|(f^{\ell})'(y)|} \, .
$$
(In particular $u_0\equiv 0$.)
Since $T_{i+1}-S_i=T_1(f^{S_i}(c_{j+1}))$,  Lemma \ref{expiii} implies 
$$
u_{T_{i+1}-S_i}(f^{S_i}(c_{j+1}))\leq \frac{C }{c(\delta)(1-\sigma^{-1})} \, ,
$$
(in particular the series converges if  $n=T_{i+1}-S_i=\infty$).
Since $f''(c)\ne 0$, the Benedicks-Carleson assumption \eqref{BeC}
implies for all $i$
$$|f'(f^{T_i}(c_{j+1}))|= |f'(f^{T_i+j}(c_1))|\geq  C^{-1} e^{-\gamma(T_i+j)}\, .
$$
Therefore, the   bounded distortion estimate \eqref{infprod} in the proof of Lemma ~\ref{cd} gives, together with the Collet-Eckmann assumption,
\footnote{The constant $C$ above depends on
$[\sup_{1\le j < H_0} \lambda_c/| (f^j)'(c_1)|^{1/j}]^{H_0}$.
By Lemma~\ref{est1aunif}, this expression
is uniform for suitable families $f_t$.}
$$
u_{S_{i}-T_i}(f^{T_i}(c_{j+1}))\leq \frac{1}{|f'(f^{T_i}(c_{j+1}))|} \sum_{\ell=0}^{\infty} \frac{C}{|(f^\ell)'(c_1)|}\leq \frac{C^3 }{ (1-\lambda_c^{-1})}e^{\gamma(T_i+j)} \, . 
$$

By Lemma \ref{expiii} we have 
$|(f^{S_i})'(c_{j+1})|\geq C\rho^{S_i}$
and
$|(f^{T_i})'(c_{j+1})|\geq C_1\rho^{T_i}$.
Therefore, there exists constants $K_1(\delta)$, $K_2(\delta)$ so that
$$ 
 \sum_{k =j+1}^\infty \frac{1}{|(f^{k-j})'(f^j(c_1))|}
\leq K_1(\delta) e^{\gamma j} \biggl [  \sum_{i=0}^{\infty} \rho^{-S_i}   +  \sum_{i=1}^{\infty} \rho^{-T_i}e^{\gamma T_i}    \biggl ]\leq K_2(\delta) e^{\gamma j} \, .
$$
(We used that $\rho > e^{\gamma}$.)
\end{proof}

We end this section of preparations by a consequence
of Proposition ~ \ref{ubalpha} which will also be needed
in our resummation Proposition ~\ref{abs}:

\begin{corollary}\label{errorii} Let $f$ be an $S$-unimodal $(\lambda_c,H_0)$-Collet-Eckmann map satisfying the Benedicks-Carleson
condition \eqref{BeC},
with non preperiodic critical point.
Then there exists $C$ so that for any $j\ge 1$
\begin{equation} 
\frac{1}{|f'(x)|} \sum_{k=1}^{j-1}  \biggl | \frac{1}{|(f^k)'(c_1)|}- \frac{1}{|(f^k)'(f(x))|} \biggr | \leq  C\frac{e^{5\gamma j/4}}{|f^{j-2}(c_1)|^{1/2}}\, , \, \forall x \in I_j\, .
\end{equation}
\end{corollary}

%(If $f$ enjoys in fact (SeR), the proof shows that $C$ diverges as $\gamma\to \zero$.)

\begin{proof} 
For any $j\ge 1$, $x \in I_j$, and $1\le k < j$ we get by 
using \eqref{infprod} from the proof of Lemma ~\ref{cd} that
\begin{align*}
& \biggl |\frac{1}{(f^{k})'(f(x))}- \frac{1}{(f^{k})'(c_1)}\biggr |\\
&\quad \le  \sum_{n=0}^k \biggl |\frac{1}{f'(f^n(f(x)))}-\frac{1}{f'(f^n(c_1))} \biggr |\prod_{\ell=0}^{n-1}
\frac{1}{|f'(f^\ell(c_1))|} \prod_{i=n+1}^k \frac{1}{|f'(f^i(f(x)))|}  \\
%& \quad \le  
%\sum_{n=0}^k \sup_{y \in f^n[f(x),c_1]}
%\frac{|f''(y)|}{|f'(y)|^2}  \, \,  | f^n(f(x))- f^n(c_1)  | \prod_{\ell=0}^{n-1}
%\frac{1}{|f'(f^\ell(c_1)|}
%\prod_{i=n+1}^k \frac{1}{|f'(f^i(f(x)))|} \\
 &\quad \le    \sum_{n=0}^k \sup_{y \in f^n[f(x),c_1]}
\frac{|f''(y)|}{|f'(y)|^2}  \, \,   \sup_{z\in [f(x), c_1]} | (f^n)'(z)  | |f(x)-c_1| \\
&\qquad\qquad\qquad\qquad\qquad\qquad 
\cdot \prod_{\ell=0}^{n-1} \frac{1}{|f'(f^\ell(c_1)|}
\prod_{i=n+1}^k \frac{1}{|f'(f^i(f(x)))|} \\
 &\quad \overset{(\ref{infprod})}{\leq}  C |f(x)-c_1|\sum_{n=0}^k \sup_{y \in f^n[f(x),c_1]}
\frac{|f''(y)|}{|f'(y)|}  \, \,   \frac{| (f^n)'(c_1)  |}{|(f^k)'(c_1) |} \\
 &\quad  
 %\overset{(\ref{BeC})}
 {\leq} \frac{C}{1- e^{-\gamma/2}} |f(x)-c_1|\sum_{n=0}^k e^{\gamma n}  \, \,   \frac{1}{|(f^{k-n})'(f^n(c_1)) |}  
\, .
\end{align*}
(In the last inequality, we used that $f^n[f(x), c_1]\subset B_{n+1}$,
and that  the Benedicks-Carleson
assumption \eqref{BeC} implies $|y-c| >  e^{-\gamma(n+1)}-e^{-\beta_1(n+1)}$ 
for $y\in B_{n+1}$, together with \eqref{betas}.)

Therefore, for any $j\ge 1$ and $x \in I_j$, since $|f(x)-c_1|\le C |x-c|^2$,  
Proposition  ~ \ref{ubalpha} 
implies
\begin{align}\label{lastt'}
 &\frac{1}{|f'(x)|}
\sum_{k=0}^{j-1}
 \biggl |\frac{1}{(f^{k})'(f(x))}- \frac{1}{(f^{k})'(c_1)}\biggr | \\
%\nonumber 
%& \qquad \le  C
%\frac{|x-c|^2}{|f'(x)|} \sum_{k=0}^{j-1}
%\sum_{n=0}^k e^{\gamma n}  \, \,   \frac{1}{|(f^{k-n})'(f^n(c_1)) |} \\
\nonumber 
& \qquad \le  C|x-c|  
\sum_{k=0}^{j-1} \sum_{n=0}^{k} e^{\gamma n}  \, \,   \frac{1}{|(f^{k-n})'(f^n(c_1)) |} \\
\nonumber 
& \qquad \le  C|x-c| \sum_{n=0}^{j-1} e^{\gamma n} \sum_{i=0}^{j-n-1}  \, \,   \frac{1}{|(f^{i})'(f^n(c_1)) |} \\
\nonumber 
 &\qquad \le  C|x-c| \sum_{n=0}^{j-1} e^{\gamma n} \sum_{i=0}^{\infty}  \, \,   \frac{1}{|(f^{i})'(f^n(c_1)) |} \\
\nonumber 
 &\qquad \overset{(\ref{serieb})}{\leq}  C|x-c| \sum_{n=0}^{j-1} e^{2\gamma n} 
  \le  C|x-c| e^{2\gamma j}   
\le  C\frac{e^{-3\gamma (j-1)/4}}{|(f^{j-2})'(c_1)|^{1/2}}e^{2\gamma j}
\, , \end{align}
where we used  \eqref{upbdx} from Lemma ~\ref{sizeIj}  in the last inequality.
\end{proof}

%%%%%%%%%%%%%%%%
\subsection{\label{resume} Resummation: Definition and boundedness of $\alpha$ 
for horizontal
$v$.}

Proposition~\ref{abs} is the heart of the proof of Theorem~\ref{alphabded}. This is where we define the dynamical resummation
for the series $\alpha_{cand}$, under a horizontality condition.

\begin{proposition}[Resummation] \label{abs}Let $f$ be an $S$-unimodal $(\lambda_c,H_0)$-Co\-llet-Eck\-mann map satisfying the Be\-ne\-dicks-Carleson
condition \eqref{BeC},
with non preperiodic critical point.
Let $v=X\circ f$, for $X$ a Lipschitz function, and assume
that $v$ satisfies the 
horizontality condition \eqref{hor} for $f$.

If $\delta$ is 
small enough,
then
for every $x \in I$, letting $T_i=T_i(x)$ and $S_i=S_i(x)$  be the times associated
to the tower for $\delta$, 
the following series  converges:
\begin{equation} 
\label{absser} 
\sum_{i=1}^{\infty} \biggl (  \frac{1}{|(f^{T_i})'(x)|} |w_{S_{i}-T_i}(f^{T_i}(x))|+ \frac{1}{|(f^{S_{i-1}})'(x)|} |w_{T_{i}-S_{i-1}}(f^{S_{i-1}}(x))| \biggr )\, ,
\end{equation}
where $w_\infty(c)=0$ and
\footnote{In particular, we claim that $w_n(y)$ converges if $n=S_i-T_i=\infty$ and $y=f^{T_i}(x)\ne c$,
or if $n=T_{i}-S_{i-1}=\infty$ and $y=f^{S_{i-1}}(x)$.}
$$
w_n(y):= \sum_{\ell=0}^{n-1} \frac{v(f^\ell(y))}{(f^{\ell+1})'(y)}\, ,\quad
y \ne c \, , 
n \in \integer_+ \cup \{\infty\} \, .
$$
Moreover the sum of the series \eqref{absser} is bounded uniformly in  $x\in I$.
\end{proposition}

The proposition allows us to give the following definition:
\begin{definition}\label{resumdef}
We  define $\alpha(x)$ for any $x \in I$ by 
\begin{equation}\label{alpdef}\alpha(x)= - \sum_{i=1}^{\infty} \biggl ( \frac{1}{(f^{T_i})'(x)} w_{S_{i}-T_i}(f^{T_i}(x))+ \frac{1}{(f^{S_{i-1}})'(x)} w_{T_{i}-S_{i-1}}(f^{S_{i-1}}(x)) \biggr ) \, .
\end{equation}
\end{definition}

If the formal
series \eqref{defa} is absolutely convergent at $x$, then \eqref{alpdef}
is just the sum of this series.
If $f^j(x)=c$ for some $j\ge 0$, minimal with
this property, then our notation ensures that  \eqref{alpdef} is just the finite sum 
\eqref{defa'}. In both these cases, $\alpha(x)=\alpha_{cand}(x)$.

\begin{proof}[Proof of Proposition \ref{abs}.] 
Choose $\delta$ small enough, as in Lemma \ref{expiii},
and $C_1(\rho)$,
$\sigma> 1$, and $c(\delta)$ from    Lemma \ref{est1a}.

For any $i \ge 1$ so that $S_{i-1}< \infty$,  since
 $T_{i}-S_{i-1}=T_1(f^{S_{i-1}}(x))$,   Lemma \ref{expiii} 
implies
 \begin{align}
\nonumber |w_{T_i-S_{i-1}}(f^{S_{i-1}}(x))|
&\leq \sum_{\ell=0}^{T_i-S_{i-1}-1} \biggl | \frac{v(f^\ell(f^{S_{i-1}}(x)))}{(f^{\ell+1})'(f^{S_{i-1}}(x))}\biggr |\\
\label{first'}&\le \frac{C }{c(\delta)(1-\sigma^{-1})} \sup |v|\, , \, \, 
 \forall  i\ge 0 \, .
 \end{align}
(Note that $T_i=\infty$ is allowed in the previous estimate.)

If $S_i=\infty$ then either $f^{T_i}(x)=c$
or $T_i=\infty$, in both cases
$w_{S_i-T_i}(f^{T_i}(x))=0$. 
Next, we  claim that whenever $S_i\ne \infty$, we have,
\begin{equation}\label{est2beq} |w_{S_i-T_i}(f^{T_i}(x))|
\le  C \max \{ \sup |v|, \Lip X\}  
\frac{e^{2\gamma (S_i-T_i)}}{| (f^{S_i-T_i-1})'(c_1)|^{1/2}}
\, ,
\forall i \ge 0 \, .
\end{equation}
We shall prove \eqref{est2beq}, which requires  horizontality as well as
the key Proposition \ref{ubalpha}  and its Corollary~ \ref{errorii},
at the end of this proof.

Putting together \eqref{first'} and \eqref{est2beq}, and recalling
Lemma \ref{expiii}, we find the following upper bound
for (\ref{absser}) :
$$
\sum_{i=1}^\infty
\frac{C \rho^{-S_{i-1}}}{ c(\delta)(1-\sigma^{-1})} \sup |v|
+  C \frac{\rho^{-T_{i}}}{C_1(\rho)} \max (\sup |v|, \Lip X )  \frac{e^{2\gamma (S_i-T_i)}}{| (f^{S_i-T_i-1})'(c_1)|^{1/2}}
\, ,
$$
Using again \eqref{BeC}, we are done.

It remains to prove \eqref{est2beq}.
The definitions imply   $f^{T_i}(x)\in I_j$ for  $j=S_i-T_i$, for all $i\ge 1$. So it suffices to show that
\begin{equation}\label{est2beq'} |w_{j}(y)|
\le  C \max ( \sup |v|, \Lip X) 
\frac{e^{2\gamma j}}{| (f^{j-1})'(c_1)|^{1/2}}\, , \quad
\forall y \in I_j \, .
\end{equation}
We shall use the decomposition
\begin{align}
\nonumber |w_j(y)|
\leq& \underbrace{ \frac{1}{|f'(y)|}
\biggr | \sum_{k=0} ^{j-1} \frac{v(c_k)}{(f^k)'(c_1)} \biggl |}_{I}
%\\+&
+ \underbrace{\frac{1}{|f'(y)|} \sum_{k=0}^{j-1} |v(c_k)|\biggl  |\frac{1}{(f^{k})'(f(y))}- \frac{1}{(f^{k})'(c_1)}\biggr |}_{II}\\
\label{I-II-III}&\qquad\qquad\qquad + \underbrace{\frac{1}{|f'(y)|}
\sum_{k=0}^{j-1}
\biggl  |\frac{v(f^{k}(y))-v(c_{k})}{(f^{k})'(y)} \biggr |}_{III} \, .
\end{align} 
We first consider $I$.
By the horizontality condition \eqref{hor} for $v$, we have
$$ 
\sum_{k=0}^{j-1} \frac{v(c_k)}{(f^k)'(c_1)} + \frac{1}{(f^j)'(c_1)} \sum_{k=j}^\infty \frac{v(c_k)}{(f^{k-j})'(c_{j+1})}=0 \,    .
$$
Therefore, \eqref{eight'} in  Lemma~\ref{sizeIj}
and  Proposition \ref{ubalpha} imply
\begin{align}
\label{discrep} \nonumber
 I&=
\frac{1}{|f'(y)|}\frac{1}{|(f^{j})'(c_1)|} 
\biggl |  \sum_{k=j}^\infty \frac{v(c_k)}{(f^{k-j})'(c_{j+1})}  \biggr |
\\
\nonumber 
&\leq C 
 e^{\gamma j } | (f^{j-1})'(c_1)|^{1/2}\frac{1}{| (f^{j})'(c_1)|} \sup |v| e^{\gamma j}
\\&
\le C \sup |v| \frac{e^{2\gamma j}}{| (f^{j-1})'(c_1)|^{1/2}} \, .
 \end{align}

Next, by Corollary \ref{errorii} 
we find
\begin{align}\label{lastt}
\nonumber II & \le \sup |v| \frac{1}{|f'(y)|}
\sum_{k=0}^{j-1}
\biggl |\frac{1}{(f^{k})'(f(y))}- \frac{1}{(f^{k})'(c_1)} \biggr | \\
 & \leq   C \sup |v| \frac{e^{5\gamma j/4}}{|(f^{j-2})'(c_1)|^{1/2}} \, .
\end{align}

Recalling our assumption $v=X\circ f$, we consider now
\begin{equation}
III=\frac{1}{|f'(y)|}
\sum_{k=0}^{j-1}
\biggl  |\frac{X(f^{k+1}(y))-X(c_{k+1})}{(f^{k})'(y)} \biggr |\, .
\end{equation}
Since $X$ is Lipschitz and $f$ is $C^1$,  for any $0\le k \le j-1$
there
exists $z$ between $c_1$ and $f(y)$ so that
\begin{align}\label{mvt}
|X(f^{k+1}(y))-X(f^{k+1}(c))|&
\le \Lip(X)| f^{k}(f(y))-f^{k}(c_1)|
\\ \nonumber 
&\le \Lip(X) | f(y)-c_1| |(f^{k})'(z)|\, .
\end{align}
Then \eqref{mvt}, together with   \eqref{infprodgena} from the proof
of Lemma ~\ref{cd} (recall $k\le j-1$ and $y\in I_j$), imply 
\begin{align*}
 \frac{|X(f^{k+1}(y))-X(f^{k+1}(c))|}{|(f^{k})'(c_1)|}
&\le \Lip X |f(y)-c_1|\sup_{z \in [f(y),c_1]} \frac{| (f^{k})'(z)|}
{(f^k)'(c_1)}\\
&\le C  \Lip X  |f(y)-c_1| \\
&\le C  \Lip X |y-c| \sup_{z \in [c,y]} |f'(z)|\, .
\end{align*}
Since $|f'(y)|\ge C |y-c|$ and $|f'(z)|\le C |z-c|$,
the bound \eqref{upbdx} in Lemma~\ref{sizeIj} gives
\begin{align}\label{needv'zero}
\nonumber III & \le C \frac{j}{|f'(y)|}  \Lip X |y-c|  \sup_{z \in [0,y]}|f'(z)|
 \le C j \Lip X   \sup_{z \in [0,y]} |z-c| \\
&  \le C j  \Lip X
 e^{-\frac{3 \gamma (j-1)}{4} } | (f^{j-2})'(c_1)|^{-1/2}
  \le C  \Lip X\frac{1}{| (f^{j-2})'(c_1)|^{1/2} }
 \, .
\end{align}
Putting  (\ref{discrep}),
(\ref{lastt}),  and (\ref{needv'zero}) together, we get (\ref{est2beq'}).
\end{proof}

%%%%%%%
\subsection{\label{alphac} Proof of Theorem ~ \ref{alphabded}: 
Continuity of $\alpha$ and checking the TCE}

We prove that Proposition ~\ref{abs} implies Theorem ~ \ref{alphabded}:

\begin{proof}[Proof of Theorem ~ \ref{alphabded}]
The so-called subhyperbolic case when the critical point is preperiodic  is  
easier and left to the reader.
For small enough $\delta$ (recall Sections ~ \ref{tower} and \ref{tower'}),
we construct a tower map
and associated times $T_i(x)$ and $S_i(x)$.

To show the uniqueness statement for bounded $v$, suppose that 
$\beta: I \to \complex$ is a bounded  function 
such that $v= \beta\circ f - f' \cdot \beta$ on $I$.
It is easy to see that for every $x$ and $n\ge 1$ such that 
$(f^n)'(x)\neq 0$ we have 
\begin{equation}\label{ite}
 \beta(x)= -\sum_{i=0}^{n-1} \frac{v(f^i(x))}{(f^{i+1})'(x)} 
 + \frac{\beta(f^n(x))}{(f^{n})'(x)} \, .
\end{equation}
If $f^n(x)\neq c$ for every $n$,  Remark \ref{liminf} implies that 
$
\limsup_n |(f^n)' (x)| = \infty 
$, so that there exists\footnote{Note that if $T_i(x)<\infty$
for all $i$, we can take $n_i(x)=T_i(x)$.}
$n_i(x)\rightarrow_i \infty$ with $\lim_k |(f^{n_i(x)})'(x)|=\infty$.
Since $\beta$ is bounded, it 
follows from (\ref{ite}) that
\begin{equation}\label{form2}
 \beta(x)=-\lim_{i\to \infty} \sum_{j=0}^{n_i} \frac{v(f^j(x))}{(f^{j+1})'(x)} \, .
\end{equation}
This proves that $\beta$ is uniquely defined on $\{ x \mid f^n (x)\ne c \, , \forall
n\}$. In particular, $\beta(c_1)=\alpha_{cand}(c_1)$, so that $v$ is horizontal
(using the TCE, that $f'(c)=0$, and that $\beta$ is bounded).

Now, if $f^i(x)\neq c$ for  $0\le i < n$ and $f^n(x)=c$,  then 
\eqref{ite} and $\beta(c)=0$ give
\begin{equation}\label{form1} 
\beta(x)= -\sum_{i=0}^{n-1} \frac{v(f^i(x))}{(f^{i+1})'(x)} \, .
\end{equation}
Therefore, $\beta(x)=\alpha_{cand}(x)$. This ends the proof of uniqueness.

\smallskip
From now on, we assume that $v=X \circ f$
with $X$ Lipschitz.
If $v$ is horizontal,  Proposition ~\ref{abs}
implies that the function $\alpha(x)$ defined by the series  (\ref{alpdef})
is bounded uniformly in $x \in I$. It remains to show that
$\alpha$ is continuous and satisfies the TCE.

The definitions easily imply 
that for every  $x \in I$ and all $i \ge 1$ the following limits exist:
\begin{align*}
&S_i^{+}(x)=\lim_{y\rightarrow x^+} S_i(y)\, ,
\quad S_i^{-}(x)=\lim_{y\rightarrow x^-} S_i(y)\, , 
\\ &
T_i^{+}(x)=\lim_{y\rightarrow x^+} T_i(y)\, ,
\quad  T_i^{-}(x)=\lim_{y\rightarrow x^-} T_i(y) \, .
\end{align*}

Writing $T_i^\pm$ and $S^\pm_i$ for $T_i^\pm(x)$ and $S^\pm_i(x)$, define
$$
\alpha^\pm(x):=-\sum_{i=1}^{\infty}  
\biggl ( \frac{1}{(f^{T_i^\pm})'(x)} w_{S_{i}^\pm-T_i^\pm}(f^{T_i^\pm}(x)) + 
\frac{1}{(f^{S_{i-1}^\pm})'(x)} w_{T_{i}^\pm-S_{i-1}^\pm}(f^{S_{i-1}^\pm}(x))\biggl )\, .
$$
We claim that for every $x \in I$ we have 
\begin{equation}\label{claimm}
 \lim_{y\rightarrow x^+}\alpha(y)= \alpha^+(x) \, ,
\mbox{ and } \lim_{y\rightarrow x^-}\alpha(y)= \alpha^-(x) \, .
\end{equation}
We shall  show \eqref{claimm} at the end of the proof of this theorem.

Let now $\SS$ be the set of  $x\in I$  so that there exists $\ell \ge 0$
with $\hat f^\ell(x,0) \in \partial E_k$ for some $k\ge 1$, or
$\hat f^\ell(x,0) =(\pm \delta,0)$, or $\hat f^\ell(x,0)=(c, 0)$. 
Clearly, if $x \not\in \SS$,  then $S_i(x)=S_i^+(x)=S_i^-(x)$ and 
$T_i(x)=T_i^+(x)=T_i^-(x)$, for every $i \ge 1$. Consequently 
$\alpha$ is continuous at $x \notin \SS$. 
If $x \in \SS$ but $\hat f^\ell(x,0)\ne (c, 0)$ for
all $\ell \ge 0$, then  the conditions on
$\{e_k\}$ in Section ~\ref{tower} imply that the series \eqref{defa}
converges absolutely at $x$.  If $\hat f^\ell(x,0)= (c, 0)$ for
some $\ell \ge 0$, then $\alpha(x)$ is the finite sum \eqref{defa'}.
Let now $x \in \SS$. The three series, or finite sums,  which 
define $\alpha(x)$, 
$\alpha^+(x)$, and $\alpha^-(x)$ are obtained by grouping together
in different ways the terms of the absolutely convergent series
\eqref{defa}, or of the sum
\eqref{defa'}. Therefore, 
$\alpha(x)=\alpha^+(x)=\alpha^-(x)$, and \eqref{claimm} implies that
$\alpha$ is continuous at $x$.

\smallskip
To show that $\alpha$ satisfies the twisted cohomological equation,
note that if $x$ is a repelling periodic point then $f^\ell(x)\ne c$ for all
$\ell$, and
the series  \eqref{defa} 
is absolutely convergent at $x$. Therefore, this series coincides with $\alpha(x)$.  
In particular one can easily check that $v(x)= \alpha(f(x))-f'(x)\alpha(x)$
for  repelling periodic points $x$. Since all periodic points of a Collet-Eckmann map are
repelling, since  the set of periodic points is dense,  and
since $\alpha$ is continuous, it follows that $\alpha$  satisfies 
the twisted cohomological equation everywhere.

The fact that $\alpha(x)=\alpha_{cand}(x)$ for all
$x$ so that $f^j(x)=c$ or such that the series $\alpha_{cand}(x)$ converges absolutely
follows from the remark after Definition ~ \ref{resumdef}.

\medskip
It remains to prove \eqref{claimm}.
We shall consider the limit as $y$ approaches $x$ from above (the proof of the other 
one-sided limit  is identical).  

Before we start, note that for any $\epsilon>0$, the uniform
constants and exponential rates in the proof of Proposition~\ref{abs}
(see \eqref{first'}, where $T_i=\infty$ is allowed, and \eqref{est2beq},
\eqref{star'})
imply that there is $n_0=n_0(\epsilon)\ge 1$ such that for all $x \in I$
\begin{equation}\label{tail}
 \sum_{i\colon T_i  \ge n_0 } \frac{1}{|(f^{T_i})'(x)|}| w_{S_{i}-T_i}(f^{T_i}(x))|
+\sum_{i\colon S_i \ge  n_0} \frac{1}{|(f^{S_i})'(x)} |w_{T_{i+1}-S_i}(f^{S_i}(x))|
 < \frac{\epsilon}{4}\, .
\end{equation}

Fix $x \in I$, and let
$T_i=T_i(x)$, $S_i=S_i(x)$, $T_i^+=T_i^+(x)$, $S_i^+=S_i^+(x)$. 
There are three cases to consider to prove \eqref{claimm} when $y \downarrow x$.
The first case occurs  when  $T_i < \infty$ and  $S_i < \infty$ for every $i$. 
This means that the forward $\hat f$-orbit of $x$  never hits the critical point $(c,0)$
and never gets trapped inside the base $E_0$. Then observe that there exists  $\epsilon_1>0$ 
such that if $x< y< x+\epsilon_1$ then, for every $i$ so that $S_i^+ < n_0$, we have  
$S_i(y)=S_i^+$, and  for every $i$ so that $T_i^+(x) < n_0$, we have $T_i(y)=T_i^+$. 
Clearly, for any $n_0 \ge 1$, the function
\begin{align*}
\tilde{\alpha}(y)=\tilde \alpha_{n_0}(y)&= -\sum_{i\colon T_i^+< n_0}
 \frac{1}{(f^{T_i^+})'(y)} w_{S_{i}^+-T_i^+}(f^{T_i^+}(y))
\\ 
&\qquad \qquad\qquad 
-\sum_{i\colon S_i^+ < n_0} \frac{1}{(f^{S_i^+})'(y)} 
w_{T_{i+1}^+(x)-S_i^+}(f^{S_i^+}(y))
\end{align*}
is continuous on $[x,x+\epsilon_1)$.
So for any $\epsilon >0$ there exists $0<\epsilon_2 < \epsilon_1$ (depending also on $n_0$) such that 
if $x\leq y< x+\epsilon_2$ then 
$|\tilde{\alpha}(y)-\tilde{\alpha}(x)|< \epsilon/2$.  
Clearly, $\tilde \alpha(x)$ is just the $n_0$-truncation of $\alpha^+(x)$ 
while the observation above implies that
$\tilde\alpha(y)$ is the $n_0$-truncation of $\alpha(y)$.
Thus, taking 
$n_0(\epsilon)$ as in  \eqref{tail}, we get   that
$$
|\alpha(y)-\alpha^+(x)|< |\tilde{\alpha}(y)-\tilde{\alpha}(x)| + \frac{2\epsilon}{4} < \epsilon\, ,\quad
\forall y \in [x, x+\epsilon_2)\, .
$$

The second case occurs when the forward $\hat f$-orbit of $x$ gets trapped in the first
level $E_0$. That is, there  exists  $i_0\ge 0$  such that $S_{i_0}< \infty$ but 
$T_{i_0+1}(x)=\infty$. Then, for any $n_0 \ge 1$, there exists  $\epsilon_1>0$ such that if $x< y < x+ \epsilon_1$ 
then    $S_i(y)=S_i^+$ and 
$T_i(y)=T_i^+$ for every $i \leq  i_0$, and $T_{i_0+1}(y)\geq n_0$. 
Clearly, the function
\begin{align*}&
\tilde{\alpha}(y)= -\sum_{i <i_0} \frac{1}{(f^{S_i^+(x)})'(y)} w_{T_{i+1}^+-S_i^+}(f^{S_i^+}(y)) 
- \frac{1}{(f^{S_{i_0}})'(y)} w_{n_0-S_{i_0}^+}(f^{S_{i_0}^+}(y)) \\
& \qquad\qquad\qquad\qquad\qquad
- \sum_{i \leq  i_0} \frac{1}{(f^{T_i^+})'(y)} w_{S_{i}^+-T_i^+}(f^{T_i}(y))  
\end{align*}
is continuous on $[x,x+\epsilon_1)$.  Choose $\epsilon_2 < \epsilon_1$ such that 
if $x\leq y < x+\epsilon_2$ then $|\tilde{\alpha}(y)-\tilde{\alpha}(x)|< \epsilon/4$. 
Using again the uniformity of the estimates in the proof
of Proposition~\ref{abs} (in particular of the exponentially decaying
term of the series \eqref{first'} for $i=i_0$), it is easy to see that
if $n_0$ is large enough then  $|\alpha(y)-\alpha^+(x)|< \epsilon$
for $y \in [x, x+\epsilon_2)$.

The third  case occurs when   $\hat f^{i_0}(x,0)=(c,0)$
for some $i_0\ge 0$.  That is, there  exists  $i_0\ge 1$  such that $T_{i_0}(x)<\infty$ but $S_{i_0}(x)=\infty$,
and $\alpha(x)$ is just a finite sum. 
(If $x=c$ then $i_0=1$.)
Then there exists  $\epsilon_2>0$ such that if 
$x< y < x+ \epsilon_2$ then  
$T_i(y)=T_i^+$ and   $S_{i-1}(y)=S_{i-1}^+$  for every $i  \leq i_0$, 
while $S_{i_0}(y)\geq n_0$. 
To finish, define
\begin{align*}
\tilde{\alpha}(y)&= -\sum_{i <i_0} \frac{1}{(f^{S_i^+})'(y)} 
w_{T_{i+1}^+-S_i^+}(f^{S_i^+}(y)) 
- \sum_{i < i_0}
 \frac{1}{(f^{T_i^+})'(y)} w_{S_{i}^+-T_i^+}(f^{T_i^+}(y)) \, .
\end{align*}
and adapt the arguments from the
first two cases. 
This ends the proof of  \eqref{claimm}
and of Theorem ~ \ref{alphabded}.
\end{proof}

%%%%%%%

\subsection{Divergence of the formal power series  (Proposition ~ \ref{div})}
\label{prdiv}

In this final subsection, we show that the formal power series  $\alpha_{cand}(x)$ diverges
for many $x$. 

\begin{proof}[Proof of Proposition ~ \ref{div}] It is enough to show that the set of points such that 
$\limsup_i \big|\frac{v(f^i(x))}{(f^{i+1})'(x)}\big| > 0$
has the desired property. We shall build a decreasing sequence of closed sets 
$$A\supset \mathcal{K}_n \supset \mathcal{K}_{n+1} \, ,
$$
where each connected component of $\mathcal{K}_n$ is a closed interval with positive length, and a sequence of functions
$$
\Gamma_{n+1}: \{ A\mid  \mbox{ $A$  connected  comp.  of } \KK_n\}\to \PP(\{1,\dots,n+i_0\}) \, ,
$$
where $\PP(\{1,\dots,n\})$ stands for the family of all subsets of $\{1,\dots,n+i_0\}$, with the following properties:
\begin{itemize}
 \item[i.] If $C$ is a connected component of $\KK_n$ and $x \in C$, then 
$\big|\frac{v(f^{j}(x))}{(f^{j+1})'(x)}\big|\geq 2$
for every $j \in \Gamma_n(C)$.
\item[ii.] If $C$ is a connected component of $\KK_n$, then $f^n$ is a diffeomorphism on $C$.
\item[iii.] If $C_{n+1} \subset C_n$ are connected components of $\KK_n$ and $\KK_{n+1}$ respectively, then $\Gamma_n(C_n)\subset \Gamma_{n+1}(C_{n+1})$.
\item[iv.] If $C_n$ is a connected component of $\KK_n$, then there exists  $m > n$ such that $\KK_m$ has at least two connected components contained in $C_n$.
\item[v.]  If $x \in \cap_n C_{n}$, where $C_n$ is a connected component of $\KK_n$, then 
$\{x\}=\cap_n C_{n}$ and $\lim_n \# \Gamma_n(C_n)=\infty$.
\end{itemize}
Note that (iv) and (v) imply that $\cap_n \KK_n$ is a Cantor set. If $x \in \cap_n C_n$ then 
$$\big|\frac{v(f^{j}(x))}{(f^{j+1})'(x)}\big|\geq 2
$$
holds for every $j \in \cup_n \Gamma_n(C_n)$. Due to (v), there are infinitely many $j$'s.
Denote
$$\mathcal{O}(c)=\{x \in I\mid \ f^i(x)=c, \mbox{ for  some } i\geq 0\}\, .
$$
Let $\KK_0 \subset A$ be a closed interval $[a,b]$, $a\not= b$, with $a,b\not\in \mathcal{O}(c)$, and $\Gamma_0(\KK_0)=\emptyset$. Suppose that we have defined $\KK_n$. Let $C$ be a connected component of $\KK_n$. If 
$c \not\in f^n(C)$ then $C$ is the unique connected component of $\KK_{n+1}$ 
which intersects $C$ and $\Gamma_{n+1}(C)=\Gamma_{n}(C)$. Otherwise, let $x \in C$ be such that $f^n(x)=c$. Since $c$ is the critical point we have $(f^{n+i_0+1})'(x)=0$. 
Moreover $v(f^{n+i_0}(x))=v(f^{i_0}(x))\not=0$, so 
$$\lim_{y\rightarrow x} \big| \frac{v(f^{n+i_0}(y))}{(f^{n+i_0+1})'(y)}\big|=\infty\, .
$$

Let  $\epsilon > 0$ be such that if $0< |y-x| \leq  \epsilon$ then $y \in C$ and 
$$\big| \frac{v(f^{n+i_0}(y))}{(f^{n+i_0+1})'(y)}\big|\geq 2 \, .
$$

Choose two closed disjoint  intervals $J_1$ and $J_2$  with positive lengths, such that $J_1\cup J_2 \subset  (x-\epsilon,x+\epsilon)\setminus \{x\}$  and $\mathcal{O}(c)\cap \partial (J_1\cup J_2) =\emptyset$. Then $J_1$  and $J_2$ will be the unique connected components of $\mathcal{K}_{n+1}$ that intercept $C$ and 
$\Gamma_{n+1}(J_1)=\Gamma_{n+1}(J_2)= \Gamma_n(C)\cup\{n+i_0\}$. Note that in this case 
$\Gamma_{n+1}(J_1)=\Gamma_{n+1}(J_2) \varsupsetneqq  \Gamma_n(C)$.

Properties (i)--(iii) follow from the definition of $\mathcal{K}_n$. To show (iv) and (v), consider 
$C_\infty=\cap_n C_n$, 
where the $C_n$ are the connected components of $\mathcal{K}_n$. The set $C_\infty$ is either a closed interval of positive length or $\{x\}$. In the first case, in particular we have that $f^n$ is a diffeomorphism on $C_\infty$, for every $n$. This is not possible, since $f$ has neither  wandering intervals nor periodic attractors.  Furthermore, note that if  $\lim_n \#\Gamma_n(C_n) < \infty$, then there exists $n_0$ such that $\Gamma_n(C_n)=\Gamma_{n_0}(C_{n_0})$ for every $n\geq n_0$, so by the construction of $\mathcal{K}_i$ this occurs only if $f^i$ is a diffeomorphism on $C_{n_0}$ for every $i$, which is not possible, as we saw above.  The proof of (iv) is similar. If (iv) does not hold for certain $C_{n_0}$, then by the construction of $\mathcal{K}_{i}$ we have that $f^i$ is a diffeomorphism on $C_{n_0}$ for every $i$, which contradicts the non-existence of  wandering intervals and periodic attractors.
\end{proof}

%%%%%%%%%%%%%%%%%%%%%%%%%%%%%%%%%%%%%%%%%%%%%%%%%%

\section{Transfer operators $\widehat \LL$ and their spectra}\label{spectralstuff}

In this section, we study a transfer operator associated to a
Collet-Eckmann $S$-unimodal map $f$
 satisfying the Benedicks-Carleson condition.
More precisely,
in Subsection~\ref{gap},  we introduce a
Banach space $\BB=\BB^1$ of smooth ($H^1_1$) functions
(see Definition ~\ref{Bspace}) on
the tower $\hat I$ defined
in the previous section, maps $\Pi : \BB \to L^1(I)$ (see \eqref{defproj}),
as well as a transfer operator $\widehat \LL$ acting on $\BB$
(Definition ~\ref{opp}).
We shall prove that $\widehat \LL$ has essential spectral radius strictly
smaller than $1$ (Proposition ~ \ref{mainprop}), and then 
(Proposition ~\ref{acip})
that $1$ is a simple eigenvalue, and  that the fixed point
$\hat \phi$ of $\widehat \LL$ is such that $\Pi (\hat \phi)$ is the invariant density of $f$.
In Subsection~\ref{trunk}, we present results of truncated versions of
$\widehat \LL$, acting on finite parts of the tower.

The methods in this section are inspired from \cite{BV}, but we would
like to point out here that nontrivial modifications were needed
in view of proving Theorem~\ref{linresp}. (See
Remark~\ref{contrad} below for the comparison with \cite{ruelle}.) 
Also, the transfer operator we use here is slightly different from the
one in \cite{BV}: First, and this is the most
original ingredient, we introduce a smooth cutoff function
$\xi_k$ at each level on which there exist points which ``fall"
to the ground level; second, we do not compose with the
%maybe the second change is not needed
dynamics until we fall (this  strategy was
used by L.S. Young in \cite{young92}, \cite{young98}).
Finally,
our Banach spaces $\BB$ are not exactly the same as the space $\widehat \BB$
used in \cite{BV}: The functions in $\BB$ are smooth and locally supported
at each level (this is possible in view of the other two changes). These are the main new ideas in
the present section, and they allow to circumscribe the effect of the 
discontinuities and square root singularities
(called ``spikes" in \cite{ruelle}) to the maps $\Pi$ and $\Pi_t$
(see Step 3 in
the proof of Theorem~ \ref{linresp}).
See also the comments after Definition~\ref{opp}.
%%%%%%%%%%%%%%%%%%%%%%%%%%%%%%

 \subsection{Spectral gap for a transfer operator $\widehat \LL$ associated
to the tower map}\label{gap}

Let 
 $f$ be a $(\lambda_c, H_0)$-Collet-Eckmann $S$-unimodal map, with
a non preperiodic critical point (the preperiodic case is easier and
left to the reader) satisfying the Benedicks-Carleson condition
\eqref{BeC}. (We shall need to strengthen the condition slightly later
on, see \eqref{stBeC}, \eqref{BeCxtra}, and also
\eqref{bd2}.)
We consider the tower map $\hat f : \hat I \to \hat I$ from 
Section ~\ref{tower},
for some small enough fixed $\delta$.
We shall not require the fact that the Lyapunov exponents
of $\pm \delta$, or
of the endpoints $a_k$ and $b_k$ of the tower levels $B_k$, are positive,
and we {\it remove}
\footnote{This additional
freedom will also be used in Subsection~\ref{Lt}.} this assumption.
(The positivity of the exponents was
useful only when proving that $\alpha_t$
is continuous, and one can use  different towers
for  $f_t$ when studying $\alpha$ and when considering
the transfer operator.)
In particular, we may take for all $k \ge H_0$
$$
B_k=[c_k-e^{-\beta_1k}, c_k+ e^{-\beta_1 k}]\, .
$$

The following refinement of the estimates
in Subsection ~ \ref{tower'} will play an important part in our argument (see Proposition~\ref{mainprop},
and---in view of \eqref{defproj}---Proposition ~\ref{acip},
see also Step 2 in the proof of
Theorem~\ref{linresp}):

\begin{lemma}\label{rootsing}
Let $f$ be an $S$-unimodal $(\lambda_c,H_0)$-Collet-Eckmann map satisfying the  Benedicks-Carleson condition \eqref{BeC},  and
with a non-preperiodic critical
point. 
Then there exists $C$ so that for any $k \ge 1$, we have
\begin{equation}\label{rootsing1}
\frac{1}{|(f^k)'(f^{-k}_\pm(x))|}  \le 
C \frac{1}{|(f^{k-1})'(c_1)|^{1/2}\sqrt{|x-c_k|}  } \, ,
\quad \forall x \in \pi(E_k\cap \hat f^k(E_0)) \, .
\end{equation}

In addition, recalling the intervals $I_k$ defined in
\eqref{ints}, we have for any $k\ge H(\delta)$
\begin{equation}
\label{rootsing2}
\sup_{x \in f^k(I_k)}
\biggl |
\partial_x \frac{1}{|(f^k)'(f^{-k}_\pm(x))|}
\biggr | \le C  \frac{e^{3 \gamma k}}
{|(f^{k-1})'(c_1)|^{1/2}}\, ,
\end{equation}
and
\begin{equation}
\label{rootsing3}
\sup_{x \in f^k(I_k)}
\biggl |
\partial^2_x \frac{1}{|(f^k)'(f^{-k}_\pm(x))|}
\biggr | \le C  C  \frac{e^{5 \gamma k}}
{|(f^{k-1})'(c_1)|^{1/2}} \, .
\end{equation}
\end{lemma}

\begin{proof}
We consider the case $\varsigma=+$, the other case is identical.
Let us first show \eqref{rootsing1}.
Putting $z= f^{-(k-1)}_+(x)$, we decompose 
$$|(f^k)'(f^{-k}_+(x))|
= |(f^{k-1})'(z)| |f' (f^{-k}_+(x))|\, .
$$
By Lemma ~ \ref{cd}, the first factor 
can be estimated by
\footnote{The constant  $C$  depends on $H_0$, by
Lemma~\ref{est1aunif} it will be uniform within our families $f_t$.}
\begin{equation}\label{firstt}
|(f^{k-1})'(z)| \ge C^{-1} |(f^{k-1})'(c_1)|\, .
\end{equation}
For the second factor, we have
$$
|f' (f^{-k}_+(x))| \ge C^{-1} |f^{-k}_+(x)| \, .
$$
Put $w = f(f^{-k}_+(x))$.  Then, Lemma ~ \ref{cd} and the mean value
theorem imply
$$
|w-c_1| \ge   \frac{|x-c_k|} {C |(f^{k-1})'(c_1)| }\, .
$$
Next, noting that
$w \in \pi (E_1\cap \hat f(E_0))$, we have
$$
|f^{-k}_+(x)| =| f^{-1}_+ (w)|\ge C^{-1} \sqrt {|w-c_1|}
$$
(Just use that $f(y)=c_1 + f''(c) y^2 + f^{'''}(\tilde y) y^3$
for some $|\tilde y |\le \delta$,
if $|y|\le \delta$.)
The three previous inequalities give
\begin{equation}\label{secondd}
|f' (f^{-k}_+(x))| \ge C^{-1}\frac{ \sqrt {|x-c_k|}}{|(f^{k-1})'(c_1)|^{1/2}}\, .
\end{equation}
Putting together \eqref{firstt} and \eqref{secondd}, we get \eqref{rootsing1}.

To prove \eqref{rootsing2}, we first note that  there is
$C \ge 1$ so that 
\begin{equation}\label{firstnote}
|f'(f^j(y))|\ge C^{-1} e^{-\gamma j}\, , \, 
\forall y\in I_k\, , \, \forall 1 \le j \le k\, .
\end{equation}
Indeed, $|f^j(y)-c_j|\le e^{-\beta_1 j}$ and $|c_j-c| \ge e^{-\gamma k}$ for $j \ge H_0$,
with $\beta_1 \ge 3\gamma/2$, and we assumed that $f$ is $C^2$
with $f'(c)=0$ and $f''(c)\ne 0$.
Next, Lemma~\ref{cd} and
Lemma ~\ref{ubalpha} give $C>0$ so that
\begin{align}
\label{basic}
\sup_{y\in I_k}\frac{1}{|(f^{k-j})'(f^j(y))|}&\le C
\sum_{\ell = 1}^{\infty}
\frac{1}{  |(f^{\ell})'(f^{j-1}(c_1))|}\le   C e^{\gamma j} \,  , 
%\\ &\qquad\qquad 
\, \forall 1 \le j \le k-1 \, .
\end{align}
Then, \eqref{expii} from Lemma~\ref{sizeIj} gives
\begin{equation}
\label{basic2}
\sup_{y \in I_k}\frac{1}{|(f^m)'(y)|}\le 
C_2^{-1}
e^{\gamma k}\frac{ 1}
{ |(f^{m-1})'(c_1)|^{1/2}}\, ,
\quad \forall  1\le m \le k \, .
\end{equation}

Applying \eqref{firstnote} and \eqref{basic},  we get $C>0$ so that
\begin{align}
\nonumber 
\sup_{y\in I_k}
\partial_y \frac{1}{|(f^k)'(y)|}&\le
\sup_{y\in I_k}
\sum_{j=0}^{k-1}
\frac{|f''(f^j(y))|}{|(f^{k})'(y)|}\frac {|(f^{j})'(y)|}{|f'(f^j(y))|}\\
\label{gluu} &= \sup_{y\in I_k}
\sum_{j=0}^{k-1}
\frac{|f''(f^j(y))|}{|(f^{k-j})'(f^j(y))|}\frac {1}{|f'(f^j(y))|}
\le C  e^{2\gamma k}  \, .
\end{align}
Finally,
\begin{equation}\label{finally}
\partial_x \frac{1}{|(f^k)'(f^{-k}_+(x))|}=\frac{1}{|(f^k)'(f^{-k}_+(x))|}\cdot
\partial_y \frac{1}{|(f^k)'(y)|} \, .
\end{equation}
The first factor in \eqref{finally} is bounded
by \eqref{basic2} for $m=k$, the second by \eqref{gluu}, so that
 we have proved \eqref{rootsing2}.

To prove \eqref{rootsing3}, we start from the decomposition
\eqref{finally}, and  we deduce from
\eqref{gluu}  that for any
$x\in f^k(I_k)$, setting $y=f^{-k}_+(x)$,
\begin{align}
\nonumber \partial^2_x \frac{1}{|(f^k)'(f^{-k}_+(x))|}
&\le C e^{2\gamma k} \partial_x \frac{1}{|(f^k)'(f^{-k}_+(x))|} 
 \\
\label{toprove} &\quad +
\frac{1}{|(f^k)'(f^{-k}_+(x))|}\cdot
 \sum_{j=0}^{k-1}
  \partial_x\biggl (
\frac{|f''(f^j(y))|}{|(f^{k-j})'(f^j(y))| |f'(f^j(y))|}\biggr )\, .
\end{align}
By \eqref{rootsing2},
the first term in the right-hand-side is bounded by
$C  e^{5\gamma k}|(f^{k-1})'(c_1)|^{-1/2}$.
For the  second term, we have,
for $0\le j \le k-1$,
$$
 \partial_x
\frac{|f''(f^j(y))|}{|(f^{k-j})'(f^j(y))||f'(f^j(y))|}
\le  
\frac{1}{|(f^k)'(f^{-k}_+(x))|}
\partial_y
\frac{|f''(f^j(y))|}{|(f^{k-j})'(f^j(y))||f'(f^j(y))|}\, .
$$
Since $f$ is $C^3$, the Leibniz formula  gives for $0\le j \le k-1$, 
\begin{align}
%\label{next}
\nonumber
&
\frac{1}{|(f^k)'(f^{-k}_+(x))|}\partial_y
\frac{|f''(f^j(y))|}{|(f^{k-j})'(f^j(y))||f'(f^j(y))|}\\
\nonumber &\qquad\le
\frac{1}{|(f^k)'(f^{-k}_+(x))|}
\frac{1}{|(f^{k-j})'(f^j(y))||f'(f^j(y))|}
 \cdot 
\bigl [
|f'''(f^j(y))||(f^j)'(y)|]\\
\label{next}  &\qquad\qquad\qquad+
|f''(f^j(y))|
\biggl [ \frac{|f''(f^j(y))||(f^j)'(y)|}{|f'(f^j(y))|}+
\sum_{\ell=j}^{k-1}
\frac{|f''(f^\ell(y))| |(f^\ell)'(y)|}{|f'(f^\ell(y))|}
\biggr ]
\bigr ]\\
\nonumber &\le
\frac{C }{|(f^{k-j})'(f^j(y))|} \cdot 
\biggl [\frac{1}{|(f^{k-j})'(f^j(y))|}
\bigl (\frac{1}{|f'(f^j(y))|} +\frac{1}{|f'(f^j(y))|^2}\bigr )
\\ \nonumber &\qquad\qquad\qquad\qquad\qquad\qquad\qquad
\qquad\qquad
+\frac{1}{|(f^{k})'(y)|}
\sum_{\ell=j}^{k-1}
\frac{|(f^\ell)'(y)|}{|f'(f^\ell(y))|}
\biggr ]\, .
\end{align}
If $j\ge 1$, we may apply \eqref{firstnote}.
Then, \eqref{next}, together with 
\eqref{basic2} for 
$m=k-\ell$ and 
\eqref{basic} imply that
\begin{align*}
&
\frac{1}{|(f^k)'(f^{-k}_+(x))|}
\sum_{j=1}^{k-1}
\partial_y
\frac{|f''(f^j(y))|}{|(f^{k-j})'(f^j(y))||f'(f^j(y))|}
 \\ \nonumber &\qquad\qquad
\le
C e^{\gamma k} \cdot 
\bigl [e^{2\gamma k}+e^{3\gamma k} +
e^{2\gamma k} 
\bigr ]\, .
\end{align*}
If $j=0$, then \eqref{next} together with \eqref{eight'} and
\eqref{basic2} for $m=k$ imply
(distinguish between $\ell=0$ and $\ell \ge 1$)
\begin{align*}
&
\frac{1}{|(f^k)'(f^{-k}_+(x))|}\partial_y
\frac{|f''(y)|}{|(f^{k})'(y)||f'(y)|}
\\ \nonumber &\qquad\qquad
\le C \frac{e^{3\gamma k}}{|(f^k)'(c_1)|^{1/2} }(e^{\gamma k}+ 
e^{2\gamma k}|(f^k)'(c_1)|^{1/2} 
+  e^{2 \gamma k})\, .
\end{align*}
Putting the two above inequalities together with \eqref{toprove}
and \eqref{basic2} for $m=k$,
we get 
\eqref{rootsing3}.
\end{proof}

In view of the definition of our Banach space
$\BB$, we need further preparations. 
First, we assume from now on that $f$
satisfies the following strengthened Benedicks-Carleson condition:
\begin{equation}\label{stBeC}
\exists 0<\gamma <\frac{\log( \lambda_c)}{8}
\mbox{ so that } |f^k(c)-c| \ge e^{-\gamma k}\, , \quad \forall k \ge H_0 \, .
\end{equation}

\begin{remark}
In \cite{BV} we needed to assume that
$f$ was $C^4$ or symmetric (see the comments before
\cite[Lemma 5]{BV}) because of the more complicated form
of the cocycle used in the transfer operator there.
\end{remark}

In view of \eqref{stBeC}
we may choose $\lambda$ so that 
\begin{equation}\label{48}
1 <\lambda < e^{\gamma} \, , \mbox{ and }
e^{4\gamma} \lambda <  \sqrt \lambda_c \, .
\end{equation}

It does
not seem possible to work on spaces of sequences of $C^1$ functions $\psi_k$
(when summing
over critical inverse branches in the proof of Proposition~\ref{mainprop}, bounded distorsion
would  allow us to replace
$|(f^k)'(x)|^{-1}$ by the square root of the length of the corresponding
monotonicity interval, instead of the length itself), and 
it will be convenient to work with Sobolev spaces:
For integer $r\ge 0$, we recall that the generalised Sobolev
 norm of
 $\psi :I \to \complex$ is 
 $$
 \|\psi\|_{H^r_1}=\| \partial^r_x  \varphi(x) \|_{L^1}\, .
 $$
Compactly supported  $C^\infty$ functions are dense in $H^r_1$ for
$r\ge 1$ (our functions will compactly supported in the interior of $I$).
 The Sobolev embedding in dimension one gives
$\|\cdot\|_{C^0}\le C \|\cdot \|_{H^1_1}$ and
$\|\cdot\|_{C^1}\le C \|\cdot \|_{H^2_1}$.

\begin{definition}[The main 
Banach space $\BB=\BB^{H^1_1}$]\label{Bspace}
Let $\BB=\BB^{H^1_1}$ be the space of sequences
$\hat \psi=(\psi_k :I \to \complex,\,  k \in \integer_+)$,  so that
each $\psi_k$ is in $H^1_1$ and, in addition,
\begin{align}
\nonumber &\supp( \psi_0) \subset (-1, 1) \, , \qquad
\supp (\psi_k ) \subset  [-\delta, \delta]
\, , \, \forall 1 \le k  \le  \max (2, H_0)\, ,\\
\label{defban}
&\supp (\psi_k )
\subset  \cap_{H_0\le j \le k} (f^{-j}_+(B_j)\cup 
f^{-j}_-(B_j))
\, , \, \forall k > \max (2, H_0)\, , \end{align}
endowed with the norm 
\comment{would a $\lambda^k$ be useful here?}
$$
\| \hat \psi\|_{\BB^{H^1_1}}= 
\sum_{k \ge 0 }    \| \psi_k\|_{H^1_1}
\, .
$$
\end{definition}

We sometimes write $\hat \psi(x,k)$ instead of $\psi_k(x)$.

\begin{remark}\label{contrad}
In contradistinction to the piecewise expanding
case treated in \cite{bs1}, or to
the Misiurewicz {\it and analytic} case
studied in \cite{ruelle}, the postcritical data is not given here
by a finite set of complex numbers for each $c_k$ with $k \ge 1$: We need
a full ``germ'' $\psi_k$, which is supported
in a neighbourhood of $c$. 
Since we shall later consider 
$(\psi_k \chi_k )\circ f^{-k}_\pm$ (see \eqref{defproj}),  we can view
$\psi_k$ as the contribution in a one-sided neighbourhood 
of $c_k$.
\end{remark}

\begin{definition}[The projection $\Pi$]
Define $\Pi(\hat \psi)$ for $\hat \psi \in \BB$ by
\begin{align}\label{defproj}
\Pi(\hat \psi)(x)&=\sum_{k \ge 0, \varsigma\in \{+,-\}}
 \frac{\lambda^k}{|(f^{k})'(f^{-k}_\varsigma(x))|} \psi_k(f^{-k}_\varsigma(x)) 
 \chi_k(x)
   \, .
\end{align}
(We set $\chi_0\equiv 1$.
When the meaning is clear, we sometimes omit the factor
$\chi_k$ in the formula.)
\end{definition}

By \eqref{rootsing1}
in Lemma ~ \ref{rootsing} and our construction
\footnote{Uniformity of  $C$ within our families is important
here, it will follow  from Lemma~\ref{est1aunif}.},
setting $[c_k,d_k]=\pi(E_k \cap \hat f^k(E_0))$,
\begin{align}
\nonumber \int_{[c_k,d_k]}  \frac{\lambda^k}{|(f^{k})'f^{-k}_+(x)|}
 |\psi_k(f^{-k}_+(x)) | \, dx &\le 
C\lambda^k 
\lambda_c^{-k/2} \sqrt{|d_k-c_k|} \sup |\psi_k|  \\
\nonumber &\le C (\lambda \lambda_c^{-1/2} e^{-3\gamma/4} )^k 
 \sup |\psi_k|  \, .
\end{align}
Since $\lambda < \sqrt {\lambda_c}$, the above bound, and
its analogue for the branch $f^{-k}_-$,
imply that $\Pi (\hat \psi) \in L^1(I)$ for $\hat \psi \in \BB$.

We shall need a weak norm, in order to write Lasota-Yorke inequalities.
Set
\begin{equation}\label{cocycle}
w(x,k)=
\lambda^k 
\, , 
\quad x \in  I \, , k \ge 0\, ,
\end{equation}
and define $\nu$ to be the  nonnegative measure 
on $\cup_{k \ge 0}\{k\}  \times I$ whose density with respect to Lebesgue 
is $w(x,k)$.

\begin{definition}[Space $\BB^{L^1(\nu)}$]\label{oBspace}
Let $\BB^{L^1}=\BB^{L^1(\nu)}$ be the space of
sequences $\hat \psi$ of 
functions $\psi_k$,  with the norm
\begin{equation}\label{trunc}
\|\hat \psi \|_{\BB^{L^1}}=\sum_{k\ge 0}  \lambda^k \|\psi_k \|_{L^1(I)}= \|\hat \psi\|_{L^1(\nu)}\, .
\end{equation}
\end{definition}

In order to define the transfer operator $\widehat \LL$, we next 
introduce 
smooth cutoff functions $\xi_k$.
Recall the constants $3\gamma/2<\beta_1<\beta_2<2\gamma$ from \eqref{betas} and \eqref{betaas}.

\begin{definition}[The cutoff functions $\xi_k$]\label{defxi}
For each $k \ge 0$, let $\xi_k : I \to [0,1]$ be a $C^\infty$ function satisfying
the following conditions:  $\xi_k\equiv 1$ for those levels $k$ from which
no point falls to level $0$ 
%(in this case
%$f_\pm^{-(k+1)}[c_{k+1} - e^{-\beta_2(k+1)},c_{k+1} + e^{-\beta_2(k+1)}]$ is not  
%a strict subset of $f^{-(k+1)}_+(B_{k+1})\cup f^{-(k+1)}_-(B_{k+1})$), 
and otherwise
\begin{align*}
&\supp(\xi_0)= [- \delta, \delta]\, , \qquad 
\xi_0|_{[-\frac{\delta}{2}, \frac{\delta}{2}]} \equiv 1\, ,
\\
&k\ge H(\delta)\, : \begin{cases}
\supp(\xi_k)=
f^{-(k+1)}_+(B_{k+1})\cup f^{-(k+1)}_-(B_{k+1})\, ,  
\\ 
\xi_k|_{f^{-(k+1)}_+[c_{k+1} - e^{-\beta_2(k+1)},c_{k+1} + e^{-\beta_2(k+1)}]} 
\equiv 1 \, ,\\
\xi_k \mbox{ is unimodal,}
\end{cases}
\end{align*}
$|\partial^j_x (\xi_0\circ f^{-1}_\pm(x))|\le c(\delta)^{-j}$ for $j=1,2,3$,  and, finally, that $\beta_1$ is close enough to $3\gamma/2$ and 
 $\beta_2$ is close enough to $2\gamma$ so that for some $C >0$  
\begin{equation}\label{kappa'}
\sup |\partial^j_x (\xi_k\circ f^{-(k+1)}_\pm(x))|\le C e^{2j \gamma k}
\, ,  \, \, j=1,2 ,3 \, .
\end{equation}
\end{definition}

\smallskip
Note that  $\xi_k(y)>0$ if and only if $\hat f (f^{k}(y),k)\in B_{k+1}\times
(k+1)$, 
and $\xi_k(y)=1$ implies that
$\pi \hat f(f^k(y),k)\in[c_{k+1} - e^{-\beta_2(k+1)},c_{k+1} + e^{-\beta_2(k+1)}]$. 
The low levels ($k\le H(\delta)$) will be taken care of by the condition
$\supp(\psi_k)\subset [-\delta,\delta]$.

\begin{definition}[Transfer operator]\label{opp}
The transfer operator $\widehat \LL$ is defined on  $\BB$ by
\begin{equation}\label{a}
(\widehat \LL  \widehat \psi ) (x,k)=
\begin{cases}
\frac{\xi_{k-1}(x)}{\lambda} \cdot \hat\psi (x,k-1)& k \ge 1 \,, \\
\sum_{j \ge 0, \varsigma\in \{+,-\}}
\frac{\lambda^j (1- \xi_j(f^{-(j+1)}_\varsigma(x)))} {|(f^{j+1})'(f^{-(j+1)}_\varsigma(x))|}\cdot
\hat \psi(f^{-(j+1)}_\varsigma(x),j)
& k=0 \, .
\end{cases}
\end{equation} 
\end{definition}

Some $j$-terms in the sum for  $(\widehat \LL \widehat \psi) (x,0)$ vanish, in particular, for all $1\le j <H_0$ because of
our choice of small $\delta$ giving $H(\delta)\ge H_0$.

As already mentioned
in the beginning of this section, there are two differences
between the present definition and the one used in \cite{BV}. First,
 $\widehat \LL$ does not act
via the dynamics when climbing the tower, only when falling.
Secondly, if  $0<\xi_j(y)<1$, 
then $y$ will contribute to  both
$(\widehat \LL \widehat \psi) (y,j+1)$
and 
$(\widehat \LL \widehat \psi) (f^{j+1}(y),0)$. In other words,
the transfer operator just defined 
is  associated to a multivalued (probabilistic-type)
tower dynamics.
For this multivalued dynamics, some points may fall from the tower a little
earlier than they would for $\hat f$. However, the conditions
on the  functions $\xi_k$s guarantee that they do not
fall {\it too} early. More precisely, if we define ``fuzzy"
analogues of the intervals $I_k$ from \eqref{ints} as follows
\begin{equation}\label{tildeint}
\widetilde I_k:=
\{ x \in I \mid \xi_k(x)<1 \, , \, \xi_j(x)>0\, , \forall
0\le j < k \}\, ,
\end{equation}
then we can replace $I_k$ by $\widetilde I_k$ in the
previous estimates, in particular in Lemma~\ref{rootsing}.
Indeed, just observe that if a point ``falls" according to our fuzzy
dynamics, it would have fallen for some choice of intervals $\widetilde B_k$ so that
$$
[c_k-e^{-\beta_2 k},c_k+e^{-\beta_2 k}]
\subset \widetilde B_k\subset B_k
\, .
$$ 

\begin{remark}
\label{overlap}
The intervals $\widetilde I_k$ do not have pairwise disjoint interiors. However, for each $k$, the cardinality
of those $\widetilde I_j$ whose interiors intersect the interior of
$\widetilde I_k$ is bounded  by $k$ for $k\ge N(f)$,
where $N(f)$ depends only on $\sigma$ from
Lemma~\ref{est1a} and on $\gamma$ (the smaller
$\gamma$, the shorter this waiting time). We
may assume by taking smaller $\delta$ that $N(f)\le H(\delta)$. Indeed, if a point
falls $x$ from level $k$ for the first time with the fuzzy dynamics, it will
fall for the last time at level $2k$,  because Lemma~\ref{est1a} 
for $\delta=e^{-4\gamma k}$ together with the Benedicks-Carleson assumption
give $|(f^k)'(x)|\ge c e^{-4\gamma k}\sigma^k$,
and $c e^{-4\gamma k}\sigma^k e^{-2\gamma k}> e^{-\frac{3}{2} 2k}$.
(The present remark will be used to get the Lasota-Yorke estimate at the heart
of Proposition~\ref{mainprop}, see Appendix~\ref{app2}.)
\end{remark}

These two modifications allow us to work with  Sobolev
spaces $H^r_1$  (as opposed to the BV functions
in \cite{BV}, where the jump singularities corresponding
to the edges of the levels are an artefact
of the construction), and will  simplify our spectral
perturbation argument in Section ~ \ref{finalproof}.

Before we continue, let us note that 
if we introduce the ordinary (Perron-Frobenius) transfer operator
$$
\LL : L^1(I)\to L^1(I)\, ,\qquad
\LL \varphi(x)=\sum_{f(y)=x} \frac{\varphi(y)}{|f'(y)|}\, ,
$$
then one easily shows that
\begin{equation}\label{commute}
\LL (\Pi (\hat \psi))=\Pi(\widehat \LL(\hat \psi))\, .
\end{equation}
Indeed, on the one hand, we have
(recall $\chi_0\equiv 1$)
\begin{align*}
&[\Pi \widehat \LL (\hat \psi)](x) =
 [\widehat \LL (\hat \psi)](x,0) \chi_0(x)\\
&\qquad+
\sum_{k\ge 1,\varsigma\in \{+,-\}}
 \frac{\lambda^k}{|(f^{k})'(f^{-k}_\varsigma(x))|} 
[\widehat \LL (\hat \psi)](f^{-k}_\varsigma(x),k) \chi_k(x)\\
\nonumber &=
\sum_{j \ge 0, \varsigma\in \{+,-\}} 
\frac{\lambda^j}{|(f^{(j+1)})'f^{-{(j+1)}}_\varsigma(x)|} 
\psi_j(f^{-{(j+1)}}_\varsigma(x)) 
(1- \xi_j)(f^{-{(j+1)}}_\varsigma(x))  \chi_1(x)\\
\nonumber&\quad+
\sum_{k\ge 1,\varsigma\in \{+,-\}}
 \frac{\lambda^{k-1}}{|(f^{k})'(f^{-k}_\varsigma(x))|} 
 \psi_{k-1}(f^{-k}_\varsigma(x)) \xi_{k-1}  (f^{-k}_\varsigma(x)) \chi_k(x)\, \\
 &=
\sum_{j \ge 0, \varsigma\in \{+,-\}} 
\frac{\lambda^j}{|(f^{(j+1)})'f^{-{(j+1)}}_\varsigma(x)|} 
\psi_j(f^{-{(j+1)}}_\varsigma(x)) 
(1- \xi_j)(f^{-{(j+1)}}_\varsigma(x))  \chi_1(x)\\
\nonumber&\quad+
\sum_{j\ge 0,\varsigma\in \{+,-\}}
 \frac{\lambda^{j}}{|(f^{(j+1)})'(f^{-(j+1)}_\varsigma(x))|} 
 \psi_{j}(f^{-(j+1)}_\varsigma(x)) \xi_{j}  (f^{-(j+1)}_\varsigma(x)) \chi_{j+1}(x)\, \, \, .
\end{align*}
On the other hand
\begin{align*}
&(\LL\Pi \hat\psi) (x)
=\sum_{f(y)=x}
\frac{\chi_1(x)}{|f'(y)|}
\biggl(\sum_{\ell \ge 0, \varsigma\in \{+,-\}}
 \frac{\lambda^\ell}{|(f^{\ell})'(f^{-\ell}_\varsigma(y))|} \psi_\ell(f^{-\ell}_\varsigma(y)) 
 \chi_\ell(y)\biggr)\\
 &\, \, =\sum_{f(y)=x}
\frac{\chi_1(x)}{|f'(y)|}
\biggl (\sum_{\ell \ge 0, \varsigma\in \{+,-\}}
 \frac{\lambda^\ell}{|(f^{\ell})'(f^{-\ell}_\varsigma(y))|} 
 \psi_\ell(f^{-\ell}_\varsigma(y)) 
 \bigl ( 1-\xi_\ell(f^{-\ell}_\varsigma(y) \bigr )
 \chi_\ell(y)\biggr ) \\
 &\, \, \, +
 \sum_{f(y)=x}
\frac{\chi_1(x)}{|f'(y)|}
\biggl ( \sum_{\ell \ge 0, \varsigma\in \{+,-\}}
 \frac{\lambda^\ell}{|(f^{\ell})'(f^{-\ell}_\varsigma(y))|} 
 \psi_\ell(f^{-\ell}_\varsigma(y)) 
 \xi_\ell(f^{-\ell}_\varsigma(y) )
 \chi_\ell(y)\biggr )\\
 &\, \, =\sum_{f(y)=x}
\biggl (\sum_{j \ge 0, \varsigma\in \{+,-\}}
 \frac{\lambda^j}{|(f^{(j+1)})'(f^{-j}_\varsigma(y))|} 
 \psi_j(f^{-j}_\varsigma(y)) 
 \bigl ( 1-\xi_j(f^{-j}_\varsigma(y) \bigr )
 \chi_1(x)\chi_j(y)\biggr ) \\
 &\, \, \, +
 \sum_{f(y)=x}
\biggl ( \sum_{j \ge 0, \varsigma\in \{+,-\}}
 \frac{\lambda^j}{|(f^{(j+1)})'(f^{-j}_\varsigma(y))|} 
 \psi_j(f^{-j}_\varsigma(y)) 
 \xi_j(f^{-j}_\varsigma(y) )
 \chi_1(x)\chi_j(y)\biggr )\, .
 \end{align*}
 Now,  our definitions ensure that
 $$
 \xi_j(f^{-j}_\varsigma(y)) \ne 0
 \Longrightarrow\chi_j(y) \chi_1(f(y))
 = \chi_{j+1}(f(y))\, ,
 $$
 and
 $$\psi_j(f^{-j}_\varsigma(y)) 
 \bigl ( 1-\xi_j(f^{-j}_\varsigma(y) )\bigr )\ne 0
 \Longrightarrow
 \chi_j(y)=1\, .
 $$
 Finally, the sums over inverse branches coincide:
 Distinguish between zero level --- where it is obvious --- and other levels --- where the last sum ($f(y)=x$)
 over two inverse branches has at most
 one nonzero contribution by the support properties
 of the $\psi_k$ and $\xi_k$. This proves (\ref{commute}).
 In particular, if $\widehat \LL( \hat \phi)=\hat \phi$
then $\LL(\Pi(\hat \phi))=\Pi(\hat \phi)$.

Our main result on the spectral properties of $\widehat \LL$ follows:

\begin{proposition}[Essential spectral radius of $\widehat \LL$]
\label{mainprop}
Let $f$ be an $S$-unimodal $(\lambda_c,H_0)$-Collet-Eckmann map satisfying the strengthened Benedicks-Carleson condition \eqref{stBeC}, 
with a non-preperiodic critical
point.  
Let $\lambda$  satisfy \eqref{48}.
Then the operator $\widehat \LL$ is bounded on 
$\BB$, with spectral radius equal to $1$.
The  dual of  $\widehat \LL$ fixes the measure $\nu$ defined after
\eqref{cocycle}.
Let
$\rho$ satisfy \eqref{2.1} and let
$\sigma >1$  be the constant from Lemma~ \ref{est1a}, and set
\begin{equation}\label{theta0}
\Theta_0:=
\min( \frac{ \lambda_c^{1/2}}{e^{4\gamma} \lambda}\, , \lambda \, , \sigma \, , \rho)>1\, .
\end{equation}
The  essential spectral radius of $\widehat \LL$
on $\BB$  is bounded by 
$\Theta_0^{-1}$.
\end{proposition}

\begin{proof}
Let $c(\delta)$  be the constant from Lemma~ \ref{est1a}.

For $\hat \psi \in \BB$, our assumptions on the 
$\xi_{j}$ ensure that   $(\widehat \LL (\hat \psi))_k\in H^1_1$ for all $k\ge 1$,
with $(\widehat \LL (\hat \psi))_k$  supported in the desired interval, 
and that  $(\widehat \LL (\hat \psi))_0$  is  supported in the desired interval.

Note that for any interval $U$ (not necessarily containing
the support of $\psi_j$), using the Sobolev embedding again,
\begin{equation}\label{sobeb}
\sup_U |\psi_j|\le \min (C \|\psi'_j\|_{L^1}, \int_{U}|\psi'_j|\, dx + |U|^{-1} \int_{U} |\psi_j|\, dx)\, .
\end{equation}
Since $\xi_\ell$ is unimodal if it is not $\equiv 1$, for each $\ell \ge 1$ there exist
$v_\ell >u_\ell$ in $B_\ell$ so that, setting $J_\ell=\{x \in \supp(\psi_j)\mid x\le u_\ell\}\cup \{x \in \supp(\psi_\ell)\mid
x \ge v_\ell\}$,
\begin{align}\label{dirtytrick}
\int_{B_\ell} |\xi_\ell'\psi_\ell|\, dx&= \int_{x \le u_\ell} \xi_\ell' |\psi_\ell|\, dx
- \int_{x \ge v_\ell} \xi_\ell' |\psi_\ell|\, dx\le (\xi_\ell(u_\ell)+\xi_\ell(v_\ell)) \sup_{J_\ell} |\psi_\ell|
\\
\nonumber &\le 2 \sup_{J_\ell} |\psi_\ell|
\end{align}
(this can be viewed as the analogue of the
``$2\sup\sup$" boundary term in \cite{BV}).
Therefore,
for all $k\ge 1$, using also \eqref{sobeb},
\begin{align}
%\nonumber & \| (\widehat \LL \hat \psi)_k\|_{L^1} \le \frac{1}{\lambda} \|\psi_{k-1}\|_{L^1} \, ,   \\
\label{der0}&
\| (\widehat \LL (\hat \psi))_k'\|_{L^1} \le \frac{3C}{\lambda} \| \psi'_{k-1}\|_{L^1} \, . 
\end{align}
More generally, for $1\le n\le k$,
\begin{align}
\label{der0N}&
\| (\widehat \LL^n (\hat \psi))_k'\|_{L^1} \le \frac{3C n}{\lambda^n} \| \psi'_{k-n}\|_{L^1} \, . 
\end{align}

If $|\psi_k(y)|>0$ then $|f^{j}(y)-c_{j}|\le e^{-\beta_1 j}$
for all $j\le k$.
If $\xi_k(f^{-(k+1)}_\pm(x))<1$ then $|x-c_{k+1}|\ge e^{-\beta_2 (k+1)}$.
Thus,  changing  variables in the integrals, using \eqref{dirtytrick}
for the terms involving $\xi'_k$ for $k \ge 0$, and
recalling   \eqref{expii} from Lemma~\ref{sizeIj}, as well as
\eqref{gluu} and \eqref{finally} from Lemma~\ref{rootsing}, 
%and using   \eqref{kappa'}, 
we see
that $(\widehat \LL (\hat \psi))_0$ belongs to  $H^1_1$ and
\begin{align}
\label{der00}\| (\widehat \LL (\hat \psi))'_0\|_{L^1} 
&\le   Cc(\delta)^{-1} (\|\psi_0'\|_{L^1}+  \|\psi_0\|_{L^1}+ \sup|\psi_0|)\\
\nonumber &\qquad +
\sum_{k\ge H(\delta)}   \frac{\lambda^k e^{2 \gamma k}}{|(f^{k+1})'(c_1)|^{1/2}}
(\|\psi'_k\|_{L^1} + \sup |\psi_k| + \|\psi_k\|_{L^1} )\, .
\end{align}
In view  of \eqref{48} and \eqref{sobeb}, we have proved that $\widehat \LL$ is bounded
on $\BB$.  (The claim on the spectral radius will be proved below.)

\smallskip
Observe that $\sum_k \int_{B_k} \hat \psi(x,k) w(x,k)\, dx$
is finite if $\hat \psi \in \BB$
(just recall that $|B_k|\le 2 e^{-\beta_1k}\le 2e^{-3\gamma k/2}$
and use the bound
$\lambda < e^{3\gamma/2}$
from \eqref{48}). So $\nu$ is an element of the dual
of $\BB$.
The fact that  $\widehat \LL^*(\nu)=\nu$  can easily be  proved using the change of variables
formula. Indeed,
\begin{align}
\label{miracle}&\int \widehat \LL (\hat \psi) \, d\nu=\int_{B_0} 
\widehat \LL (\hat \psi)(y,0) \, dy +
\sum_{k\ge 0} \int_{B_{k+1}}\widehat \LL (\hat \psi)(y,k+1)\, w(y,k+1)\, dy\\
\nonumber &\qquad=
\sum_{j \ge 0, \varsigma\in \{+,-\}} \int
\frac{\lambda^j}{|(f^{j+1})'f^{-{(j+1)}}_\varsigma(y)|} 
\psi_j(f^{-{(j+1)}}_\varsigma(y)) 
(1- \xi_j)(f^{-{(j+1)}}_\varsigma(y))  \, dy\\
\nonumber&\qquad\qquad\qquad\qquad
 +
\sum_{k\ge 0} \int
\frac{1}{\lambda}\psi_k(x) \xi_k  (x)\,  w(x,k+1)\, dx\\
\nonumber&\qquad=
\sum_{j \ge 0} \int \psi_j (x)\, (1-\xi_j)(x) 
 \, w(x,j) \, dx
 +
\sum_{k\ge 0} \int\psi_k(x) \xi_k  (x)\,  w(x,k)\, dx\, .
\end{align} 
Note for further use that $\widehat \LL^* (\nu)=\nu$ implies
\begin{equation}\label{l1bd}
\nu(|\widehat \LL^N(\hat \psi)|)\le \nu(\widehat \LL^N( |\hat \psi|))=\nu(|\hat \psi|)\, .
\end{equation}
\smallskip

We next estimate the spectral and essential spectral radii
of  $\widehat \LL$ on $\BB$.
Using \eqref{dirtytrick} and the overlap control of fuzzy intervals, it is  not very difficult 
(see Appendix~\ref{app2}) 
to adapt the proof of \cite[Sublemma]{BV}
to show inductively that for any $\Theta<\Theta_0$,
there exists $C$, and for all $n$ there exists $C(n)$, so that 
\begin{equation}
\|(\widehat \LL^n( \hat \psi))'_0(x)\|_{L^1(I)} \le
C \Theta^{-n}
\|\hat \psi\|_{\BB} + C(n) \|\hat \psi\|_{\BB^{L^1}}\, .
\end{equation}
Recalling \eqref{der0N}, and using \eqref{l1bd}, up to slightly decreasing
$\Theta$, one then finds
$C'$ so that for all $n\ge 1$
(see the proof of \cite[Variation Lemma]{BV})
\begin{equation}\label{keyLM}
\|\widehat \LL^n( \hat \psi)\|_{\BB} \le
C' \Theta^{-n}
\|\hat \psi\|_{\BB} + C' \|\hat \psi\|_{\BB^{L^1}}
\, .
\end{equation}
Since \eqref{defban} implies that the length
of the support of $\psi_k$ is (much) smaller than 
$\lambda^{-2k}$, we find $\|\hat \psi\|_{\BB^{L^1}}\le   \|\hat \psi\|_{\BB^{L^1}}
+C \lambda^{-M}\|\hat \psi\|_{\BB^{H^1_1}}$
for all $M\ge 1$ (we used again the Sobolev embedding to estimate
the supremum by the $H^1_1$ norm).
The bound \eqref{keyLM} implies that the spectral radius of $\widehat \LL$ on $\BB$
is at most one, and thus equal to one.

Finally,  since
Rellich--Kondrachov implies that $\BB^{H^1_1}$ 
is compactly included in $\BB^{L^1}$ (the total length of the tower
is bounded, even up to $\lambda^k$-expansion at
kevel $k$), the Lasota-Yorke estimate \eqref{keyLM} together with
Hennion's theorem \cite{Hen}  give the claimed bound on 
essential spectral radius of $\widehat \LL$ on $\BB=\BB^{H^1_1}$.
This ends the proof
of Proposition ~ \ref{mainprop}.
\end{proof}

We next state  
further spectral properties of $\widehat \LL$.

\begin{proposition}[Maximal eigenvalue of $\widehat \LL$]\label{acip}
Let $f$ be an $S$-unimodal $(\lambda_c,H_0)$-Collet-Eckmann map satisfying the strengthened Benedicks-Carleson condition \eqref{stBeC}, 
with a non-preperiodic critical
point.

The maximal eigenvalue $1$ is a simple eigenvalue of $\widehat \LL$, for
a nonnegative eigenvector $\hat \phi$.
If  $\nu(\hat \phi)=1$, then $\phi:=\Pi (\hat \phi)$
is the density of the unique absolutely continuous $f$-invariant
probability measure.
Finally,   if $f$ is $C^4$
and the Benedicks-Carleson condition \eqref{stBeC} is
strengthened to
\begin{equation}\label{BeCxtra}
0<\gamma< \frac{\log \lambda_c}{14}\, ,
\end{equation}
then one can choose the parameter $\lambda$ 
so that   $\hat \phi_0 \in H^2_1$.
\end{proposition}

Note  that $\sup_k \|\hat\phi\|_{H^2_1}=\infty$, since otherwise $\sup_k\|\hat \phi_k\|_{C^1}<\infty$ 
so that $\sup|\hat \phi_k|< C |\supp(\hat\phi_k)|$, which is impossible
since $\sup| \hat\phi_k|=\lambda^{-k} |\hat \phi_0(c)|\ne 0$.

\begin{proof}
By Proposition ~ \ref{mainprop},  the essential
spectral radius of $\widehat \LL$ on $\BB$ is strictly smaller
than one and the spectral radius is equal to  $1$. The fact that $1$ is a simple eigenvalue
for a nonnegative eigenvector then follows from standard arguments
(see, e.g., \cite[Corollaries 1, 2]{BV} or \cite[Propositions 5.13, 5.14]{V}).
The normalisation $\nu(\hat \phi)=1$ implies that
$\int_I \phi\, dx=\int_I \Pi(\hat \phi)\, dx=1$.
Recalling \eqref{commute}, we get that $\LL(\Pi(\hat \phi))= \phi$
so that  $\Pi(\hat \phi)\in L^1$ is indeed the invariant density
of $f$ (which is known to be unique and ergodic).

It only remains to show that, under
a stronger Benedicks-Carleson condition, $\hat \phi_0 \in H^2_1$
if $f$ is $C^4$.
For this,  take $\hat\psi$ so that $\hat\psi_k=0$
for all $k\ge 1$ and $\hat\psi_0$ is $C^\infty$, of Lebesgue
average $1$ (we can even take $\hat \psi_0$ constant
in a neighbourhood of $[c_2,c_1]$), and use
that $\widehat \LL^n(\hat \psi)$ converges to $\hat \phi$
in the $\BB^{H^1_1}$ norm
(exponentially fast) as $n\to \infty$. 
We claim that
$\|(\widehat \LL^n(\hat \psi))_0\|_{H^3_1}\le C$ for all $n$, up to a suitable
modification of the conditions on
$\lambda$  in \eqref{48}.
Adapting the proof of \eqref{rootsing3}, one shows
$\sup_{x \in f^k(I_k)}\biggl |
\partial^3_x \frac{1}{|(f^k)'(f^{-k}_\pm(x))|}
\biggr | \le   C  \frac{e^{7 \gamma k}}{|(f^{k-1})'(c_1)|^{1/2}}$.
Then, in view of \eqref{BeCxtra}, one
can exploit (in addition to the properties already used in
the proof of the Lasota-Yorke estimates
for the $H^1_1$ norm in Proposition~\ref{mainprop})
the  conditions on  $\xi''_k$, $\xi'''_k$ in \eqref{kappa'}
to adapt \eqref{recurs} in Appendix~\ref{app2} (noting also
that $\hat\psi|\omega=0$
if the interval $\omega $ is in some level $E_k$ with $k>0$). 
Note that \eqref{dirtytrick} is not needed, since we only look at
the component of $\widehat \LL^n( \hat \psi)$ at level $0$.
\comment{recheck...\eqref{dirtytrick} could be used to find
Bspace containing $\hat \phi$}
 Details are straightforward, although tedious, and left to the reader.
  (We do not claim that the factor $14$ in 
 \eqref{BeCxtra} is optimal. In any case, we shall need to work
 with the stronger TSR condition soon.) 
To conclude, use that if a sequence  converging
to $\hat \phi_0$ in $H^1_1(I)$ has bounded $H^3_1(I)$ norms 
then $\hat \phi_0\in H^2_1(I)$ by Rellich Kondrakov.
\end{proof}

%%%%%%%%%%%%%%%%%%%%%%%%%%%%%%%%%%%%%%%%%%%%%%%%%%%%%%%%%%%%%%%%%%%%%%%%%

\subsection{Truncated transfer operators $\widehat \LL_{M}$}
\label{trunk}

We introduce 
for each $M \ge 0$ the {\it truncation operator}
$\TT_M$ defined  by 
\begin{align}\label{truncc}
\TT_M(\hat \psi)_k=
\begin{cases} \psi_k & k \le M\\
0 & k > M  \, .
\end{cases}
\end{align}
Clearly, $\|\TT_M\|_{\BB^{H^r_1}}\le 1$ for all $r\ge 0$ and $\|\TT_M\|_{\BB^{L^1}}\le 1$.
We  consider the
operator defined  by
$$
\widehat \LL_{M} = \TT_M \widehat \LL \TT_M \, .
$$

By using the results of Keller and Liverani \cite{kellerliverani},  we shall prove
the following result:

\begin{lemma}[Spectrum of the truncated operators]\label{truncspec}
Let $f$ be an $S$-unimodal $(\lambda_c,H_0)$-Collet-Eckmann map satisfying the strengthened Benedicks-Carle\-son condition \eqref{stBeC}, 
with a non-preperiodic critical
point.   Recall $\Theta_0$ from  \eqref{theta0}.

The essential spectral radius of $\widehat \LL_{M}$ acting on 
$\BB$ 
is not larger than $\Theta_0^{-1} <1$.

In addition,
there exists $M_0\ge 1$ so that
for all $M \ge M_0$  the operator $\widehat \LL_{M}$ has a real
nonnegative maximal eigenfunction $\hat \phi_{M}$,
for an eigenvalue $\kappa_{M}> \Theta_0^{-1}$,
the dual operator of  $\widehat \LL_{M}$ has a 
nonnegative maximal eigenfunction $\nu_{M}$, and, setting
\begin{equation}\label{tauM}
\tau_M=e^{3\gamma M} |(f^M)'(c_1)|^{-1 /2}<1\, ,
\end{equation}
for any $\eta <1$ 
there exists $C>0$ so that,
normalising by 
$\nu(1)=\nu_M(1)$ and
$\int \hat \phi_{M}\, d\nu_{M}=1$, we have
\begin{equation}\label{klbounds}
\|\hat \phi-\hat \phi_{M}\|_{\BB^{L^1}} \le C 
\tau_M^\eta \, , \, \, 
\| \nu- \nu_{M}\|_{(\BB^{H^1_1})^*} \le C \tau_M^\eta\, , \, \, 
|\kappa_{M} -1 | \le C \tau_M^\eta\, .
\end{equation}
If $f$ is $C^4$ and \eqref{BeCxtra} holds, we may choose
$\lambda>1$   so that
$\sup_M\|\hat \phi_{M,0}\|_{H^2_1} < \infty$.
\end{lemma}

\begin{proof}
The claim about the essential spectral radius can be obtained by
going over the proof of Proposition ~ \ref{mainprop} and checking that it applies
to $\widehat \LL_{M}$, and that the constants
are uniform in $M$. The reader is invited to do this, and
to check  that we have the following uniform
Lasota-Yorke estimates for $\widehat \LL$ and
$\widehat \LL_{M}$: There exists $C \ge  1$  so that for all
$N$ and all $M$
\begin{equation}\label{LYM}
\max(\|\widehat \LL^N (\hat \psi)\|_{\BB^{H^1_1}}, 
\|\widehat \LL_{M}^N (\hat \psi)\|_{\BB^{H^1_1}})
\le  C \Theta^{-N} \| \hat \psi\|_{\BB^{H^1_1}} + 
C  \|\hat \psi\|_{\BB^{L^1}} \, ,
\end{equation}
and  (recall \eqref{l1bd}
and note that $\nu(|\widehat \LL^N_M(\hat \psi)|)\le\nu(\widehat \LL^N_M( |\hat \psi|))\le
\nu(\widehat \LL^N( |\hat \psi|))=\nu(|\psi|)$)
$$
\| \widehat \LL^N  \|_{\BB^{L^1}}\le 1\, ,\quad
\|( \widehat \LL_{M})^N  \|_{\BB^{L^1}}\le 1 \, , \, \forall M \, ,
\forall N \, . 
$$ 
Finally,
there exists $C$ so that for all large enough $M$ 
$$
\|(\widehat   \LL - \widehat \LL_{M}) (\hat \psi)\|_{\BB^{L^1}}\le C 
\tau_M \|\hat \psi \|_{\BB^{H^1_1}}\, . 
$$ 
The last inequality is an easy consequence of
\begin{align*}
&\|\ (\id - \TT_ M) \hat \psi\|_{\BB^{L^1}}
\le C \tau_M \|\hat \psi\|_{\BB^{H^1_1}} \, ,
\end{align*} 
which follows from 
Lemma~\ref{rootsing} and Proposition~\ref{ubalpha}, since
$|\sup \psi_k|\le C \|\psi_k'\|_{L^1}$.
The bounds \eqref{klbounds} for  
$\eta \in (0, 1)$
then follow from 
 \cite[Theorem 1, Corollary 1]{kellerliverani}.

Note for use in Step 1 of the proof of Theorem~\ref{linresp}
in Section ~\ref{finalproof}
that the first claim of
\cite[Theorem 1]{kellerliverani} gives a small disc $B$ around $1$ so that
\begin{equation}\label{forrlater}
\sup_{M\ge M_0} \sup_{z \notin B}\| z -\widehat \LL_M\|^{-1} _{\BB^{H^1_1}}< \infty \, ,
\end{equation}
while, letting
$
\PPP_{M}(\hat \psi)=
\hat \phi_{M}  \nu_{M}(\hat \psi)
$ 
be
the  spectral projector corresponding to the maximal eigenvalue
of $\widehat \LL_M$, and setting
\begin{equation}\label{forrlater'}
 \NN_M:=(\kappa_{M} -\widehat \LL_M)^{-1}(\id-\PPP_M)- 
(\id -\widehat \LL_0)^{-1}(\id-\hat \phi \nu(\cdot))\, .
\end{equation}
the second claim of \cite[Theorem 1]{kellerliverani} 
with the first lines of \cite[Appendix B]{bs1} give
\begin{equation}\label{forrlater''}
\| \NN_M (\hat \psi)\|_{\BB^{L^1}}\le \widehat C\tau_M^{\eta} \|\hat \psi\|_{\BB^{H^1_1}}\, ,
\end{equation}
and
\begin{align}\label{newKL'}
\Delta&:=
\|(\kappa_{M} -\widehat \LL_{M})^{-1}(\id-\PPP_{M})\|_{\BB^{H^1_1}}
< \infty \, .
\end{align}

It follows from what has been done up to now and \cite{kellerliverani} that
$\sup _M \|\hat \phi_M\|_{\BB^{H^1_1}}<\infty$.
For the last claim of Lemma ~\ref{truncspec},
we proceed  like in the analogous statement
of Proposition ~ \ref{acip},
and get uniform bounds in $M$. 
\end{proof}

%%%%%%%%%%%%%%%%%%%%%%%%%%%%%%%%%%%%%%%%%%%%%%%%%%%%%%%%%%%%%%%

\section{Topological invariance 
and uniformity of constants for various recurrence conditions}
\label{recurr}

It is well-known that the Collet-Eckmann
property is an invariant of topological conjugacy,
and the fact that $\lambda_{c}(f_t)$ can be estimated uniformly
in $t$ for a smooth deformation $f_t$
of $f_0$
is explained, e.g., in \cite[Appendix]{bs3}.
Our argument requires more: We need a Benedicks-Carleson-type
condition of the form  \eqref{stBeC} or
\eqref{BeCxtra} and  uniform estimates
on the constants 
\begin{equation}
\label{csts}
\lambda_c(f_t)\, , \, H_0(f_t)\, , \, \gamma(f_t)\, ,
\mbox{ and also }
\sigma(f_t)\, , \, C_1(f_t)\, , \, c(\delta,f_t)\, , \, \rho(f_t)
\end{equation}
(recall  
Lemma \ref{est1a}), as $t$ varies.
The constant $\sigma(f_t)$ is bounded away from $1$ uniformly in small $t$, by
the proof of \cite[Theorem III.3.3]{dMvS}, in particular
the choice of $m$ and $\lambda$ there, and noting that
all $f_t$ have only repelling periodic orbits and are $S$-unimodal.
However, if $f_t$ is a smooth deformation of a
Benedicks-Carleson $S$-unimodal map,
we do not know how to estimate $\gamma(f_t)$ in general.

Lemma~\ref{est1aunif}, the main result of this section,
is proved in Subsection~\ref{uunif}:
It says that all  constants in \eqref{csts}
are uniform,
for deformations $f_t$ which satisfy the   TSR condition
 ~ \eqref{tsr}.  
 In Subsection~\ref{Lt}, we exploit a consequence of this uniformity
 which will play an important part in the proof of Theorem~\ref{linresp}: If $f_t$ is a deformation, one can use the same lower
 part of the tower for all
 operators $\widehat \LL_t$ with $|t|\le t_0$,
 up to some level depending on $t_0$.

In order to apply directly the results
of Nowicki, we shall assume  that $f$ is symmetric, i.e.,
\begin{equation}\label{symmetric}
f(x)=f(-x) \, .
\end{equation}

%%%%%%%
\subsection{Uniformity of constants}\label{uunif}

Recall that our definition of $S$-unimodal includes the condition
$f''(c)\ne 0$, that is, all our $S$-unimodal maps are quadratic.
Let $R_f(x)$ be the function from \eqref{Rf} in the definition of
the TSR condition.

\begin{proposition}[Uniform Collet-Eckmann condition \cite{luzzatto}]\label{unce} Let $f_0$ be a $C^3$
$S$-unimodal map satisfying the topological slow recurrence condition
\eqref{tsr}. Then there exist $\lambda_c >1$,
$\underline{\kappa}>0$, $K > 0$,  and $\epsilon > 0$ such that for every $S$-unimodal map $f$  in the topological class of $f_0$ such that $|f-f_0|_{C^3}< \epsilon$, we have   
\begin{equation}\label{topmet}  R_f(f^j(c))\geq -\underline{\kappa} \log |f^j(c)-c|\, ,
\quad \forall j \ge 0 \, , 
\end{equation}
and
\begin{equation}\label{topce}  |(f^j)'(f(c))|\geq K \lambda_c^j\, ,
\quad \forall j \ge 0 \, . 
\end{equation}

\end{proposition}
\begin{proof} Except for the explicit statement on the dependence of  $\underline{\kappa}$,   (\ref{topmet})   is Lemma~ 2 in \cite{luzzatto}. 
We  say that $f \in V(D, L, \theta)$ if
\begin{align*}
&D_f =\max_{x \in I} |f'(x)|<  D\, , 
\,\, L_f=\sup_{x \in I} \frac{|f(x)-f(c)|}{|x-c|^2} <  L\, ,\,\, 
\theta_f=\sup_{f(x)=f(y)} \frac{|x-c|}{|y-c|}< \theta\, .
\end{align*}

For $\epsilon$ small enough, we have $f \in V(D,L,\theta)$, with $D=2D_{g_0}$, $L=2L_{f_0}$ and $\theta=2\theta_{f_0}$. The proof of Lemma ~2 relies on Sublemmas~ 2.1 and ~ 2.2
in \cite{luzzatto}. The constants  $C$ and $\underline{\kappa}$ in 
\cite[Sublemma~ 2.1]{luzzatto}  depend only on $D$, $L$ and $\theta$. The constant  $N_\varepsilon$ in \cite[Sublemma~ 2.2]{luzzatto} depends only on the topological class of $f$. In the proof of Lemma ~2 in \cite{luzzatto}, since $f$ has a unique critical point we can take $\delta_0=|I|$ and $N_0=1$  in  (5) and (6) of \cite{luzzatto}. Moreover, we can find $\epsilon > 0$ such that 
$$ \inf_{|f-f_0|_{C^3}< \epsilon} \min\{   |f^i(c)-c| \ s.t. \ 0< i\leq \max\{N_0,N_\varepsilon\}\} > 0\, .
$$ 
This shows  (\ref{topmet}). 

Except for the explicit statement on the dependence of  $K$
and  $\lambda_c$,   (\ref{topce})  is Corollary~ 5.1  in \cite{luzzatto}. The proof of this result relies on  (\ref{topmet}) above, and on Lemmas~ 3, 4, 5 and Sublemma~ 5.1 in \cite{luzzatto}.  The estimates obtained in Lemma~ 3  depend only on the topological class of $f$. Given $T > 0$, we can find $\epsilon > 0$ such that 
$$
 \inf_{|f-f_0|_{C^3}< \epsilon} 
\min\{   |x-y| \ s.t. \ x \not=y, \ x, y \in \{ f^i(c)\}_{i\leq T}
\cup \{z \colon \ f^i(z)=c   \}_{ i\leq T}    \} > 0\, ,
$$
so we can see from the proof of \cite[Lemma~ 4]{luzzatto} that there exists  $\gamma(T)$ that satisfies the estimates obtained in Lemma ~4 for every $S$-unimodal map $f$ in the topological class of $f_0$ satisfying $|f-f_0|_{C^3} < \epsilon$. Sublemma~ 5.1 
in \cite{luzzatto} follows directly from Lemmas ~3 and~ 4 for  every $f$ satisfying the same conditions, with the same constants $\eta$ and  $\gamma$.  Finally the proof of the estimates in Lemma~ 5 in \cite{luzzatto}
depends only on the topological class of $f$, estimates in Sublemma ~2.2,  (\ref{topmet}) above and $D$.  \end{proof}

\begin{proposition}[Uniform Benedicks-Carleson type conditions]\label{ubc} Let $f_0$ be a $S$-unimodal map satisfying the topological slow recurrence condition
\eqref{tsr}. Then for every $\gamma > 0$ there exist  $H_0>0$
and $\epsilon > 0$ such that for every $S$-unimodal map $f$ in the topological class of $f_0$ such that $|f-f_0|_{C^3}< \epsilon$, we have 
$$|f^k(c)-c|\geq e^{-\gamma k} \, , \quad \forall k \ge H_0\, .$$  
\end{proposition} 

\begin{proof} Let $\underline{\kappa}$ and $\epsilon$ be as  in Proposition \ref{unce}. Choose $m_0$, $n_0$ large enough so that 
$$\frac{1}{n} \sum_{ \substack{1\leq j\leq n \\ R_{f_0}(f_0^j(c))\geq m_0}} R_{f_0}(f_0^j(c)) < \underline{\kappa} \gamma\, ,\forall n \ge n_0\, .
$$
Consequently, we have the same estimate for every map $f$ topologically conjugate to $f_0$, 
that is
$$\frac{1}{n} \sum_{ \substack{1\leq j\leq n \\ R_f(f^j(c))\geq m_0}} R_f(f^j(c)) < \underline{\kappa} \gamma \, ,\forall n \ge n_0\, .
$$
In particular, if $R_f(f^k(c))\geq m_0$ and $k\geq n_0$,  we have  
$$\frac{R_f(f^k(c))}{k} < \underline{\kappa} \gamma\, ,
$$
so by  (\ref{topmet}) we obtain 
$$-\frac{\log |f^k(c)-c|}{k} <  \gamma\, ,$$
so $|f^k(c)-c|\geq e^{-\gamma k}$. Since  $c$  is not periodic
for $f_0$, we can find $\eta, \epsilon > 0$ such that for each 
$S$-unimodal map $f$  such that $|f-f_0|_{C^3}< \epsilon$ and for every $x \in (c-\eta,c+\eta)$ we have  $|f^i(x)-c|> 0$ for  $1\leq i\leq 2m_0$.  In particular, 
$\mbox{dist}\, (\Omega_f,c) > \eta$,
where  
$$\Omega_f =\{ x \in I \colon \ R_f(x)<  m_0\}\, .$$
Let $H_0 > n_0$ be large enough such that $\eta > e^{-\gamma H_0}$. Then $|f^k(c)-c|\geq e^{-\gamma k}$ for every $k\geq H_0$. 
\end{proof}

We are going to use some results by Nowicki \cite{symmetric}. 
An interval $[a_1,a_2]$ is a  nice interval if $c \in (a_1,a_2)$ and
$f^j(a_i) \not\in (a_1,a_2)$, for every $j\geq 1$ and $i=1,2$.
We say that an interval $(c,b)$ is a $*(n)$ interval if $f^n$ is a diffeomorphism on $(c,b)$ and $f^n(b)=c$.

\begin{proposition}[Lemma 9 and Proposition 11 in \cite{symmetric}]\label{now2} 
Let $(c,b)$ be an $*(n)$ interval of a symmetric   $S$-unimodal map $f$ satisfying  (\ref{topce}).  Then $|f^n(c)-f^n(b)|> |c-b|$. Furthermore
$$|f^n(c)-f^n(b)|\geq K_1 \lambda_c^{n/4}|c-b|\, ,$$
where $K_1= (Km/4M)^{1/2}$, where  $K$ is as in  (\ref{topce}),
and $m$ and $M$ satisfy $m|x-c|\leq|f'(x)|\leq M|x-c|$.
\end{proposition}

\begin{proposition}[Proposition 13 in \cite{symmetric}]\label{now3} Let $f$ be a symmetric $S$-unimodal map $f$ satisfying  (\ref{topce}). 
Let $b \in [-1,1]$ be such that $f^n(b)=c$. Then
$$|(f^n)'(b)|\geq \rho^n\, ,$$
for every   \footnote{Note that $\rho$ is called $\lambda_T$ in \cite{symmetric}.}
\begin{equation}\label{losecontrol}
\rho \leq \min (\inf_n \inf_{(c,b)\, \rm{ an }\, *(n)\,  \rm{interval}} 
\biggl|\frac{f^n(c)-f^n(b)}{c-b}\biggr|^{1/n}, |f'(-1)|^{1/2})\, .
\end{equation}
(The right-hand-side above is $> 1$ by Proposition~\ref{now2}.)
\end{proposition}

\begin{proposition}\label{fe} Let $f_0$ be a $S$-unimodal map satisfying the topological slow recurrence condition \eqref{tsr}. Then for every $\beta \in (0,1)$ there exist  $\epsilon, \delta > 0$ and $\tilde{K} > 0$  with the following property:  Let $f$ be a symmetric $S$-unimodal map in the topological class of $f_0$ such that $|f-f_0|_{C^3}< \epsilon$, let  $[-q,q]\subset [-\delta,\delta]$ be a nice interval for $f$, and let 
$x \in [-1,1]\setminus[-q,q]$ be  such that $f^n(x) \in [-q,q]$ for some $n\geq 1$. 
Define
$$n_0(x)= \min\{n \geq 1 \ s.t. \ f^n(x) \in [-q,q]   \}\, .$$
Then there exist intervals $I_{n_0(x)} \subset J_{n_0(x)}$ such that 
\begin{itemize}
\item[1.] For every $y \in I_{n_0(x)}$ we have $n_0(y)=n_0(x)$  and $f^{n_0(x)}I_{n_0(x)}=[-q,q]$.
\item[2.]  The map  $f^{n_0(x)}\colon J_{n_0(x)}\rightarrow f^{n_0(x)}J_{n_0(x)}$ is a diffeomorphism, and each connected component of  $f^{n_0(x)}J_{n_0(x)}\setminus \{c\}$ is larger than $\tilde{K} q^{\beta}$. 
\end{itemize}
\end{proposition}

\begin{proof} The existence of $I_{n_0(x)}$ satisfying Claim 1 follows from  the fact that $[-q,q]$ is a nice interval. Let $[a,b]=J_{n_0(x)}\supset I_{n_0(x)}$ be the largest interval such that $f^{n_0(x)}$ is a diffeomorphism on $(a,b)$. In particular 
there are $n_a, n_b < n_0(x)$ such that $f^{n_a}(a)\in \{1,-1,c\}$ and 
$f^{n_b}(b)\in \{1,-1,c\}$. Suppose $f^{n_b}(b)=c$. We are going to show that 
$|f^{n_0(x)}b-c|\geq \tilde{K} q^{\beta}$. The proof of the analogous  statement for $a$  is similar.    By Claim 1   there is $d \in I_{n_0(x)}$ such that $f^{n_0(x)}(d)=c$ and, moreover,
$f^{n_b}(d)\not\in [-q,q]$, so either $[-q,c]\subset [f^{n_b}(d),c]=f^{n_b}[d,b]$ or
$[c,q]\subset [c,f^{n_b}(d)]=f^{n_b}[b,d]$. Since  $(f^{n_b}(d),c)$ is a $*(n_0(x)-n_b)$
interval \cite{symmetric}, by  Proposition ~\ref{now2}  and Proposition~ \ref{unce}, we have 
\begin{equation} \label{esum} |f^{n_0(x)}(b)-c|=|f^{n_0(x)-n_b}(c)-c|\geq K_1 \lambda_c^{(n_0(x)-n_b)/4}|c-f^{n_b}(d)|\geq K_1 \lambda_c^{(n_0(x)-n_b)/4}q \, ,
\end{equation}
where $K_1$ is uniform on a $C^3$ neighbourhood of $f_0$. Choose
$$0< \gamma < \frac{\beta \log \lambda_c}{4(1-\beta)}\, .
$$
Reducing this neighbourhood, if necessary, we have by Proposition \ref{ubc} that 
\begin{equation} \label{uniformbc} 
|f^{n_0(x)}(b)-c|=|f^{n_0(x)-n_b}(c)-c|\geq Ke^{-\gamma (n_0(x)-n_b)}\, .
\end{equation}
We have two cases. If $n_0(x)-n_b > -4(1-\beta)\log q/\log \lambda_c$ then,  by  (\ref{esum}), we easily obtain $|f^{n_0(x)}(b)-c|\geq K_1 q^{\beta}.$
Otherwise $n_0(x)-n_b \leq  -4(1-\beta)\log q/\log \lambda_c$, so by  (\ref{uniformbc}), we get 
$$
|f^{n_0(x)}(b)-c|=|f^{n_0(x)-n_b}(c)-c|\geq Ke^{-\gamma (n_0(x)-n_b)}
\geq Ke^{- \frac{\beta \log \lambda_c}{4(1-\beta)}(n_0(x)-n_b)} \geq K q^\beta\, .
$$
Choose $\tilde{K}= \min (K,K_1)$.
If $|f^{n_b}(b)|=1$, then $f^{n_0(x)}(b)=-1$. Choose $\delta$ such that  $\delta^{-\beta} \geq \tilde{K}$. Then
$$
|f^{n_0(x)}(b)-c|=1\geq \tilde{K} \delta^\beta\geq \tilde{K} q^\beta\, .
$$
\end{proof}

\begin{corollary}[Uniformity of $C_1$ and $\rho$] \label{inv}
Let $f_0$ be a symmetric $S$-unimodal map satisfying the topological slow recurrence condition \eqref{tsr}. There exist $\rho > 1$ and $\epsilon > 0$  with the following property: For every $C_1 \in (0,1)$ there exists $\delta > 0$ 
so that, for every symmetric  $S$-unimodal map $f$  in the topological class of $f_0$ 
such that $|f-f_0|_{C^3}< \epsilon$, and for every nice interval $[-q,q]$ of $f$ 
such that $q < \delta$, if $x \not\in [-q,q]$ and $n\geq 1$ is the first entrance time 
of $x$ in $[-q,q]$, then 
$$|(f^n)'(x)|\geq C_1\rho^n\, .
$$
\end{corollary}

\begin{proof} By Proposition \ref{unce}, we can find  $\epsilon_0 > 0$ and $\rho> 1$ such that Proposition ~\ref{now3} holds for every $f$ such that $|f-f_0|_{C^3}< \epsilon_0$, with $f$ in the topological class of $f_0$. 
Take  $\beta=1/2$,  and let $\epsilon<\epsilon_0$, $\delta$ be as in Proposition ~\ref{fe}. Reducing $\delta$ if necessary, we have that if $q < \delta$ then each connected component of $f^n(J_n(x))\setminus [-q,q]$ is far larger than $q$. In particular by the Koebe lemma 
$$
\frac{(f^n)'(z)}{(f^n)'(w)}<C_1^{-1} \, , \, \forall z,w \in I_n(x)\, .
$$ 
But there exists $b \in I_n(x)$ such that $f^n(b)=c$, so  $(f^n)'(b) \geq \rho^n$. We conclude that $|(f^n)'(x)|\geq C_1\rho^n$.
\end{proof}

Finally, we shall need the following result:

\begin{corollary} [Uniformity of $c(\delta)$ and $\sigma$]\label{mane}
Let $f_0$ be a symmetric
$S$-unimodal map satisfying the topological slow recurrence condition
\eqref{tsr}. There exists
\footnote{The analogue of $\sigma$ is called $\lambda_M$
in \cite{No85}.} $\sigma > 1$ such that for every $\delta>0$ there exist
 $c(\delta) > 0$ and  $\epsilon > 0$  with the following property: For
every symmetric  $S$-unimodal map $f$  in the topological
class of $f_0$ such that $|f-f_0|_{C^3}< \epsilon$,  if $|f^i(x)|> \delta$
for $0\leq i< n$ then
$$|(f^n)'(x)|\geq c(\delta) \sigma^n\, .$$
\end{corollary}

\begin{proof} 
By Proposition \ref{unce}, we can find  $\epsilon_0 > 0$
and $\rho> 1$ such that Proposition \ref{now3} holds for every
$f$ such that $|f-f_0|_{C^3}< \epsilon_0$, with $f$ in the topological
class of $f_0$. Using the same argument as in Proposition 3.9 in
\cite{No85}, we can show that for every periodic point $q$ such that
$f^n(q)=q$ we have $|(f^n)'(q)|\geq \rho^n$. Note that since $c$ is
recurrent by $f_0$, there exists a sequence of periodic points for $f_0$
converging to $c$. So given $\delta > 0$ there exists a periodic point
$p$ for $f_0$ such that $|p| < \delta$. Let $n_0$ be the prime period of
$p$. There exists $\epsilon_1 < \epsilon_0$ such that every map $f$
such that $|f-f_0|_{C^3}< \epsilon_1$ has an analytic continuation $p_f$
for $p$  such that $|p_f|< \delta$ and
$$\eta_{per}=\inf_{|f-f_0|_{C^3}< \epsilon_1} |p_f|>0\, .
$$
 Without loss of generality, we can assume that $|f^i(p_f)| \geq
|p_f|$ for every $i$. So $[-p_f,p_f]$ is a nice interval. Let $x
\not\in [-\delta,\delta]$ be  such that $|f^i(x)|> \delta$ for $0\leq
i< n$. If $f^n(x) \in [-p_f,p_f]$ we can use Corollary \ref{inv} to
conclude that $|(f^n)'(x)|\geq C_1\rho^n$. So assume that $f^n(x)
\not\in [-p_f,p_f]$. Let $(a,b)$ be the largest interval such that $x
\in (a,b)$ and $f^i(y) \not\in [-p_f,p_f]$ for every $0\leq i\leq n$
and $y \in (a,b)$. In particular, $f^n$ is a diffeomorphism on $(a,b)$,
and there exist $n_a, n_b\leq n$ such that
$|f^{n_a}(a)|,|f^{n_b}(b)|\in \{ |p_f|,1\}$. Without loss of
generality, we can assume that 
$|f^i(a)|, |f^{j}(b)|\not\in \{|p_f|,1\}$ 
for every $i < n_a$, $j < n_b$.  If $f^{n_a}(a)\in
\{-1,1\}$, then indeed $a \in \{-1,1\}$, so $|(f^n)'(a)|=|f'(-1)|^n$.
We have a similar statement for $b$. Otherwise either
$f^{n_a}(a)$ (respectively $f^{n_b}(b)$) or $-f^{n_a}(a)$ (respectively $-f^{n_b}(b)$)
is a periodic point with period $n_0$. Then $n_a$ and $n_b$ are the
first entry times of $a$ and $b$ in $[-p_f,p_f]$. By Corollary~\ref{inv}, we have
$$
|(f^{n_a})'(a)|\geq C_1\rho^{n_a} \mbox{ and }
 |(f^{n_b})'(b)|\geq C_1\rho^{n_b}
\, .
$$
Since $p_f$ is a periodic point of period $n_0$,
$|(f^{n_0})'(p_f)|\geq\rho^{n_0}$ and $f$ is symmetric and
quadratic, we have
\begin{align*}
|(f^{n-n_a})'(f^{n_a}(a))|&\geq 
\rho^{n-n_a-n_0} \min \{|f'(f^i(p_f))|, \ 0\leq i < n_0   \}^{n_0}\\
&\geq C^{n_0}|p_f|^{n_0}\rho^{n-n_a-n_0}\, ,
\end{align*}
so
$$
|(f^{n})'(a)|\geq C^{n_0}|p_f|^{n_0}\rho^{n-n_0} 
\geq C^{n_0}\eta_{per}^{n_0}\rho^{-n_0}\rho^{n} 
= c(\delta) \rho^{n}\, .
$$
We can obtain similarly $|(f^{n})'(b)|\geq c(\delta) \rho^{n}$.
In any case
$$
\min(|(f^{n})'(a)|,|(f^{n})'(b)|)\geq 
\min (1, c(\delta)) \cdot  \min (\rho,|f'(-1)|)^n\, .
$$
By the minimum principle
$$|(f^{n})'(x)| \geq \min(1, c(\delta))\cdot \min (\rho,|f'(-1)|)^n\, .
$$
So choose $\sigma= \min (\rho,|f'(-1)|)> 1$. 
\end{proof}

Summarising the results of this section, we have proved:

\begin{lemma}[Uniformity of constants in topological classes of
TSR maps]\label{est1aunif}  If $f_0$ is a symmetric 
$S$-unimodal $(\lambda_c(f_0),H_0(f_0))$-Collet-Eckmann map satisfying
topological slow recurrence \eqref{tsr}, 
for every $C_1 \in (0,1)$ there exists $\lambda_c \in (1, \lambda_c(f_0))$ so that for
any $\gamma>0$ there exists $H_0 > H_0(f_0)$ so that for
each $\rho \in (1, \lambda_c^{1/2})$,
there exists $\sigma >1$ and
$\delta_0>0$ so that for every 
$\delta\in  (0, \delta_0)$ there exist  $c(\delta)>0$
and $\epsilon >0$ so that the following holds
for each symmetric  $S$-unimodal  map $f$ topologically conjugated to $f_0$ and so that
$|f-f_0|_{C^3}< \epsilon$:
 
The map $f$ is $(\lambda_c,H_0)$-Collet-Eckmann and satisfies 
\eqref{BeCxtra} for $\gamma$, $\lambda_c$, and $H_0$.
For any
$y \in I$, if $j\geq 0$ is  minimal satisfying  $|f^j(y)|\leq \delta$, then 
 \begin{equation}
\label{estlaeq2'}|(f^j)'(y)|\geq C_1  \rho^j \, ,
\end{equation}
for any $x \in I$, if $j \ge 1$  is such that $|f^k(x)|>\delta$ 
for all $0 \le k < j$, then
\begin{equation}
\label{estlaeq1'} |(f^i)'(x)|\geq c(\delta) \sigma^i\, , \, \forall 0 \le i \le j 
\, . 
\end{equation}
\end{lemma}

Comparing the above result to Lemma~\ref{est1a} we emphasize
that {\it we do not claim} that $\lambda_c$ can be taken arbitrarily
close to $\lambda_c(f_0)$ or $H_0$ close to $H_0(f_0)$, where
$\lambda_c(f_0)$, $H_0(f_0)$ are the best possible constants
for $f_0$.
So \eqref{BeCxtra} cannot be viewed strictly as a Benedicks-Carleson assumption. (This is mostly because of the infimum in the right-hand-side
of \eqref{losecontrol} from Proposition ~\ref{now3}.)
However, this does not matter since we are assuming the much stronger TSR assumption in any case (see also
\eqref{bd2} below), which implies that we can take $\gamma$ arbitrarily close to $0$
after $\lambda_c$ has been fixed.
The advantage of the notation introduced in   Lemma~\ref{est1aunif} is that
we can  use the estimates from Sections ~\ref{alphabdedproof}
and ~\ref{spectralstuff} directly, with the
same notation for the constants, for deformations
of  maps $f_0$ satisfying the assumptions of 	Lemma~\ref{est1aunif}.

%%%%%%%%%%%%%%%%%%%%%%%%%%%%%%
\subsection{Transfer operators $\widehat \LL_t$, $\widehat \LL_{t,M}$
for a (TSR) smooth deformation $f_t$}
\label{Lt}

If $f_t$ is a $C^1$ one-parameter family  of   $S$-unimodal 
symmetric Collet-Eckmann maps $f_t$,
with a  non preperiodic critical point, 
Lemma~\ref{est1aunif} implies that all $f_t$ satisfy estimates
for uniform parameters $\lambda_c$ and $H_0$, and
satisfy the strengthened Benedicks-Carleson
condition  \eqref{BeCxtra} for some
$\gamma$. We can
associate
a tower $\hat f_t : \hat I_t \to \hat I_t$ to each $f_t$, choosing
small $\delta_t$ and intervals $B_{k,t}$ by
using the parameters $\lambda_c$, $H_0$, $\gamma$
as in Subsection ~ \ref{tower}, replacing $c_k$ by $c_{k,t}$. 
Then,
we can define spaces  $\BB^{H^1_1}_t$, and an operator $\widehat \LL_t$
in Subsection ~ \ref{gap}, replacing $f^k$
by $f^k_t$ and $c_k$ by $c_{k,t}$
 in \eqref{defban} and the definition
 of $\xi_{k,t}$. We summarize first the results which
follow from applying
Proposition~ \ref{mainprop} and Proposition~\ref{acip}
to each  $f_t$, in order to fix notation
(note however that 
we shall modify slightly
the lower parts of the tower maps $\hat f_t$ in Proposition~\ref{bottomok}):
There is $\Theta<1$ so that each operator
$\widehat \LL_{t}$ has essential spectral radius bounded by $\Theta$ on
$\BB_{t}=\BB_t^{H^1_1}$. Outside of a disc of radius $\theta_t<1$
 the spectrum of
$\widehat \LL_{t}$ on $\BB_t$  consists in a simple eigenvalue at $1$, with
a nonnegative eigenfunction $\hat \phi_t$, so that $(\hat\phi_t)_0$ belongs to
$H^2_1$.
Define $\Pi_t : \BB_t \to L^1(I)$ by
\begin{align}\label{defpit}
\Pi_t(\hat \psi)(x)&=\sum_{k \ge 0
, \varsigma\in \{+,-\}}
 \frac{\lambda^k}{|(f^{k}_t)'(f^{-k}_{t,\varsigma}(x))|} \psi_k(f^{-k}_{t,\varsigma}(x)) 
 \chi_{k,t}(x)
\, ,
\end{align}
where $\chi_{k,t}$ is defined like $\chi_k$ (see
Proposition~\ref{acim}), replacing $f^k$
by $f^k_t$ and
$c_k$ by $c_{k,t}$.
The fixed point of the dual of 
$\widehat \LL_t$ is  the  nonnegative measure 
$\nu$ on $\widehat I_t$,
absolutely continuous with respect to Lebesgue on 
$\widehat I_t$ whose density is $w(x,k)$.
If we normalise by requiring 
$\nu(\hat \phi_t)=1$,   the invariant density of
$f_t$ is just 
$\phi_t =\Pi_t (\hat \phi_t)$.
Lemma ~ \ref{truncspec} also holds
for $\widehat \LL_{t,M}$, using the weak norm
$\|\cdot \|_{\BB^{L^1}}$.
This gives
$\hat \phi_{t,M}$, $\nu_{t,M}=\nu_M$, and $\kappa_{t,M}$.
If $\sigma(f_t)$, $C_1(f_t)$ and
$c(\delta, f_t)$ from Lemma~\ref{est1a} applied to
$f_t$ are  uniform in $t$, then all objects constructed are uniform in $t$
(including the $H^2_1$ norms of $(\hat \phi_t)_0$ and
$(\hat \phi_{t,M})_0$ for $r\in (1,2)$).

There is of course some flexibility in choosing the intervals
$B_{k,t}$ and the functions $\xi_{k,t}$.
It is tempting, in order to get conjugated tower dynamics $\hat f_t : \hat I_t \to \hat I_t$, to choose  $B_{k,t}=h_t(B_{k,0})$
and $\xi_{k,t}=\xi_{k,0}\circ h_t^{-1}$,
where the homeomorphisms $h_t$ are given by
Lemma ~ \ref{robustchaos}. Then, in order to prove
Theorem~\ref{linresp} on linear response,
one would need additional information on the $h_t$
(for example, but not only,
the fact that $\partial_t h_t(x)=\alpha(x)$
at all points $x$). 
We shall work  instead
with truncated operators $\LL_{t, M}$,  disregarding the
top part of the tower via Lemma~\ref{truncspec}, and
{\it artificially forcing the lower
parts of the towers} associated to the various $f_t$
to {\it coincide.}
This is going to be possible in view of the following consequence
of  Lemmas~\ref{first} and ~\ref{first'''}:

\begin{proposition}[Controlling the truncated tower]\label{bottomok}
Let $f_t$ be a $C^1$  deformation  of  $S$-unimodal maps $f_t$
satisfying the Benedicks-Carleson
condition \eqref{BeCxtra} for $\gamma_0=\gamma$. 
Let $\hat f_0 : \hat I \to \hat I$ be a tower associated
to $f_0$ as in Section~ \ref{tower},
for some $\delta>0$
and $3\gamma_0/2 < \beta_1 < \beta_2 < 2 \gamma_0$.
Let $\alpha_s$ be the  solution of the TCE 
\eqref{tce} for $f_s$
and $v_s= \partial_t f_t|_{t=s}$,
with $\alpha_s(c)=0$, given by
\footnote{Recall Remark ~ \ref{defhor}.} Theorem ~ \ref{alphabded}.
Fix 
$$
\frac{3\gamma_0}{2} <\tilde \beta_1 < \beta_1 < \beta_2 < \tilde \beta_2 < 2 \gamma_0 \, .
$$
Then,  for any $M \ge 1$,
and for any $t$ so that 
\begin{equation}\label{tcond}
 \sup_{ |s| \le t} |\alpha_s(c_k)| |t| < \min( (e^{- \tilde \beta_1 k}-e^{- \beta_1 k}) ,
 (e^{-  \beta_2 k}-e^{- \tilde \beta_2 k}))\, ,
1 \le k \le M \, ,
\end{equation}
one can construct the  tower maps $\hat f_t:\hat I_t\to
\hat I_t$, the Banach spaces $\BB_t^{H^1_1 }$,
$\BB^{L^1}$,   and the transfer
operators $\widehat \LL_t$, using 
parameters $\delta_t>0$, intervals $B_{k,t}$ admissible for
$\tilde \beta_1$, and $\tilde \beta_2$,
and  smooth cutoff functions $\xi_{k,t}$  
in such a way as to ensure
$$
\delta_t= \delta\, , \,\, \, 
% B_{k,t}=B_k\, , \, \, 
\xi_{k,t}= \xi_k\, , \, \, \forall k \le M \, ,
$$
and, in addition,
so that all results of Section~\ref{spectralstuff} hold for $\widehat \LL_t$.
\end{proposition}

If $f_0$ enjoys TSR then, up to taking smaller $\epsilon$,
Lemma~\ref{first'''}
implies that
$$\sup_{|x|\le \delta, |s|\le \epsilon} |\alpha_s(x)| |\epsilon|< \infty\, ,$$
so that we can exploit the above proposition.

\begin{proof}[Proof of Proposition ~ \ref{bottomok}]
Recall $h_t$ as given by \eqref{starstar} and recall
Lemma~\ref{first}.
By the proof of Lemma \ref{first'''},  we have
\begin{align*}
|a_k-h_t(c_k)|&\le
|a_k-c_k| +|h_t(c_k)- h_0(c_k)|\\
&\le e^{-\beta_1 k} +  \sup_{|s|\le |t|} |\alpha_s(c_k)| |t| \, ,
\end{align*}
and
\begin{align*}
|a_k-h_t(c_k)|&\ge
|a_k-c_k| -|h_t(c_k)- h_0(c_k)|\\
&\ge e^{-\beta_2 k} -  \sup_{|s|\le |t|} |\alpha_s(c_k)| |t| \, .
\end{align*}
The claim of Proposition ~ \ref{bottomok} follows.
\end{proof}

%%%%%%%%%%%%%%%%%%%%%%%%%%%%%%%%%%%%%%%%%%%%%%%%%%%
\section{Proof of  linear response}
\label{finalproof}

In this section, we prove Theorem ~ \ref{linresp}.
Let $f_t$ satisfy the assumptions of  the theorem.
We suppose in addition that the critical point is not preperiodic
(the proof is much easier if it is).
Applying Lemma ~\ref{est1aunif}, we fix $\epsilon >0$
and constants $\gamma$, $\lambda_c$, $\sigma$, $C_1$, $\rho$,
$\delta$, and $c(\delta)$ and we choose $3\gamma/2<\beta_1<\beta_2<2\gamma$.
By Lemma ~\ref{est1aunif},
we may assume that the strong Benedicks-Carleson
condition \eqref{BeCxtra} holds, and in some places 
in the proof below 
we may require a stronger upper bound on $\gamma$.

Constructing a tower, Banach space, and transfer operator
for each $f_t$ as in Section~\ref{spectralstuff},
the invariant density of $f_t$  can  be written as
$\phi_t =\Pi_t( \hat \phi_t)$,
where  $\Pi_t$ was defined in
\eqref{defpit},  and where $\hat \phi_t$ is the nonnegative
and normalised fixed point of $\widehat \LL_t$ on $\BB_{t}^{H^1_1}$
given by Proposition \ref{acip} applied to $f_t$.
It will be convenient 
to work with truncated transfer operators
$\widehat \LL_{t,M}$, recalling Subsection ~ \ref{trunk}, 
in particular Lemma ~ \ref{truncspec}, which gives  $\hat \phi_{t, M}$.  We shall in fact require the 
lower part of
the towers of $f_t$ for small enough $t$, up to
$M=M(t)$ as given by
Proposition~\ref{bottomok} to coincide with that of $f_0$.

When the meaning is clear, we shall remove $0$
from the notation, writing,
e.g., $\Pi $, $\hat \phi$, and $\hat \phi_M$, instead
of $\Pi_0$, $\hat\phi_0$, and $\hat \phi_{0,M}$.

We start with the decomposition
\begin{equation}\label{decc}
\phi_t - \phi=\Pi_t (\hat \phi_t - \hat \phi_{t, M}) +
 \Pi (\hat \phi_{M}-\hat \phi)
+ \Pi_t (\hat \phi_{t, M})- \Pi (\hat \phi_{ M}) \, .
\end{equation}
Note that  if $f^k|_{[c,y]}$ is injective then
\begin{equation}\label{glu2}
\int_{c_k}^{f^k(y)} \frac{|\psi_k(f^{-k}_+(z))|}{|(f^k)'(f^{-k}_+(z))|}\, dz
=\int_c^y |\psi_k(x)| \, dx \, .
\end{equation}
Lemma~\ref{truncspec} implies that, for large enough $M$ (uniformly in $t$)
\begin{align}\label{glu1}
 \|\Pi_t(\hat \phi_t - \hat \phi_{t, M})\|_{L^1(I)}
&\le C 
\| \hat \phi_t - \hat \phi_{t, M}\|_{\BB^{L^1}}\le C  \tau_M^{\eta}  \, ,
\forall |t|< \epsilon  \, .
\end{align}
Fix $\zeta >0$,  then \eqref{glu1}  implies
\begin{align}
\nonumber 
\max(\| \Pi_t (\hat \phi_t - \hat \phi_{t, M})\|_{L^1(I)},&
\| \Pi(\hat \phi - \hat \phi_{M})\|_{L^1(I)})
\le C |t|^{1+\zeta} \, , \\
\label{bd1}&\quad\forall t \mbox{ so that }\, 
(\tau_M^{\eta})^{\frac{1}{1+\zeta} } < |t|< \epsilon \, .
\end{align}

It is now sufficient to estimate the third term in the right-hand-side
of \eqref{decc} for  $t$ and $M=M(t)$ satisfying
\eqref{bd1}. 

For this, in order to apply
Proposition ~ \ref{bottomok}, and noting that the
right-hand-side of \eqref{tcond} is $\ge C e^{-2\gamma M}$, 
we want $t$ and $M$ to satisfy
\begin{equation}\label{bd2'}
\sup_{k,|s|< \epsilon} |\alpha_s(c_k)||t| < C e^{- 2\gamma M} \, .
\end{equation}
(Recall that
$\sup_{k,|s|< \epsilon} |\alpha_s(c_k)|\le L$
by Lemma~\ref{first'''}.)
In several places  below 
we shall require a  stronger version of \eqref{bd2'}, of the form
\begin{equation}\label{bd2}
|t| < e^{-\Gamma \gamma M}\, ,
\end{equation}
where $\Gamma>2$ is large (but uniformly bounded over the
argument).
Since
$\tau_M < \lambda_c^{-M/2}e^{3M\gamma}$
(recalling Lemma~\ref{truncspec})
and $\lambda <e^{\gamma}$
(by \eqref{48}), we see that \eqref{bd1} and \eqref{bd2} are compatible
if $\epsilon$ is  small enough and
\begin{equation}\label{finalcond}
\frac{1}{2} \log \lambda_c > \gamma \bigl (  \frac{1+\Gamma(1+\zeta)}{\eta} + 3
\bigr )\, .
\end{equation}
By our TSR assumption
and Lemma~\ref{est1aunif}, we may indeed require that $\gamma$ is small
enough for the above Benedicks-Carleson
condition to hold, even
if  $\Gamma>2$ is large.

We shall call pairs $(M,t)$ so that
$|t|< \epsilon$ and  \eqref{bd1} and \eqref{bd2}
hold {\it admissible pairs.}
In the remainder of this section, $(M,t)$ will always be an
admissible pair, and we shall work with the towers and operators
given by Proposition~\ref{bottomok} for a given such pair.

The key decomposition for an admissible pair is then 
\begin{align}\label{decompfinal}
 \Pi_t (\hat \phi_{t, M})- \Pi (\hat \phi_{M})
&=\Pi_t ( \hat \phi_{t,M} -\hat \phi_{M}  )
+\Pi_t  (\hat \phi_{M}) -\Pi  (\hat \phi_{M}) \, .
\end{align}

Before we start with the proof, let us briefly sketch it:
The term  $ \hat \phi_{t,M} -\hat \phi_{M}$, will 
be handled using  spectral perturbation-type methods.
This is the content of Steps ~1 and ~2  below, the outcome
of which are claims \eqref{claim1}, \eqref{formula1},
and \eqref{formula1''}.
(Horizontality is used here to get uniform
estimates, in view of Proposition~\ref{bottomok}
but also, e.g., in Lemma~\ref{goodlemma}.)

The other term requires the analysis of $\Pi_t - \Pi_0$. 
This will produce derivatives of the ``spikes,"
i.e., of functions of the type $(x-c_k)^{-1/2}$
(recall the definition of $\Pi_t$ and see Lemma~\ref{rootsing}). 
Since  $\partial_t(x-h_t(c_k))^{-1/2}$
is not integrable, this will require working with
$\int A (\Pi_t  \hat \phi_t -\Pi_0  \hat \phi_t ) \, dx$,
with $A$ a $C^1$ function and integrating by parts, as
well as using again horizontality.
We perform this analysis in Step~3 of the proof, which yields
\eqref{formula2}.

\medskip

{\bf Step 1: The first term of \eqref{decompfinal}: perturbation theory
via resolvents}

(Recall that $(M,t)$ is an admissible pair.)
In order to get a formula for the limit
as $t\to 0$ (in a suitable norm) of the  first term of \eqref{decompfinal} divided by $t$, we shall first analyse  $(\hat \phi_{t,M}-\hat \phi_M)/t$, and then
see how $\Pi_t$ enters in the picture.

Since $\nu_{M}(\hat \phi_{M})=1$, 
we have
$$
\hat \phi_{t,M} \nu_{t,M}(\hat \phi_{M})-\hat \phi_{M} 
=
\hat \phi_{t,M}-\hat \phi_{M} +
\hat \phi_{t,M}( \nu_{t,M}(\hat \phi_M)- \nu_M(\hat \phi_M))\, .
$$
Now,  Lemma~\ref{truncspec} applied to
$f$ and $f_t$ implies that
\begin{align*}
| \nu_{t,M}(\hat \phi_M)- \nu_M(\hat \phi_M)|
&\le 2 \max( \|\nu_{t,M}-\nu_t\|_{(\BB^{H^1_1})^*}, \|\nu_M-\nu\|_{(\BB^{H^1_1})*}) \|\hat \phi_M\|_{\BB^{H^1_1}}\\
&\le C \tau_M^{\eta} \|\hat \phi_M\|_{\BB^{H^1_1}}\, .
\end{align*}
Our choices imply that $C \tau_M^{\eta} =O(|t|^{1+\zeta})$ while
  $\|\hat \phi_M\|_{\BB^{H^1_1}}$ is uniformly bounded, e.g., by the proof of Lemma ~\ref{truncspec}.
Therefore, to study 
$\hat \phi_{t,M}-\hat \phi_M$, it
suffices to estimate
$\hat \phi_{t,M} \nu_{t,M}(\hat \phi_{M})-\hat \phi_{M}$, that we
shall express as
a difference of spectral projectors.

Next, set  
$$
\widehat \QQ_{t,M}=\widehat \QQ_{t,M}(z)= z-\widehat \LL_{t,M}\, ,
\quad
\widehat \QQ_{M}=\widehat \QQ_{M}(z)= z-\widehat \LL_{M}\, ,
$$
recall 
$\PPP_M$ from the proof
of Lemma~\ref{truncspec}, and denote by 
$$
\PPP_{t,M}(\hat \psi)=
\hat \phi_{t,M}  \nu_{t,M}(\hat \psi)\, ,
$$ 
the  spectral projector corresponding to the maximal eigenvalue
of $\widehat \LL_{t,M}$.  Using
$$
\widehat \QQ_{t,M}^{-1}- \widehat \QQ_M^{-1}= \widehat \QQ_{t,M}^{-1} (\widehat \LL_{t,M} -\widehat \LL_M)
 \widehat \QQ_{M}^{-1}\, ,
 \mbox{ and  }
 \QQ_M^{-1} (\hat \phi_{M})= 
\frac{\hat \phi_{M}}{z-\kappa_{M}}\, ,
$$
we rewrite
$\hat \phi_{t,M}\nu_{t,M}(\hat \phi_M)-\hat \phi_{M}  
=(\PPP_{t,M}-\PPP_M)(\hat \phi_{M})$
as follows:
\begin{align}\label{idea}
\hat \phi_{t,M}\nu_{t,M}(\hat \phi_M)-
\hat \phi_{M} &=
- \frac{1}{2 i \pi} \oint  \frac{\widehat \QQ_{t,M}^{-1}(z)} {z-\kappa_{M}} 
(\widehat \LL_{t,M} -\widehat \LL_M) (\hat \phi_{M})\, dz\\
\nonumber &=
(\kappa_{M}  -\widehat \LL_{t,M})^{-1} (\id-\PPP_{t,M})(\widehat \LL_{t,M}-\widehat \LL_M) (\hat \phi_{M}) \, ,
\end{align}
where the  contour is a circle centered at $1$,
outside of the disc of radius $\max(\theta_0, \theta_t)$
(using the notation from Subsection~ \ref{Lt}).

We are going to use again  the arguments in \cite{kellerliverani}.
By uniformity of the constants in
Lemma~\ref{est1aunif},  the Lasota-Yorke estimates
\footnote{We use that the constant
$c(\delta, f_t)$ associated to $f_t$ by Lemma ~ \ref{est1aunif}
does not depend on $t$
and that 
$\inf_{\delta \ge \delta_0} c(\delta)>0$ for any $\delta_0>0$.} in
the proofs of Proposition ~ \ref{mainprop} 
and Lemma ~ \ref{truncspec} say that there exist
$\epsilon >0$ and $C\ge 1$ 
so that,  for  all  $|t|\le \epsilon$, all $M$, 
all $j$,
\begin{align}\label{hand0}
&\|\widehat \LL_{t,M}^j(\hat \psi)\|_{\BB^{L^1}}  \le C \|\hat \psi\|_{\BB^{L^1}}
\, , \quad
\|\widehat \LL_{t,M}^j (\hat \psi )\|_{\BB^{H^1_1}} \le C  \Theta^{-j}\|\hat \psi\|_{\BB^{H^1_1}}
+C\|\hat \psi\|_{\BB^{L^1}}\, , \\
\label{hand0'}
& \|\widehat \LL_{M}^j(\hat \psi)\|_{\BB^{L^1}}  \le C \|\hat \psi\|_{\BB^{L^1}}
\, , \quad
\|\widehat \LL_{M}^j (\hat \psi )\|_{\BB^{H^1_1}} \le C  \Theta^{-j}\|\hat \psi\|_{\BB^{H^1_1}}
+C \|\hat \psi\|_{\BB^{L^1}}\, .
\end{align}

In Step~2   we shall find 
$\widetilde C \ge 1$ and $\tilde \eta >0$
so that for each admissible pair  $(M,t)$ 
\begin{equation}\label{missarg}
\|\widehat \LL_{t,M}(\hat \psi)-\widehat \LL_M(\hat \psi)\|_{\BB^{L^1}}
\le \widetilde C |t|^{\tilde \eta}  \|\hat \psi\|_{\BB}\, , \, 
\forall
\hat \psi \in \BB \, .
\end{equation}
We are not exactly in the setting of \cite{kellerliverani}, since
we have a ``moving target" $\widehat \LL_{M(t)}$ as
$t\to 0$.
However, since the right-hand-side of \eqref{missarg} does not
depend on
$M$, 
setting
$$
\NN_{t,M}:=(\kappa_{M} -\widehat \LL_{t,M})^{-1}(\id-\PPP_{t,M})- 
(\kappa_{M} -\widehat \LL_M)^{-1}(\id -\PPP_{M})\, ,
$$
then
\eqref{hand0}--\eqref{hand0'} and (\ref{missarg})
imply, by a small modification of the
proofs of \cite[Theorem~1, Corollary ~1]{kellerliverani}, that 
there exist $\widehat C\ge 1$ and  $\hat\eta >0$   so that for 
all admissible pairs $(M,t)$
\begin{equation}\label{limm}
|| \NN_{t,M} \|_{\BB^{H^1_1}}
\le \widehat C \, ,
\quad \| \NN_{t,M}(\hat \psi) \|_{\BB^{L^1}}
\le \widehat C |t|^{\tilde \eta \hat \eta} \|\hat \psi\|_{\BB^{H^1_1}}\, .
\end{equation}

In Step 2, we shall show that  there exist
$C>0$ and $\DD_M\in \BB=\BB^{H^1_1}$ 
with 
$$
\int_I \DD_{M,0}\, dx =0\, ,\, \,
\DD_{M,k}=0\, , \forall k \ge 1\, ,  \mbox{ and }
\| \DD_{M} \|_{\BB }\le C  
%\, ,  \quad \| \DD_M\|_{\BB^{L^1}} \le C
 \, ,
$$
and $\tilde \zeta>0$, so that
for all admissible pairs $(M,t)$
\begin{equation}\label{almostder}
\|\widehat \LL_{t,M}(\hat \phi_{M})-\widehat \LL_M(\hat \phi_M) - t  \DD_M \|_{\BB}\le  C |t|^{1+\tilde \zeta} \, .
\end{equation}
Writing 
$$
(\kappa_{M}-\widehat \LL_{t,M})^{-1}(\id - \PPP_{t,M})
= 
\NN_{t,M} 
+(\kappa_{M}-\widehat \LL_M)^{-1}(\id -\PPP_M)\, ,
$$
we see that  \eqref{almostder} together with
\eqref{newKL'}
from the proof of Lemma~\ref{truncspec} and
\eqref{idea} 
imply
\begin{align}
\nonumber
\hat \phi_{t,M}-\hat \phi_{M} &=
[\NN_{t,M} +(\kappa_{M}-\widehat \LL_M)^{-1}(\id - \PPP_M)](t  \DD_M + O_{\BB}(|t|^{1+\tilde \zeta})) \\
\label{lllast} &= t   \NN_{t,M}(\DD_M) +t(\kappa_{M}-\widehat \LL_M)^{-1}(\id - \PPP_M)( \DD_M)
+ \Delta O_{\BB^{H^1_1}}(|t|^{1+\tilde \zeta}) \, .
\end{align}

Note for further use (in \eqref{notenough} below) that, since 
$\|\hat \psi\|_{C^0}:=\sup_k \sup |\psi_k|\le C \|\hat \psi\|_{H^1_1}$, 
the bound
\eqref{lllast} with the first inequality in \eqref{limm} imply 
\comment{$L^p$ estimate for $p>1$ would be enough}
\begin{equation}\label{notenough'}
\|\hat \phi_{t,M}-\hat \phi_{M} \|_{C^0} =O(t)\, ,
\mbox{ as $t\to 0$, uniformly in admissible
pairs $(M,t)$.}
\end{equation}

 \smallskip
Recalling from \eqref{hatY} the definition 
$\hat Y$, we shall see in Step 2 that 
the following expression defines an element of $\BB^{H^1_1}$
\begin{equation}\label{dammit'}
\DD=-(\TT_0(\widehat \LL (\hat Y \hat \phi))'\, ,
\end{equation}
 and that, in addition
$ \lim_{M \to \infty}\| \DD-\DD_M\|_{H^1_1}=0$.
More precisely, there is
 $\zeta'>0$ so that for admissible pairs $(M,t)\to (\infty,0)$
\begin{equation}\label{dammit}
\| \DD-\DD_M\|_{H^1_1}=
\| \DD_0-\DD_{M,0}\|_{H^1_1(I)} =O(|t|^{\zeta'}) \, .
\end{equation}
Note  that $\DD_k=0$ for $k\ge 1$,
and that $\int \DD_0\, dx=0$, so that $\nu(\DD)=0$. Similarly, $\nu(\DD_M)=0$.
Therefore,
by the spectral properties of
$\widehat \LL$ on $\BB^{H^{1}_1}$ from Propositions~\ref{mainprop} and ~\ref{acip}, 
we have that
 $$
 (\id - \widehat\LL)^{-1}(\DD)\in \BB^{H^{1}_1}
\, ,\quad (\id - \widehat\LL)^{-1}(\DD_M)\in \BB^{H^{1}_1}\, .
 $$

Recalling $\NN_M$ from \eqref{forrlater'} in
the proof of Lemma~\ref{truncspec}, the estimate
\eqref{forrlater''} and our condition \eqref{bd1} on $M$ give
\begin{equation}
\|\NN_M (\hat \psi)\|_{\BB^{L^1}}\le \widehat C |t|^{1+\zeta} \|\hat \psi\|_{\BB^{H^1_1}}\, .
\end{equation}
In particular, since $\| \DD_{M} \|_{\BB }\le C $, we get
$\lim_{t\to 0}\| \NN_M (\DD_M)\|_{\BB^{L^1}}=0$, exponentially
in $M$.
Therefore, recalling that $\nu(\DD)=\nu(\DD_M)=0$, and using
\eqref{dammit},  there exists $\zeta''>0$
so that for  admissible  $(M,t)\to (\infty,0)$, 
\begin{align*}
&\bigl \|
(\kappa_M-\widehat \LL_M)^{-1} (\id-\PPP_M) (\DD_M)
-(\id-\widehat \LL)^{-1}  (\DD)\bigr \|_{\BB^{L^1}}\\
&\qquad\qquad\le
  \|\NN_M(\DD_M)\|_{\BB^{L^1}}
+ \|(\id-\widehat \LL)^{-1}  (\DD-\DD_M)\|_{\BB^{H^1_1}}=O(|t|^{\zeta''})\, .
\end{align*}
We may now conclude the first part of Step 1: 
Dividing \eqref{lllast} by $t$, letting $t\to 0$, 
and applying the second bound of \eqref{limm} gives $\hat \zeta >0$
so that 
(using again  $\nu(\DD)=0$) 
\begin{equation}\label{claim1}
\bigl \|
\frac{1}{t}(\hat \phi_{t,M(t)}-\hat \phi_{M(t)}) -
 (\id-\widehat \LL)^{-1}  (\DD)
 \bigr  \|_{\BB^{L^1}}=O(|t|^{\hat \zeta})
 \mbox{ as $t\to 0$.}
 \end{equation}

It remains to assess the effect of composition
by $\Pi_t$ in the first term of \eqref{decompfinal} divided by $t$.
We claim that  \eqref{claim1} implies
\begin{equation}\label{formula0}
\lim_{t \to 0}
\|\frac{1}{t}
\Pi_t [ \hat \phi_{t,M} -\hat \phi_{M}  ]-
\Pi_0((\id - \widehat \LL)^{-1} (\DD))\|_{L^1(I)}=0\, .
\end{equation}
To prove \eqref{formula0}, we start from the decomposition
\begin{equation}\label{abovv}
\frac{1}{t}
\Pi_t (\hat \phi_{t,M}-\hat \phi_M)=
\frac{1}{t}
\Pi_0(\hat \phi_{t,M}-\hat \phi_M)+
\frac{1}{t}
(\Pi_t-\Pi_0)(\hat \phi_{t,M}-\hat \phi_M)\, .
\end{equation}
 Note that \eqref{glu2} implies 
$
\|\Pi_0 \TT_M (\hat \psi)\|_{L^1}\le C 
\|\hat \psi\|_{\BB^{L^1} }
$.
Therefore, since $\nu(\DD)=0$,
 \eqref{claim1} takes care of the first term in \eqref{abovv}, 
 and it suffices
to show that the
second term in \eqref{abovv} tends to zero 
in $L^1(I)$ as $t\to 0$.

Recall  that Lemma~\ref{est1aunif} allows us to take
a larger value of $\Gamma$ in \eqref{bd2}
if necessary.
Estimate \eqref{Psi} in Step 3  implies 
$$
\|(\Pi_t-\Pi_0) (\hat \psi)\|_{L^1(I)}
\le C \sqrt t \|\hat \psi\|_{C^0}\, .
$$
Therefore, by  \eqref{notenough'}, we have
\begin{equation}\label{notenough}
\|(\Pi_t-\Pi_0)(\hat \phi_{t,M}-\hat \phi_M)\|_{L^1}
\le  C \sqrt t \|\hat \phi_{t,M}-\hat \phi_M\|_{C^0}=o(t)\, ,
\end{equation}
proving \eqref{formula0}.
\smallskip

Finally, \eqref{formula0} immediately implies that
\begin{equation}\label{formula1}
\lim_{t \to 0}
\frac{1}{t}\int A 
\Pi_t [ \hat \phi_{t,M} -\hat \phi_{M}] \, dx=
-\int A
\Pi_0((\id - \widehat \LL)^{-1}[\TT_0(\widehat \LL (\hat Y \hat \phi))'])\, dx\, .
\end{equation}
In other words,
\begin{equation}\label{formula1'}
\lim_{t \to 0}\frac{1}{t}
\int (A -A\circ f)
\Pi_t [ \hat \phi_{t,M} -\hat \phi_{M}] \, dx=
-\int A
(\TT_0(\widehat \LL (\hat Y \hat \phi)))'\, dx\, ,
\end{equation}
where we used  \eqref{commute}.
If $A$ is $C^1$, we can integrate by parts, and we find
\begin{equation}\label{formula1''}
\lim_{t\to 0}
\frac{1}{t}
\int (A-A\circ f) 
\Pi_t [ \hat \phi_{t,M} -\hat \phi_{M}]\, dx=
\int A'
 \TT_0(\widehat \LL (\hat Y \hat \phi))\, dx\, .
\end{equation}

\medskip

{\bf Step 2: The first term of \eqref{decompfinal}:
Computing $\lim\frac{1}{t}( \widehat \LL_{t,M}-\widehat \LL_{M})(\hat \phi_{M})$.}

In this step, we prove  \eqref{missarg}, (\ref{almostder}), 
and \eqref{dammit'},  \eqref{dammit}, for admissible pairs $(M,t)$.
The following estimates will play a crucial part in the argument
(their proof is given in Appendix \ref{app3},
it uses the fact that $t\mapsto f_t\in C^3$ is a $C^2$ map):

\begin{lemma}
 [Taylor series 
 for $f^{-k}_{\pm}(x)-f^{-k}_{t,\pm}(x)$]
 \label{goodlemma}
Let $f_t$ satisfy the assumptions of Theorem~\ref{linresp}. 
Recall the functions $Y_{k,t}$ 
from \eqref{defvk},  the maps $f^{-k}_{t,\pm}$ from \eqref{deff+},
and the smooth cutoff functions $\xi_k$
from Definition~\ref{defxi}. Then 
there is $C>0$ so that for any $k\ge H_0$ and
any $|s| \le \epsilon$,  if  \eqref{bd2'} holds, then
\begin{equation}\label{claim0}
 \sup_{y \in  \widetilde I_k}\frac{|Y_{k,s}(y)|}{|(f^k_s)'(y)|} \le C \frac{e^{2\gamma k}}
 {|(f^{k-1}_s)'(c_{1,s})|^{1/2}}
\, .
\end{equation}
In addition, for all $k\ge 1$, we have
\begin{align}\label{claim12}
&\sup_{x \in f^k_s(\widetilde I_k)}
\biggl |
\partial_x 
\bigl (\frac{Y_{k,s}(f^{-k}_{s,\pm}(x))}{(f^k_s)'(f^{-k}_{s,\pm}(x))} \bigr )
\biggr | \le C \frac{e^{3\gamma k}}
 {|(f^{k-1}_s)'(c_{1,s})|^{1/2}}\, ,
\\
\label{claim22}
&\sup_{x \in f^k_s(\widetilde I_k)}
\biggl |
\partial^2_x \bigl (\frac{Y_{k,s}(f^{-k}_{s,\pm}(x))}{(f^k_s)'(f^{-k}_{s,\pm}(x))} \bigr )
\biggr | \le C \frac{e^{6\gamma k}}
 {|(f^{k-1}_s)'(c_{1,s})|^{1/2}}
 \, .
\end{align}
Finally, for all $k \le M$ and all
$x\in f^k_s(\widetilde I_k)$
\begin{align}\label{claim3}
&\biggl | f^{-k}_{\pm}(x)-f^{-k}_{t,\pm}(x)
-t\frac{Y_k(f^{-k}_\pm(x))}{(f^k)'(f^{-k}_\pm(x))}
\biggr | \le
C |t|^2  e^{4\gamma k} 
 \, ,
\end{align}
and, for the same $k$ and $x$
\begin{align}\label{claim4}
&
\biggl |
\frac{1}{(f^k_t)'(f^{-k}_{\pm}(x))}-\frac{1}{(f^k)'(f^{-k}_{t,\pm}(x))}-
t \biggl ( \frac{Y_k(f^{-k}_\pm(x))}{(f^k)'(f^{-k}_\pm(x)) } \biggr )'
\biggr |
%\\ \nonumber &\qquad\qquad 
\le C |t|^2 e^{7\gamma k} \, .
\end{align}
\end{lemma}

We first prove  \eqref{missarg}.
If $j >M$ then 
$\widehat \LL_{t,M}(\hat \psi)(x,j)=\widehat \LL_M(\hat \psi)(x,j)
 =0$.
If $1 \le j\le M$, since $\xi_j=\xi_{j,t}$ (recall the construction
in Proposition~\ref{bottomok}), 
$$
\widehat \LL_{t,M}(\hat \psi)(x,j)-\widehat \LL_M(\hat \psi)(x,j)
 =0\, .
$$
Therefore, we need only worry about $j=0$.

Recall the definition \eqref{a} of 
$\widehat \LL_{t,M}(\hat \psi)(x,0)$.
The definition \eqref{hatY}  of $\hat Y_{s}$ 
(the shift in indices there mirrors that in \eqref{a}) together with
\eqref{Phi} and
\eqref{derr} from the
proof of Lemma~\ref{goodlemma}
imply the following:
Assume that
$\varphi$
is $C^1$ and supported in $\widetilde I_k$.
Then there exists $s(t) \in [0,t]$ so that
\begin{align*}
& \biggl | \frac{\varphi(f^{-k}_{+}(x))}{|(f^k)'(f^{-k}_{+}(x))|}- 
\frac{\varphi(f^{-k}_{t,+}(x))}{|(f^k_t)'(f^{-k}_{t,+}(x))|}
\biggr |\\
&\le 
|t| | \varphi(f^{-k}_{s,+}(x)) |
\biggl |
\bigl ( \frac{Y_{k,s}(f_{s,+}^{-k}(x))}{(f^k_s)'(f^{-k}_{s,+}(x)) }  
\bigr )' \biggr |
+|t| \frac{|\varphi'(f^{-k}_{s,+}(x))|}{|(f^k_s)'(f^{-k}_{s,+}(x))|} 
\biggl | \frac{Y_{k,s}(f^{-k}_{s,+}(x))}{(f^k_s)'(f^{-k}_{s,+}(x))}
\biggr | \, .
\end{align*}

Of course, the branch $f^{-k}_-$ is handled similarly.
Recall that $C^\infty$ is dense in $H^1_1$.
Therefore,
summing over the  inverse branches, 
and taking into
account the contribution of $(1-\xi_k)(f^{-k}_{t,+}(x))-(1-\xi_k)(f^{-k}_{+}(x))$
via \eqref{kappa'} or \eqref{dirtytrick}
(our assumptions imply that each $\psi_k'$
and $\xi_k'$ vanishes at
the boundary of its support), and averaging,
we get $C>0$ so that
for any $\hat \psi \in \BB$ and
any admissible pair $(M,t)$, using
 \eqref{claim0}
and \eqref{claim12} from Lemma~\ref{goodlemma} 
and the upper bound \eqref{48} on $\lambda$
\begin{equation}\label{????}
\|\widehat \LL_{t,M}(\hat \psi)-
\widehat \LL_{M}(\hat\psi) \|_{\BB^{L^1}}
\le C |t| e^{5\gamma M}\lambda^{2M} \|\hat \psi\|_{\BB^{H^1_1}} 
\le C |t| e^{7\gamma M} \|\hat \psi\|_{\BB^{H^1_1}}  \, .
\end{equation}
In view of \eqref{bd2}, this proves \eqref{missarg}.

Next, we show (\ref{almostder}).
Note that
\begin{align*}
& \varphi(f^{-k}_+(x)) 
\biggl ( \frac{Y_k(f^{-k}_+(x))}{(f^k)'(f^{-k}_+(x)) }  
\biggr )'
+  \frac{\varphi'(f^{-k}_+(x))}{(f^k)'(f^{-k}_+(x))}
\frac{Y_{k}(f^{-k}_{+}(x))}{(f^k)'(f^{-k}_{+}(x))}\\
&\qquad\qquad\qquad\qquad\qquad\qquad\qquad\qquad
=\biggl ( \frac{\varphi(f^{-k}_+(x)) Y_k(f^{-k}_+(x))}
{(f^k)'(f^{-k}_+(x))}\biggr)' \, ,
\end{align*}
and set (recall that $(\hat \phi_{M})_0'\in H^1_1$)
$$
\DD_M := -(\TT_0(\widehat \LL_M(\hat Y \hat \phi_{M})))'  \in \BB\, .
$$
Clearly
$\nu(\DD_M)=0$, integrating by parts. 
Since
Lemma ~\ref{truncspec}
implies that  $\hat \phi_M$ is an eigenvector of $\widehat \LL_M$
for an eigenvalue $\kappa_M$ close to $1$ 
(so that the $\lambda^k$ factor can be replaced by
$\kappa_M^k$, which is strictly smaller than
$C e^{2\gamma k}$ by our choices), we get,  
using \eqref{claim0} and
\eqref{claim12} from Lemma~\ref{goodlemma}, 
as well as Lemma~\ref{rootsing},
$ \sup_M\|\DD_M\|_{\BB^{L^1}} < \infty$.
 Up to increasing $\Gamma$ in \eqref{bd2}, 
 \eqref{claim12} and \eqref{claim22} 
imply that
$\sup_M \|\DD_{M}\|_{\BB^{ H^1_1}}<\infty$.

Using again that $(\hat \phi_{M})_0'\in H^1_1$, we may write the
average of the $t$-Taylor series  of order two of
 $$
  \frac{\hat \phi_{M,k-1}(f^{-k}_{+}(x))}{|(f^k)'(f^{-k}_{+}(x))|}- 
\frac{\hat \phi_{M,k-1}(f^{-k}_{t,+}(x))}{|(f^k_t)'(f^{-k}_{t,+}(x))|}\, ,
$$
and of its $x$-derivative. 
By \eqref{claim22},  \eqref{claim3}, and \eqref{claim4},
this gives
\begin{equation}\label{uselater}
\|\widehat \LL_{t,M}(\hat \phi_M)-\widehat \LL_{M}(\hat\phi_M) 
-t\DD_M \|_{\BB}
\le C e^{11 \gamma M} |t|^{1+\zeta} \,  .
\end{equation}
Since we can take $\Gamma$ in
\eqref{bd2} as large as necessary,  this
establishes \eqref{almostder}.

Set $\DD=-\TT_0(\widehat \LL(\hat Y \hat \phi))'$. Clearly,
$\nu(\DD)=0$, integrating by parts.
The estimates we proved imply that
$\|\DD_0-\DD_{M,0}\|_{H^1_1}\to 0$, exponentially
\comment{recheck}
fast 
as $M \to \infty$, and that $\DD \in \BB^{H^1_1}$. This shows
\eqref{dammit'} and \eqref{dammit}.
\medskip

%%%%%%%%%%%%%%%%%%%%%%%%%%%%

{\bf Step 3: The second term of \eqref{decompfinal}: Estimating $\frac{1}{t}(\Pi_t-\Pi_0)(\hat \phi_M) \in (C^{1}(I))^*$}

In this step, the points are not necessarily falling
from the tower, so that  the analogues of the derivatives
in Lemma ~\ref{goodlemma} have nonintegrable spikes.
Therefore,
as already mentioned, we shall not
only require horizontality, but we shall
also need to perform integration by parts, using that
the observable $A$ is $C^1$.

As before, the index $k$ ranges between $1$ and $M$,
where $(M,t)$ is an admissible pair. We focus on the branch $f^{-k}_+$, the other one is handled in a similar way.

Note for further use that, recalling
 Proposition~\ref{bottomok},
Lemma \ref{rootsing}
implies that there exists $C$ so that for any 
admissible pair $(M,t)$ and any $1\le k\le M$
\begin{align}\label{Psi}
\Psi_{k,t}&:=
\sup_{z \in [c_k, c_{k,t}]}\int_{c_k}^z\frac{\lambda^k}{|(f^k)'(f^{-k}_+(x))|} \, dx
%\\ \nonumber &
\le   \int_{c_k}^{c_{k,t}}\frac{C}{\sqrt{x-c_k}}\,  dx
%\le C \sqrt{c_{k,t}-c_k} 
\le C L |t|^{1/2}\, .
\end{align}
(We used  \eqref{Lbound} to get $|c_{k,t}-c_k|=|h_t(c_k)-h_0(c_k)|\le L |t|$.)
In particular,
\begin{equation}\label{inpart}
|f^{-k}_+(c_{k,t})- f^{-k}_+(c_k)|\le \Psi_{k,t} \le C |t|^{1/2}\, .
\end{equation}

Assume to fix ideas
that $c_k > c_{k,t}$, 
with $c_k$ and $c_{k,t}$  local maxima for $f^k$ and $f^k_t$, respectively
(the other possibilities are treated similarly and left
to the reader).

We   first study the points in $f^k(\mbox{supp}(\phi_{M,k}))$
for which $f^{-k}_+(x)$ exists 
but not $f^{-k}_{t,+}(x)$, i.e., the interval
$[c_{k,t}, c_k]$. This gives the following contribution: 
\begin{align}
\label{firstudy}- \int_{c_{k,t}}^{c_k}
A(x) &\frac{\lambda^k}{|(f^k)'(f^{-k}_+(x))|}  \phi_{M,k}(f^{-k}_+(x))
\, dx\\
\nonumber &= \int_{f^{-k}_+(c_{k,t})}^{c_0}
[(A(f^k(y)) -A(c_k)) + A(c_k)]\lambda^k \phi_{M,k}(y)\, dy\\
\nonumber &=  \int_{f^{-k}_+(c_{k,t})}^{c_0}
A'(z_{k,t}(y))  (c_k -f^k(y))\lambda^k \phi_{M,k}(y)\, dy\\
\label{singular}
&\qquad\qquad
+A(c_k)\int_{f^{-k}_+(c_{k,t})}^{c_0} \lambda^k  \phi_{M,k}(y) \, dy\, ,
\end{align}
where $z_{k,t}(y)\in [c_{k,t}, c_k]$, and we used that $f^{-k}_+$ is orientation
reversing. Now
\begin{align*}
&|\int_{f^{-k}_+(c_{k,t})}^{c_0}
A'(z_{k,t}(y))  (c_k -f^k(y))\lambda^k \phi_{M,k}(y)\, dy| \\
&\qquad \le \sup |A'|  |c_k -c_{k,t}|
\int_{c_{k,t}}^{c_k}  \frac{\lambda^k}{|(f^k)'(f^{-k}_+(x))|} \phi_{M,k}(f^{-k}_+(x)) \, dx\\
&\qquad\le C \sup |A'|  |c_k -c_{k,t}| \sqrt{c_k -c_{k,t}}    \\
&\qquad\le C \sup |A'| \sup_s |\alpha_s(c_k) |^{3/2} |t|^{3/2}\, ,
\end{align*}
where we used Lemma ~\ref{first'''} together 
with \eqref{Psi}.
Since there are $M$
terms and since $\lim_{t\to 0}M\sqrt{|t|}=0$ for admissible pairs
$(M,t)$, the relevant contribution of \eqref{firstudy} is 
fully contained
in the last line \eqref{singular} of \eqref{firstudy}.
(We shall see in a moment that \eqref{singular}
cancels out exactly with another term.)

Second, we need to consider
\begin{align}
\label{signja} &- \int_{-1}^{c_{k,t}}
\lambda^k A(x) \bigl (\frac{ \phi_{M,k}(f^{-k}_+(x))}
{|(f^k)'(f^{-k}_+(x))|}
-\frac{ \phi_{M,k}(f^{-k}_{t,+}(x))}{|(f^k_t)'(f^{-k}_{t,+}(x))|}\bigr ) 
\, dx\\
\label{other} &\qquad=-\int_{-1}^{c_{k,t}}
\lambda^k A'(x) 
[\tilde \phi_{M,k}(f^{-k}_+(x))-\tilde \phi_{M,k}(f^{-k}_{t,+}(x))]\, dx\\
\label{just} &\qquad\qquad\quad
+ \lambda^k [A(x) 
(\tilde \phi_{M,k}(f^{-k}_+(x))-\tilde \phi_{M,k}(f^{-k}_{t,+}(x))) ]|_{-1}^{c_{k,t}}\, .
\end{align}
where $\tilde \phi_{M,k}'= \phi_{M,k}$, $\tilde \phi_{M,k}(-1)=0$, and
we used that
$(f^k)'(f^{-k}_{+}(x))<0$ when integrating by parts.
One  term in \eqref{just} vanishes because of 
the support of $\tilde \phi_{M,k}$. The other term is
\begin{equation}\label{justt}
 \lambda^k A(c_{k,t})
(\tilde \phi_{M,k}( f^{-k}_+(c_{k,t}))-\tilde \phi_{M,k}(f^{-k}_{t,+}(c_{k,t})))
\, .
\end{equation}
Since $A(c_{k,t})=A(c_k)+ A'(z_{k,t}) (c_{k,t}-c_k)$
for $z_{k,t}\in [c_{k,t},c_k]$,  and since
$f^{-k}_{t,+}(c_{k,t})=c_0$, \eqref{justt} 
can be written as
$$
 \lambda^k
A(c_k) (\tilde \phi_{M,k}( f^{-k}_+(c_{k,t}))-\tilde \phi_{M,k}(c_0))
+H_{k,t}\, , $$
with $|H_{k,t}| \le C |t|^{3/2}$, uniformly
in $k\le M$ for admissible pairs $(M,t)$ (recall 
\eqref{Lbound} and \eqref{inpart}).
Summing over
$1\le k \le M$ and dividing
by $t$, we have proved that, as
$t \to 0$,  the contributions from \eqref{justt} cancel out exactly with the
singular terms from line \eqref{singular}.
(Recall that $(M,t)$ are admissible, in particular \eqref{bd2} holds.)
 
 The other term,  \eqref{other}, is
 \begin{align}
\nonumber & -\int_{-1}^{c_{k,t}}
\lambda^k A'(x) 
[\tilde \phi_{M,k}( f^{-k}_+(x)) - \tilde \phi_{M,k} (f^{-k}_{t,+}(x))]\, dx
\\
\label{otherr} &\qquad=
-\int_{-1}^{c_{k,t}}
\lambda^k A'(x)  \phi_{M,k}(f^{-k}_{u(t,x),+}(x))
[ f^{-k}_+(x)- f^{-k}_{t,+}(x)]\, dx\, ,
 \end{align}
for $u=u(t,x)\in [0,t]$.
To finish, we shall next prove that the sum over
$1\le k \le M$ of \eqref{otherr} divided by $t$
converges as $t\to 0$ and $(M,t)$ is an admissible pair.

Recalling the definition \eqref{defvk} of $Y_{k,s}$,
the
proof of Lemma~ \ref{goodlemma} (in particular \eqref{Phi}) implies that
there is  $s=s(t,x) \in [0,t]$ so that
\begin{align*}
f^{-k}_{t,+}(x)-f^{-k}_+(x)
&=-t \frac{Y_{k,s}(f^{-k}_{s,+}(x))}{(f^k_s)'(f^{-k}_{s,+}(x))}
=-t\sum_{j=1}^{k}
 \frac{X_s (f^j_s(f^{-k}_{s,+}(x)))}{(f^{j}_s)'(f^{-k}_{s,+}(x))} \, .
\end{align*}
Since $X$
is $C^1$, and since bounded distorsion holds for points which climb
(Lemma ~\ref{cd}),  we find (recall 
\eqref{est2beq'} in the proof of Proposition~ \ref{abs})
\begin{align*}
\biggl | \sum_{j=1}^{k}
 \frac{X_s (f^{j-k}_{s,+}(x))}{(f^{j}_s)'(f^{-k}_{s,+}(x))} &-
 \frac{1}{f'_s(f^{-k}_{s,+}(x))}\sum_{j=1}^{k} 
\frac{X_s (f^{j-1}_{s,+}(c_{1,s}))}{(f^{j-1}_s)'(c_{1,s})}
\biggr |\\
&\qquad\qquad \le C e^{2\gamma k} |(f_s^{k-1})'(c_{1,s})|^{-1/2}\, .
\end{align*}
Then, using horizontality, we find
$$
\sum_{j=1}^{k} 
\frac{X_s (f^{j-1}_{s,+}(c_{1,s}))}{(f^{j-1}_s)'(c_{1,s})}
=-
\frac{1}{(f^k_s)'(c_{1,s})}\sum_{\ell=k}^{\infty} 
\frac{X_s (f^{\ell}_s(c_{1,s}))}{(f^{\ell-k}_s)'(c_{k+1,s})}\, .
$$
The proof of Proposition~ \ref{abs} implies that the above expression
is bounded, uniformly in $k\le M$ and
admissible pairs $(M,t)$. 
(Here we use the uniform bounds
from Lemma~\ref{est1aunif}.)
Finally, recalling the properties of the support
of $ \phi_{M,k}$
\begin{align*}
|\int_{c_k-e^{-\beta_1 k}}^{c_{k,t}}\frac{1}{|(f_s')(f^{-k}_{s,+}(x))|}\, dx|
&\le C \int_{c_k-Ce^{-\beta_1 k}}^{c_{k,t}}\frac{1}{\sqrt {|x-c_{k,s}|}}\, dx\\
&= C \sqrt {|x-c_{k,s}|}|_{c_k-Ce^{-\beta_1 k}}^{c_{k,t}} \, .
\end{align*}
Summarizing, we have proved that \eqref{otherr} divided
by $t$ satisfies (recall also Lemma~\ref{first} and (\ref{bd2'}))
\begin{align*}
& |
\frac{1}{t}
\int_{-1}^{c_{k,t}} \lambda^k A'(x) 
 \phi_{M,k}(f^{-k}_{u(t,x),+}(x))
[ f^{-k}_+(x)- f^{-k}_{t,+}(x)]\, dx|\\
&\qquad\qquad\qquad=
|
\int_{-1}^{c_{k,t}} \lambda^k A'(x) 
\phi_{M,k}(f^{-k}_{u(t,x),+}(x)) \frac{Y_{k,s}(f^{-k}_{s,+}(x))}{(f^k_s)'(f^{-k}_{s,+}(x))}\, dx|
\\
&\qquad\qquad \qquad\le C 
\kappa_M^{-k} \sup| A'|\sup |  \phi_{M,0}|
e^{-\beta_1 k/2}\, .
\end{align*}
The bound
in the third line above
is  summable over $k\ge 1$, uniformly in $M$.

Since it is easy to check for each fixed $k$ that
\begin{align*}
&\lim_{t \to 0}
\int_{-1}^{c_{k,t}} \lambda^k  A'(x) 
 \phi_{M,k}(f^{-k}_{u(t,x),+}(x)) \frac{Y_{k,s}(f^{-k}_{s,+}(x))}{(f^k_s)'(f^{-k}_{s,+}(x))}\, dx\\
&\qquad\qquad\qquad\qquad=
\int_{-1}^{c_{k}} \lambda^k 
A'(x) \phi_k(f^{-k}_+(x))
\frac{Y_k(f^{-k}_+(x)) }{(f^k)'(f^{-k}_+(x)}\, dx\, ,
\end{align*}
we have proved that 
\begin{align*}
&\lim_{t \to 0}\frac{1}{t}
\int_I A(x)
\bigl (\Pi_t  ( \phi_{M}) -\Pi  ( \phi_{M}) \bigr )
(x)\, dx\\
\nonumber &\quad =-\sum_{k=1}^\infty
\sum_{\varsigma\in \{+,-\}}
\pm \int_{-1}^{c_{k}}
\lambda^k A'(x) \phi_k(f^{-k}_+(x))
\frac{Y_k(f^{-k}_\varsigma (x)) }{(f^k)'(f^{-k}_\varsigma(x)}\, dx
\\
\nonumber &\quad =-\sum_{k=1}^\infty
\sum_{\varsigma\in \{+,-\}}
\int_{-1}^{c_{k-1}}
A'(f(y)) 
\frac{\lambda^{k-1}}{|(f^{k-1})'(f^{-(k-1)}_\varsigma(y)|} 
\lambda (\phi_k\cdot Y_k)(f^{-(k-1)}_\varsigma(y)) 
 \, dy\, .
\end{align*}
(The sign in the second line above comes from 
in \eqref{signja}, that is, it is the sign of $(f^{k-1})'(f^{-(k-1)}_\varsigma(x)$.)
The fixed point property of $\hat \phi$
implies $\lambda\phi_{k+1}=\phi_k \xi_k$.
Therefore, setting $\hat\xi=(\xi_k)$, and recalling
the shift in indices in the definition 
\eqref{hatY} of $\hat Y$, we have
\begin{align}
\nonumber
\lim_{t \to 0}\frac{1}{t}
\int_I A(x)
\bigl (\Pi_t  (\hat \phi_{M}) -\Pi  (\hat \phi_{M}) \bigr )
(x)\, dx
 &=-
\int_I (A'\circ f )\cdot \Pi  (\hat \xi \hat Y \hat \phi) \, dy\\
\label{formula2}
&=
-\lambda
\int_I A' \cdot (\Pi \circ ( \id- \TT_0) \circ \widehat \LL) (\hat Y \hat \phi) \, dy
\, ,
\end{align}
where we used \eqref{commute}.
This ends Step 3 and the proof of Theorem ~ \ref{linresp}.

%%%%%%%%%%%%%%%%%%%%%%%%%%%%%%%%%%%%%%%%%%%%%%%%%%%%%%%%

\begin{appendix}
\label{appp}

\section{Relating the conjugacies $h_t$ with the infinitesimal conjugacy $\alpha$}
\label{htC1etc}

We show here that $\alpha$ deserves
to be called an infinitesimal conjugacy.

\begin{proof}(Proof of Proposition~\ref{htC1}.)
Let $\alpha_t\colon[-1,1]\to \mathbb{R}$ be the unique continuous solution for the TCE
 $$v_t= \alpha_t\circ f_t - f'_t\cdot \alpha_t.$$
 
 Since the family $\{\alpha_t\}_{|t|< \epsilon}$ is equicontinuous and the solutions are unique, it is easy to see that $(t,x)\to \alpha_t(x)$ is a continuous and bounded function in $(-\epsilon,\epsilon)\times [-1,1]$. Note that $\alpha_t(-1)=\alpha_t(1)=0$ for every $t$. For each $x_0 \in [-1,1]$ $t_0 \in (-\epsilon,\epsilon)$,  the
Peano theorem ensures that the ODE
\begin{equation}\label{ode}
\partial u_s(t_0,x_0)|_{t=s}=\alpha_t(u_t(t_0,x_0))\, ,
\quad u_{t_0}(t_0,x_0)=x_0 \, 
\end{equation}
admits a $C^1$ solution $u_t(t_0,x_0)$. It is not difficult  to see that this solution is defined for every $t \in (-\epsilon,\epsilon)$. 
Since $f_t$ is a deformation, there exists an unique  conjugacy $h_t$ such that 
$$f_t\circ h_t = h_t\circ f_0\, .$$

If $x_{t_0}$ is an eventually periodic point for $f_{t_0}$, since all periodic points are hyperbolic there exists an analytic continuation $x_t$ for $x_0$. Then $x_t=h_t(x_0)$. An easy calculation shows that $u_t(t_0,x_0)=h_t(x_0)$ is a solution of the above ODE. \\

\noindent {\it Claim: } If $w_t$ is a solution of the ODE $\partial_t w_t=\alpha_t(w_t)$ and $w_{t_0}=h_{t_0}(x_0)$ for some $t_0$ and eventually periodic point $x_0$  then 
$w_t=h_t(x_0)$ for every $t$. Indeed, denote $w_t^n=f_t^n(w_t)$. Note that the TCE implies that $w_t^n$ and $f_t^n(h_t(x_0))$ are  also solutions of the ODE above.  Since $\alpha_t(x)$ is a bounded function
$$|w_t^n-f^n_t(h_t(x_0))|\leq  |w_t^n-w_{t_0}^n|+ |f_{t_0}^n(h_{t_0}(x_0))-f_t^n(h_t(x_0))|\leq 2\sup_{(t,x)}|\alpha_t(x)||t-t_0|\, .
$$
So if  $t$ is sufficiently close to $t_0$ then
$$|f_t^n(w_t)-f^n_t(h_t(x_0))|< \delta$$
for every $n$.  If $\delta$ is small enough, since the orbit of $h_t(x_0)$ by $f_t$ eventually lands on a repelling periodic point, it follows that $w_t=h_t(x_0)$ for $t$ close enough to $t_0$. This argument  implies that 
$$\{t\colon \   w_t=h_t(x_0)\}$$
is an open set in $(-\epsilon, \epsilon)$. Since it is obviously a closed set, it follows that $w_t=h_t(x_0)$ for every $t$. This finishes the proof of the claim. 

In particular this claim implies the uniqueness of the solution of the ODE when $x_0$ is an eventually periodic point.

Now let $x$ be a point that is not eventually periodic for $f_0$. We can find  sequences $p_n$, $q_n$ of eventually
periodic points  for $f_0$  such that
$$p_n< x < q_n$$
and $\lim_n p_n=\lim_n q_n =x.$
Let $w_t$ be a solution for $\partial_t w_t=\alpha_t(w_t)$ such that $w_0=x$. The claim above implies that 
$$h_t(p_n) < w_t < h_t(q_n)$$ 
for every $n$. Since $h_t$ is continuous we get $\lim_n h_t(p_n)=\lim_n h_t(q_n)= h_t(x)$, so $h_t(x)=w_t$, for every $t$.  
\end{proof}

%%%%%%%%%%%%%%%%%%%%%%%%%%%%%%%%%%%%%%%%%%%

\section{The Lasota-Yorke bound \cite[Sublemma]{BV} for Sobolev
norms and probabilistic operators}
\label{app2}

We adapt the argument of \cite[Sublemma]{BV} (see also \cite[Lemma 5.5]{V})
to complete the proof of Proposition~\ref{mainprop}, by showing
that there exists $C$, and for all $n$ there is $C(n)$, so that
\begin{equation*}
\|(\widehat \LL^n( \hat \psi))'_0\|_{L^1(I)} \le
C \Theta_0^{-n}
\|\hat \psi\|_{\BB} + C(n) \|\hat \psi\|_{\BB^{L^1}}\, .
\end{equation*}
Keeping in mind  \eqref{der0N} and \eqref{der00},
as well as the properties of the supports of the $\xi_j$ and $\psi_j$,
\eqref{finally}, \eqref{gluu}, \eqref{sobeb}, \eqref{dirtytrick}, and
the conditions \eqref{48}
on $\lambda$, 
one first obtains the following analogue of \cite[(4.14)--(4.15)]{BV} or 
\cite[(5.25)--(5.26)]{V}: There is a constant
$C$ and for each $n$ a constant $C(n)$ so that
for  each interval $A\subset E_0$
\begin{align}
\nonumber \int_A | (\widehat  \LL^n (\hat \psi))'_0(x)|\, dx
&\le C \frac{\Theta^{-n}}{c(\delta)}
\biggl (\sum_{\omega \in \Omega_0}
\int_{\omega} |\hat \psi'(x)|\, dx+ \|\hat \psi\|_{\BB^{H^1_1}}\biggr )\\
\label{recurs} &+ \frac{C(n)}{c(\delta)}\|\hat \psi\|_{\BB^{L^1}}
+ C \sum_{\ell=H(\delta)}^{n-1}\frac{\Theta^{-\ell}}{c(\delta)}\sum_{\omega'
\in \Omega_\ell}
\int_{\omega'} |(\widehat \LL^{n-\ell}(\hat \psi)) '(x)|\, dx\, ,
\end{align}
while if the interval $A \subset (-\delta,\delta)$, we get a stronger estimate
(by Lemma~\ref{expiii})
where all factors $c(\delta)$ in the right-hand-side
above may be replaced by $1$.

In \eqref{recurs}, the intervals $\omega$ in $\Omega_0$
are either inverse images of $A$ in $E_0$ through branches of our probabilistic
version of $\hat f^n$ which always stay in $E_0$, or  inverse image of $A$ in $E_k$,
$k(\omega)\ge 0$,  through branches of the probabilistic  $\hat f^n$ which start from
$E_k$, climb to $E_{k+n-j}$, for $n-j-1$ iterations with $0\le j \le n-1$, then drop to level $E_0$ and stay there for the last $j$ iterations.
(The detailed analysis is slightly different
depending on whether $k(\omega)>N$ or $\le N$,
 for some $N$ to be chosen
much larger than $n$, in order to avoid dividing by small lengths $|U|$
in \eqref{sobeb}, see \cite{BV} or \cite{V} for details.)
The intervals $\omega\in \Omega_0$ in
$E_k$ for $k\ge 0$ which will drop from
$E_{k+n-j}$ for $k+n-j\ge H(\delta)$ have intersection multiplicity at most $k+n-j$
by Remark ~\ref{overlap}. Since the corresponding $k+n-j$ factor is 
killed by
an exponential $\rho^{-(k+n-j)}$ factor implicit
in \eqref{recurs} (see \cite[Sublemma]{BV}), 
we may in fact glue overlapping intervals
together up to taking a slightly smaller $\Theta>1$.
The intervals $\omega\in \Omega_0$ whose orbits never
leave $E_0$ are disjoint by construction. Each of them meets
at most one of the (just grouped) intervals which go through 
$E_{k+n-j}$, and we can glue them together at the cost of
replacing $C$ by $C+C=2C$.

The intervals $\omega'\in \Omega_\ell$ 
in \eqref{recurs} are inverse images of $A$ in $E_0$ 
via the branches of our probabilistic
version of $\hat f^\ell$ which climb the tower up to some
level $k=k(\omega')$ with $H(\delta)\le k < \ell <n<N$, fall to level $E_0$, and then stay in $E_0$ for the remaining $\ell-k-1$  iterations.
Remark ~\ref{overlap} about the maximum overlap  of fuzzy monotonicity
intervals $\tilde I_k$ ensures that the intersection multiplicity
of the intervals in $\Omega_\ell$ dropping
from level $k\ge H(\delta)$ is at most $k\le \ell$.
Since $\ell \Theta^{-\ell}$ may be replaced by $\Theta^{-\ell}$ for
$\ell \ge H(\delta)$, up to
taking a slightly smaller $\Theta$, we may 
regroup 
overlapping intervals in $\Omega_\ell$.
% (in fact, there is again an implicit $\rho^{-k}$ factor).

If $\Omega_\ell$ is empty for $\ell\ge 1$, we are done.
Otherwise, to perform the  inductive step, let us rename
$\Omega_\ell=\Omega_\ell^1$ for $\ell \ge 0$. 
Exploiting Lemma~\ref{expiii} to see that we get at
most one factor $c(\delta)^{-1}$, we can then inductively conclude the argument,
just like in \cite[Sublemma]{BV} (see also \cite[Lemma 5.5]{V}).
The only difference with respect to the analysis in \cite{BV} is that the intervals
of $\Omega_0^{2}$ (after regrouping, which may
be done as above) may overlap with those of $\Omega_0^1$.
More generally, the intervals
of $\Omega_0^{m+1}$ may  intersect those of 
$\cup_{i=1}^m \Omega_0^i$. 
Since there are at most $n/H(\delta)$ inductive steps, 
the overlap factor $n/H(\delta)$ is negligible in front
of $C \Theta^{-n}$.

%%%%%%%%%%%%%%%%
\section{Proof of Lemma ~ \ref{goodlemma} on Taylor expansions}
\label{app3}

As usual, we consider  $f^{-k}_+$, the other branch is similar.
The assumptions imply that $\partial_t f_t|_{t=s}= X_s \circ f_s$, 
where $X_s\circ f_s$ is $C^2$ and horizontal for $f_s$.

We prove \eqref{claim0} and \eqref{claim12}--\eqref{claim22}
for $s=0$, the general case then follows from 
Lemma~\ref{est1aunif}, using that for all
$H_0 \le k\le M$ so that \eqref{bd2'} holds, we may take $\xi_{k,t}=\xi_k$
by Proposition~\ref{bottomok}.

By horizontality, the  estimate
\eqref{est2beq'} in the proof of Proposition~\ref{abs} 
(using the notation $w_k(y)$ introduced there) gives
$C>0$ so that for
any $k\ge H_0$
\begin{equation}\label{byhor}
\sup_{y \in I_k}\frac{|Y_k(y)|}{|(f^k)'(y)|} =
 \biggl |
\sup_{y \in I_k}\sum_{j=1}^{k} \frac{X(f^{j}(y))}{(f^{j})'(y)}
\biggr |=
\sup_{y \in I_k} |w_k(y)|
\le C \frac{e^{2\gamma k}}{|(f^{k-1})'(c_1)|^{1/2}} \, ,
\end{equation}
proving \eqref{claim0}.

For the  claim \eqref{claim12} on the derivative,  note that
\begin{equation}\label{factor}
\partial_x
\frac{Y_k(f^{-k}_+(x))}{(f^k)'(f^{-k}_+(x))} =
\frac{1}
{(f^k)'(f^{-k}_+(x))}
\partial_y
\frac{Y_k(y)}{(f^k)'(y)}\, ,
\end{equation}
with
\begin{align}
\nonumber \partial_y \frac{Y_k(y)}{(f^k)'(y)} &=
\sum_{j=1}^{k} \partial_y \frac{X(f^{j}(y))}{(f^{j})'(y)}\\
\label{secondline}
&=\sum_{j=1}^{k} X'(f^{j}(y))
-
\sum_{j=1}^{k} X(f^{j}(y))
\sum_{\ell=0}^{j-1}\frac{f''(f^\ell(y))}{(f^{j-\ell})'(f^\ell(y))
f'(f^\ell(y))}\, .
\end{align}

We shall use the estimates in the proof of Lemma~ \ref{rootsing}:
For $y\in I_k$,  the bound \eqref{basic}
says that $|(f^{k-m})'(f^m(y))|\ge Ce^{-\gamma m}$,
for $1 \le m \le k-1$, the bound \eqref{basic2}
says that $|(f^m)'(y)|\ge C_2 e^{-\gamma k}
|(f^{m-1})'(c_1)| |(f^{k-1})'(c_1)|^{-1/2}$
for $1 \le m \le k$,
 while \eqref{firstnote} gives $|f'(f^\ell(y))|\ge C e^{-\gamma \ell}$ 
 for $1\le \ell \le k$.
 (These bounds do not use horizontality.)
 
 The second term in the right-hand-side of \eqref{secondline}
 can be decomposed as
 \begin{align}\label{seccond}
 &-\sum_{j=1}^{k} X(f^{j}(y))
\sum_{\ell=0}^{j-1}\frac{f''(f^\ell(y))}{(f^{j-\ell})'(f^{\ell}(y))
f'(f^\ell(y))}\\
\nonumber &\qquad=-
\frac{f''(y)}{f'(y)}
\sum_{j=1}^k  \frac{X(f^{j}(y))}{(f^{j})'(y)}
-\sum_{\ell=1}^{k-1}
\frac{f''(f^\ell(y))}{f'(f^\ell(y))}
\sum_{j=\ell+1}^k  \frac{X(f^{j}(y))}{(f^{j-\ell})'(f^\ell(y))}\, .
\end{align}
By \eqref{basic2} for $m=1$, combined
with \eqref{claim0} (which holds by horizontality),
we find
$$
\left | \frac{f''(y)}{f'(y)}
\sum_{j=1}^k  \frac{X(f^{j}(y))}{(f^{j})'(y)} \right |
\le   C e^{3\gamma k}\, .
$$
The second term in the right-hand-side of \eqref{seccond} does not require horizontality, only 
\eqref{basic} and \eqref{firstnote}, which give
\begin{equation}\label{firsterm}
\sum_{\ell=1}^{k-1}
\frac{|f''(f^\ell(y))|}{|f'(f^\ell(y))|}
\sum_{j=\ell+1}^k  \frac{|X(f^{j}(y))|}{|(f^{j-\ell})'(f^{\ell}(y))|}
\le C e^{2\gamma k} \, .
\end{equation}
 Remembering \eqref{factor}, and using again \eqref{basic2}
 (for $m=k$), we have proved \eqref{claim12}.

The proof of  \eqref{claim22}
is similar. We start by noting that
$\partial^2_x
\frac{Y_k(f^{-k}_+(x))}{(f^k)'(f^{-k}_+(x))} =$
\begin{equation}
\label{factor'}
\partial_x \frac{1}
{(f^k)'(f^{-k}_+(x))}
\partial_y
\frac{Y_k(y)}{(f^k)'(y)}
+ \frac{1}
{((f^k)'(f^{-k}_+(x))^2}
\partial^2_y
\frac{Y_k(y)}{(f^k)'(y)}\, ,
\end{equation}
where $\partial^2_y \frac{Y_k(y)}{(f^k)'(y)}=$
\begin{align}
\label{der2}
&\sum_{j=1}^{k} X''(f^{j}(y)) (f^j)'(y)
-\sum_{j=1}^{k} X'(f^{j}(y))
\sum_{\ell=0}^{j-1}\frac{f''(f^\ell(y)) (f^\ell)'(y)}{
f'(f^\ell(y))}\\
\label{der22}&\qquad -
\sum_{j=1}^{k} X(f^{j}(y))
\biggl [
\sum_{\ell=0}^{j-1}\frac{f'''(f^\ell(y)) (f^\ell)'(y)}
{(f^{j-\ell})'(f^\ell(y)) f'(f^\ell(y))}
-\frac{f''(f^\ell(y))^2 (f^\ell)'(y)}
{(f^{j-\ell})'(f^\ell(y))( f'(f^\ell(y)))^2}\\
\label{der23}&\qquad\qquad\qquad\qquad-
\frac{f''(f^\ell(y))}{f'(f^\ell(y))}
\sum_{i=0}^{j-\ell-1}
\frac{f''(f^i(y)) }{(f^{j-\ell-i})'(f^{\ell+i}(y)) f'(f^{\ell+i}(y))}
\biggr ]
\, .
\end{align}
The first term in \eqref{factor'} is bounded
by $e^{5\gamma k}|(f^{k-1})'(c_1)|^{-1/2}$ by 
the proof of \eqref{claim12} and
\eqref{rootsing2} from Lemma~\ref{rootsing}.
The first term in \eqref{der2} is bounded by
$C |(f^{k-1})'(c_1)|^{1/2}$, in view of Lemma~\ref{cd}
and Lemma~\ref{sizeIj}. Hence, dividing
by $((f^k)'(f^{-k}_+(x))^2$, and using
Lemma~\ref{sizeIj}, the contribution of this term is
bounded by $C e^{2 \gamma k} |(f^{k-1})'(c_1)|^{-1/2}$.

Using again the same observations, we find that
the second term in \eqref{der2} is bounded
by $C e^{3\gamma k}$. Dividing
by $((f^k)'(f^{-k}_+(x))^2$,  we get a contribution 
bounded by $C e^{5 \gamma k} |(f^{k-1})'(c_1)|^{-1}$.

 For   \eqref{der22}, one must 
distinguish between the terms for $\ell=0$, for which 
horizontality  gives a bound 
$C e^{4 \gamma k} |(f^{k-1})'(c_1)|^{1/2}$
(note the factor $(f'(y))^2$ in the denominator),
and the  terms where $\ell \ge 1$,
for which a  straightforward estimate,
using the remarks above
(and in particular Lemma~\ref{cd}), gives
an upper bound of the form $C e^{3\gamma k} $.
Dividing
by $((f^k)'(f^{-k}_+(x))^2$,  we get a contribution 
bounded by $C e^{5 \gamma k}|(f^{k-1})'(c_1)|^{-1/2}$.

For  \eqref{der23}, if $\ell =0$ and $i=0$,
horizontality gives a bound 
$C e^{4 \gamma k} |(f^{k-1})'(c_1)|^{1/2}$, while
if $i+\ell \ge 1$, we get a bound
$C e^{4\gamma k} $.
Dividing
by $((f^k)'(f^{-k}_+(x))^2$,  we get
a contribution 
 $\le C e^{6 \gamma k}|(f^{k-1})'(c_1)|^{-1/2}$.
This ends the proof of \eqref{claim22}.

\smallskip

In view of the more complicated estimates to follow,
we notice the following pattern: The dangerous factors in the
above estimates are powers of $f'(y)$ in the denominator
and factors $(f^\ell)'(y)$ for large
$\ell$  in the numerator. The ``white knight" available
to fight them is a power of $(f^k)'(y)$ in the denominator. 
An additional such power appears each time we differentiate
with respect to $x$.
The
terms for which the power of $f'(y)$ in the denominator
exceeds that of $(f^k)'(y)$ in the numerator can be handled by
horizontality. The price to be paid for the control is a 
power of $e^{\gamma k}$.

\smallskip
We turn to \eqref{claim3} and
\eqref{claim4}. 
If $(x,t)\mapsto \Phi_t(x)\in I$ is a $C^1$ map on $I\times [-\epsilon, 
\epsilon]$
so that $x\mapsto \Phi_t(x)$ is invertible, then we have
\begin{equation}\label{Phi}
\partial_t \Phi^{-1}_t(x)|_{t=s}
= -\frac{(\partial_t \Phi_t|_{t=s}) \circ \Phi_{s}^{-1}(x)}
{(\partial_x \Phi_{s}) \circ \Phi_{s}^{-1}(x)}\, ,
\end{equation}
and
\begin{align}\label{Phi''}
&\partial^2_{tt} \Phi^{-1}_t(x)|_{t=s}
= \partial_t \frac{(\partial_t \Phi_t) \circ \Phi_{t}^{-1}(x)}
{(\partial_x \Phi_{t}) \circ \Phi_{t}^{-1}(x)}|_{t=s}\\
\nonumber
&\quad= \frac{1}
{(\partial_x \Phi_{s}) \circ \Phi_{s}^{-1}(x)}
\cdot \bigl (
\partial^2_{tt} \Phi_t|_{t=s} \circ \Phi_{s}^{-1}(x)
+ \partial^2_{xt} \Phi_t|_{t=s} \circ \Phi_{s}^{-1}(x)
\cdot \partial_t \Phi^{-1}_t(x)|_{t=s} \bigr )\\
\nonumber
&\quad-
\frac{(\partial_t \Phi_t|_{t=s}) \circ \Phi_{s}^{-1}(x)}
{(\partial_x \Phi_{s}) \circ \Phi_{s}^{-1}(x)}
%\cdot
\frac{(\partial^2_{tx}  \Phi_{s}|_{t=s}) \circ \Phi_{s}^{-1}(x)
+ (\partial^2_{xx} \Phi_{s}) \circ \Phi_{s}^{-1}(x)
\cdot  \partial_t \Phi^{-1}_t(x)|_{t=s}}
{(\partial_x \Phi_{s}) \circ \Phi_{s}^{-1}(x)} 
 \,
 .
\end{align}
Since 
 $t \mapsto f_t\in C^2(I)$ is $C^2$,  we may
 apply the above to $\Phi_t(x)=f^k_t(x)$
 restricted to a suitable domain. The right-hand-side
 of \eqref{Phi} is just 
 $Y_{k,s}(f^{-k}_{s,+}(x))/(f^k_s)'(f^{-k}_{s,+}(x))$.
 Then, a Taylor series
 of order $2$ gives 
\begin{align}
f^{-k}_{+}(x)-f^{-k}_{t,+}(x)
&=t\frac{ Y_k ( f^{-k}_+(x))}{(f^k)'(f^{-k}_+(x))}+t^2 F_k(x,s)\, ,
\end{align}
where $x$ is as in
\eqref{claim3} and  $s \in [0,t]$. 
In order to estimate $F_k(x,s)$, we look at the various terms
in \eqref{Phi''}.  The (identical) factors
$\partial_t \Phi_t^{-1} $
and 
$(\partial_t \Phi_t /\partial_x\Phi_t )\circ \Phi_t^{-1}$ can
be bounded by \eqref{claim0}.
Since
$\partial^2_{xt} \Phi_t|_{t=0}=\partial_y Y_k$,
the two terms containing this expression can  
be controlled, when divided
by $(\partial_x \Phi_t) \circ \Phi_t^{-1}$, respectively
by $e^{-2\gamma k} |(f^k)'(c_1)|^{1/2}$,
by using the ideas to bound \eqref{secondline}.
Analysing the term containing $\partial^2_{xx} \Phi_t$ is of the same
type as (but simpler than) what we did for $\partial^2_{yy} Y_k$,
and the available factor $|(f^k)'(c_1)|^{-1}$ gives the right control.
The only new expression is 
$$\partial^2_{tt} \Phi_t|_{t=s}(x)=
Z_{k,s}(x):= \partial_t Y_{k,t}(x)|_{t=s}=\lim_{t \to s} \frac{Y_{k,t}(x) - Y_{k,s}(x)}{t-s}\, .
$$
This 
involves functions such as  $f'$, $f''$, $f''$,
$X_s$, and $X'_s$, but also $\partial_t X_t$.
The dominant term contains a factor $|(f^{k})'(y)$,  which can be
controlled by $(\partial_x \Phi_s) \circ \Phi_s^{-1}$ in the
denominator.

Finally, using   
\begin{equation}
\label{derr}\frac{1}{(f^k_t)'(f^{-k}_{+}(x))}- 
\frac{1}{(f^k)'(f^{-k}_{t,+}(x))}=(f^{-k}_{t,+}(x)-f^{-k}_+(x))'\, ,
\end{equation}
we find
\begin{align}
 \frac{1}{(f^k_t)'(f^{-k}_{+}(x))}- 
\frac{1}{(f^k)'(f^{-k}_{t,+}(x))}
&=t\biggl ( \frac{ Y_k ( f^{-k}_+(x))}{(f^k)'(f^{-k}_+(x))}\biggr )'
+t^2  G_k(x,s)
\, ,
\end{align}
for $x$ as in \eqref{claim4} and $s\in [0,t]$.
The new derivatives appearing in $G_k$ are
$\partial^3_{xxx}\Phi$, $\partial^3_{txx}\Phi$ and
$\partial^3_{ttx}\Phi$ (but not $\partial ^3_{ttt}\Phi$,
which is a priori undefined).
The claimed estimates on   $\sup|G_k|$ 
  can be obtained by horizontality,
similarly
to those for $F_k$, using now the $x$-derivative of \eqref{Phi''}
and exploiting in addition to the previous remarks
the bound \eqref{claim22}. 
The cancellation pattern described above emerges again.
The computations are straightforward, although 
cumbersome to write, and left to the reader.

\end{appendix}

 \nocite{avila}
\nocite{sandst}

\bibliographystyle{plain}

%\end{thebibliography}

\end{document}